\documentclass[review,onefignum,onetabnum]{siamonline190516}
\usepackage{amsmath, amssymb, mathrsfs}
\usepackage{lipsum}
\usepackage{amsfonts}
\usepackage{graphicx}

\usepackage{enumitem}
\setlist[enumerate]{leftmargin=.5in}
\setlist[itemize]{leftmargin=.5in}

\newtheorem{thm}{Theorem}[section]
\newtheorem{prop}[thm]{Proposition}
\newtheorem{lem}[thm]{Lemma}

\theoremstyle{remark}
\newtheorem{rem}[thm]{Remark}
\newtheorem{defin}[thm]{Definition}

\newcommand{\ie}{{\it i.e.}}
\newcommand{\eg}{{\it e.g.}}
\newcommand{\etc}{{\it etc.}}

\graphicspath{{./figs/}}

\makeatletter
\DeclareFontFamily{U}{tipa}{}
\DeclareFontShape{U}{tipa}{m}{n}{<->tipa10}{}
\newcommand{\arc@char}{{\usefont{U}{tipa}{m}{n}\symbol{62}}}%

\newcommand{\arc}[1]{\mathpalette\arc@arc{#1}}

\newcommand{\arc@arc}[2]{%
  \sbox0{$\m@th#1#2$}%
  \vbox{
    \hbox{\resizebox{\wd0}{\height}{\arc@char}}
    \nointerlineskip
    \box0
  }%
}
\makeatother

\usepackage{pgf}

\newcommand*{\SectorRadius}{1ex}
\newcommand*{\SectorHalfAngle}{45}
\newcommand*{\SectorLineWidth}{.4pt}

\newcommand*{\sector}{%
  \begin{pgfpicture}
    \pgfpathmoveto{\pgforigin}%
    \pgfpathlineto{\pgfpointpolar{90-(\SectorHalfAngle)}{\SectorRadius}}%
    \pgfarc{90-(\SectorHalfAngle)}{90+\SectorHalfAngle}{\SectorRadius}%
    \pgfpathclose
    \pgfsetlinewidth{\SectorLineWidth}%
    \pgfusepath{stroke}%
  \end{pgfpicture}%
}

\def\R{{\mathbb R}}
\def\N{{\mathbb N}}
\def\E{{\mathbb E}}

\def\D{{\mathbb D}}
\def\e{{\varepsilon}}
\def\co{\textrm{co}}
\def\vol{\textrm{vol}}
\def\supp{\textrm{supp}}
\newcommand{\calA}{\mathcal{A}}

\headers{Consistency of archetypal analysis}{B. Osting, D. Wang, Y. Xu, and D. Zosso}

\title{Consistency of archetypal analysis\thanks{Braxton Osting acknowledges partial support from NSF DMS 16-19755 and 17-52202. Dong Wang acknowledges support by the University Development Fund from The Chinese University of Hong Kong, Shenzhen (UDF01001803). Dominique Zosso acknowledges support by a Simons collaboration grant for mathematicians (\#586942).}}

\author{Braxton Osting\thanks{Department of Mathematics, University of Utah, Salt Lake City, UT
  (\email{osting@math.utah.edu}).}
\and Dong Wang\thanks{School of Science and Engineering, The Chinese University of Hong Kong, Shenzhen, Guangdong 518172, China
  (\email{wangdong@cuhk.edu.cn}).}
\and  Yiming Xu\thanks{Department of Mathematics, University of Utah, Salt Lake City, UT
  (\email{yxu@math.utah.edu}.)}
  \and  Dominique Zosso\thanks{Department of Mathematical Sciences, Montana State University, Bozeman, MT
  (\email{dominique.zosso@montana.edu}).}} 

\usepackage{amsopn}

\ifpdf
\hypersetup{
  pdftitle={Consistency of archetypal analysis},
  pdfauthor={B. Osting, D. Wang, Y. Xu, and D. Zosso}
}
\fi

\nolinenumbers

\begin{document}

\maketitle

\begin{abstract}
Archetypal analysis is an unsupervised learning method that uses a convex polytope to summarize multivariate data.
For fixed $k$, the method finds a convex polytope with $k$ vertices, called \emph{archetype points}, such that the polytope is contained in the convex hull of the data and the mean squared distance between the data and the polytope is minimal.
In this paper, we prove a consistency result that shows if the data is independently sampled from a probability measure with bounded support, then the archetype points converge to a solution of the continuum version of the problem, of which we identify and establish several properties.
We also obtain the convergence rate of the optimal objective values under appropriate assumptions on the distribution.
If the data is independently sampled from a distribution with unbounded support, we also prove a consistency result for a modified method that penalizes the dispersion of the archetype points.
Our analysis is supported by detailed computational experiments of the archetype points for data sampled from the uniform distribution in a disk, the normal distribution, an annular distribution, and a Gaussian mixture model.
\end{abstract}



\section{Introduction}
Fix $k \in \N$ and write $[k] = \{1,2,\ldots, k\}$. 
For given data  
$X_N = \{ x_i\}_{i\in [N]} \\
 \subset \R^d$, 
the \emph{archetypal analysis} problem 
is to find a cardinality $k$ pointset $A = \{a_\ell\}_{\ell \in [k]} \subset \R^d$ that solves
\begin{subequations}
\label{e:arch}
\begin{align}
\min_{A \subset \R^d} \ & F(A) \\
& A \subset \co(X_N), 
\end{align}
\end{subequations}
where\[F(A) = \left(\frac{1}{N} \sum_{i = 1}^N d^2(x_i, \co(A) )\right)^{1/2} = \left( \int_{\R^d}d^2(x, \co(A)) \ d\mu_N(x)\right)^{1/2}. \]
Here, $\co(\cdot)$ denotes the convex hull, 
$d^2(\cdot,\cdot)$ denotes the squared Euclidean distance, and 
$\mu_N(x)= \frac{1}{N} \sum_{i \in [N]} \delta_{x_i}(x)$ is the empirical measure associated with the data $X_N$. 
The constraint in \eqref{e:arch} imposes that the convex hull of the dataset should contain the 
pointset $A$, and hence $\co(A)$. 
The objective function, $F$, in \eqref{e:arch} is the root mean squared distance from the data to the convex hull of $A$. 
By pointset, we mean the unordered collection of $k$ points in $\mathbb R^d$, which we denote by $\{ \mathbb R^d \}^k$. 
One may think of $\{\R^d\}^k$ as the product space $(\R^d)^k$ modulo permutation of its components, \ie, 
\[\{a_\ell\}_{\ell\in [k]}=\left\{(b_\ell)_{\ell\in [k]}\in(\R^d)^k: a_\ell = b_{\sigma(\ell)} \ \text{for some}\ \sigma\in S[k]\right\},
\] 
where $S[k]$ denotes the set of permutations on $[k]$. 
We refer to a minimizing pointset  $A = \{a_\ell\}_{\ell \in [k]} \in \{ \mathbb R^d \}^k$ solving \eqref{e:arch} as an \emph{archetype pointset}, the points $a_i \in A$ as \emph{archetype points}, and the convex hull, $\co(A)$ as the \emph{archetype polytope}.

Archetypal analysis was introduced  in 1994 by A.~Cutler and L.~Breiman as an unsupervised method to summarize multivariate data \cite{Cutler_1994}. 
They proved the following results: 
(i) If $k=1$, then the archetype point is the mean of the data, $X_N$. 
(ii) For $1<k<N$, there exists an archetype pointset, $A = \{a_\ell\}_{\ell \in [k]}$, \ie, there exists a solution to \eqref{e:arch} and furthermore, there exists an archetype pointset on the boundary of $\co(X_N)$. 
(iii) Finally for $k \geq N$, the archetype pointset is given by $A = X_N$, which attains the value $F(A) = 0$. 
They demonstrated that archetypal analysis can be reformulated as a nonlinear least squares problem and efficiently solved using an alternating minimization algorithm.
In their concluding remarks, they note that ``Because the archetypes are located on the boundary of the convex hull of the data, the procedure can be sensitive to outliers. Robust versions could be developed using convex hull peeling or the outlyingness idea of Donoho and Gasko (1992)''. 
Subsequently, there has been a lot of work to extend archetypal analysis to improve robustness and sensitivity to outliers \cite{Chen_2014,wu2017prototypal,mair2019coresets} and  to develop improved computational methods for computing archetypes \cite{Bauckhage_2009}. 
M.~M{\o}rup and L.~K.~Hansen  carefully compared archetypal analysis to other matrix factorization and clustering methods such as SVD/PCA, NMF, soft $k$-means, $k$-means, and $k$-medoids, and showed how archetypal analysis can be effectively used as an unsupervised machine learning tool for a variety of data analysis problems \cite{M_rup_2012}. 
Interestingly, nonlinear versions of archetypal analysis based on neural networks have recently been proposed \cite{Dijk2019,Keller_2019,keller2020learning}.
Finally, archetypal analysis has been applied to a variety of real-world applications, including 
finding archetype soccer players based on performance data \cite{Seth_2015}, 
query-focused multi-document summarization \cite{Canhasi_2014}, 
classification of galaxy spectra \cite{Chan_2003}, 
spatio-temporal dynamics \cite{Stone_1996}, \etc. 
We comment that archetypal analysis is also sometimes referred to as \emph{principal convex hull analysis} as it approximates the ``best'' convex hull, although we do not use this language here.

In this work, we consider the following consistency problem.   
Suppose that $x_1, x_2, \ldots$ are independently sampled from the probability measure $\mu$ and denote  
the first $N$ points by $X_N = \{ x_i\}_{i \in [N]}$. 
Let $A_N$ denote the optimal solution to \eqref{e:arch} for each $N$. 
\emph{Is there a set $A$ (depending on $\mu$), such that $A_N \to A$ as $N\to \infty$ in some sense? } 

Consistency results are fundamental in statistics and have important consequences in applications. 
Namely,  an estimate (\eg, archetype polytope) obtained using a consistent method will asymptotically stabilize, and so the collection of more data will yield diminishing returns. 
Generally, a consistency result for a problem posed on sampled data requires two ingredients: 
(i) a model from which the samples are drawn from; the original problem is viewed as a finite sample size problem for this model and
(ii) a notion of convergence for the estimated quantity. 
We give a consistency result for the archetypal analysis problem where the data is sampled from a measure with compact support and 
a consistency result for a modification of the archetypal analysis problem where the data is sampled from a measure with non-compact support. In both cases, we will use a Euclidean notion of distance; see \cref{dist:points}. 

\subsection{Consistency for measures with compact support} In the following, let $\nu$ be a general probability measure on $\R^d$ and $\mu$ be the sampling measure for data.
For convenience, we assume that $\mu$ has a bounded density. Note that \eqref{e:arch} is a special case of the following minimization problem when we take $\nu=\mu_N$:  
\begin{subequations}
\label{e:arch2g}
\begin{align}
\min_{ A \in \{ \mathbb R^d \}^k } \ &F_\nu(A)\\
 \textrm{s.t.}  \ &  A \subset \co(\supp(\nu)).
\end{align} 
\end{subequations} 
where $F_{\nu}(A) = \left(\int_{\R^d} d^2\left( x, \co (A) \right) \  d \nu(x)\right)^{1/2}$. Since $\mu_N\rightharpoonup\mu$ as $N\to\infty$, a natural candidate limiting problem for \eqref{e:arch} is thus given by
\begin{subequations}
\label{e:arch2}
\begin{align}
\min_{ A \in \{ \mathbb R^d \}^k } \ & F_\mu(A)\\
 \textrm{s.t.}  \ &  A \subset \co(\supp(\mu)).
\end{align} 
\end{subequations}  

In \cref{s:pf-bd}, we prove the following elementary existence result. 

\begin{thm}\label{t:existence}
Suppose that $\nu$ is a probability measure on $\R^d$ with compact support. Then the minimization problem \eqref{e:arch2g} admits at least one minimizer  with optimal value $F_\nu^\star \leq (\int_{\mathbb R^d} \|x - \bar x\|_2^2 \ d \nu(x))^{1/2}$,
 where $\bar x  = \int_{\mathbb R^d} x \ d\nu(x)$ is the mean. 
\end{thm}

Taking $\nu =\mu_N$ and $\mu$ in \cref{t:existence} deduces the existence of minimizers for both problems \eqref{e:arch} and \eqref{e:arch2}. However, it should be noted that solutions to \eqref{e:arch2} are not unique in general; non-unique  examples can easily be constructed when the support of $\mu$ has symmetries. 

As proven in  \cite{Cutler_1994} for the discrete problem,  the following theorem states that if $\nu$ has bounded support, then there exists a minimizer for problem \eqref{e:arch2g} on the boundary of the convex hull of the support of $\nu$.
 \begin{thm}\label{thm:bd}
Let $k\geq 2$.
Suppose that $\nu$ is supported on a bounded set in $\R^d$. 
Then there exists a minimizing pointset $\{ a_\ell\}_{\ell\in [k]}$ of problem \eqref{e:arch2g} such that $\{ a_\ell\}_{\ell\in [k]}\subset\partial(\co(\supp(\nu)))$. 
\end{thm}

As the proof of \cref{thm:bd} is similar to that in  \cite{Cutler_1994}, we omit the proof. 
 
\begin{rem}
We remark that \cref{thm:bd} does not imply that every minimizing pointset of problem \eqref{e:arch2g} is on the boundary of $\co(\supp(\nu))$. For example, consider the case when $k=3$ and $\nu$ is the empirical measure of the magenta data points $X_N$, as illustrated in \cref{exam:bd}. It is clear from our construction that all red and blue triangles are archetype triangles (optimal solutions). The topmost vertex of the blue triangle is in the interior of the $\co(X_N)$.  
Note that due to the two clusters of magenta data points, the gray triangle is not optimal. 
\end{rem}

\begin{figure}[t!]
\centering
\includegraphics[width = 0.9\textwidth]{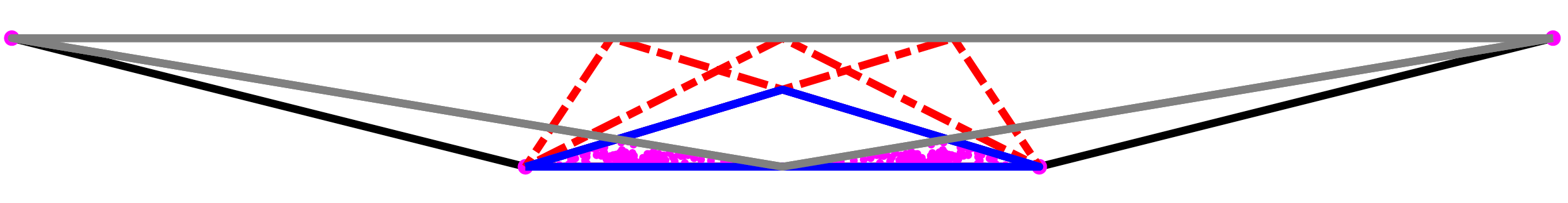}
\caption{A illustration of an example where the archetype pointset is attained on the interior of $\co(X_N)$.} 
\label{exam:bd}
\end{figure}

We first consider the consistency problem under the assumption that $\mu$ has compact support. To describe the convergence of archetype points, we will use the $d_{2, \infty}$ metric on $\{\R^d\}^k$ which is defined as follows: 

\begin{defin}\label{dist:points}
For fixed integers $d$ and $k$, the metric $d_{2, \infty}(\cdot, \cdot )$ on $\{\R^d\}^k$ is defined by
\begin{align*}
d_{2, \infty}(\{ a_\ell\}_{\ell \in [k]}, \{ b_\ell\}_{\ell \in [k]}):=\min_{\sigma\in S[k]}\max_{\ell\in [k]}\| a_{\sigma(\ell)}-b_\ell \|_2,
\qquad  \qquad  \{a_\ell\}_{\ell\in [k]}, \{b_\ell\}_{\ell\in [k]}\in\{\R^d\}^k,
\end{align*} 
where $S[k]$ is the set of permutations on $[k]$. 
\end{defin} 

Properties of $d_{2, \infty}$ will be given in \cref{s:pf-bd}.  The following Theorem, which is proven in \cref{s:pf-bd}, gives a consistency result under the assumption that $\mu$ has compact support. 
\begin{thm}\label{t:bdd consistency}
Fix $k \in \N$. 
Let $\mu$ be a probability measure on $\R^d$ with compact support. 
Let $x_1, x_2, \ldots $ be iid samples from $\mu$ and denote by $X_N = \{ x_i\}_{i \in [N]}$ the first $N$ samples. 
Let $A_N$ be an archetype pointset solving \eqref{e:arch} for each $N$. 
Then, $A_N$ has a convergent subsequence $A_{N_m}$ (in the $d_{2, \infty}$ sense) whose limit lies in the solution set of \eqref{e:arch2} $\mu$-a.s..
If \eqref{e:arch2} has a unique solution $A_\star$, then $A_N$ converge to $A_\star$ in $d_{2, \infty}$ $\mu$-a.s..
\end{thm}

Although the consistency result stated in \cref{t:bdd consistency} only holds for a subsequence in general, the optimal objective values of the discrete problems are, in fact, convergent. 
Moreover, under appropriate assumptions on $\mu$, we can further establish the convergence rate using a result from random geometry \cite{D_mbgen_1996, Brunel_2019} and Dudley's inequality \cite{Vershynin_2018}. 
Before stating the convergence rate in \cref{t:rate}, we recall the $\alpha$-cap condition on a probability measure which is introduced in \cite[Assumption 1]{Brunel_2019}.

\begin{defin}[$\alpha$-cap condition]
Let $\alpha>0$. A probability measure $\nu$ defined on $\R^d$ is said to satisfy the $\alpha$-cap assumption with parameters $L$ and $\eta$ if it has bounded support and
\begin{align}
\label{cap}
\min_{u\in\mathbb S^{d-1}}\nu\left(\left\{y\in\co(\text{supp}(\nu)) \colon   \langle y, u\rangle\geq \max_{z\in\co(\text{supp}(\nu))}\langle z, u\rangle-\e\right\}\right)\geq L\e^{\alpha}, 
\qquad \forall  \e \in (0,\eta]. 
\end{align}
\end{defin}

\begin{rem}
The $\alpha$-cap condition can be interpreted as that for any $x\in\partial(\co( \supp (\nu)))$, the slab between the supporting hyperplane at $x$ and its $\epsilon$-perturbation captures $\mathcal O(\e^\alpha)$ mass under $\nu$. 
One can check that $\mu$ satisfies the $\alpha$-cap assumption with $\alpha=d$   when $\mu$  is supported on a convex domain with density bounded from below. Moreover, $\mu$ satisfies the $\alpha$-cap assumption with $\alpha\geq d$ whenever $\mu$ has a density, since Lebesgue measure of the slab defined in \eqref{cap} scales as $\mathcal O(\e^d)$ for $\e\to 0$.   
\end{rem}

\begin{thm}\label{t:rate}
Under the same conditions as \cref{t:bdd consistency}, $F_\mu(A_N)\to F_\mu(A_\star)$ as $N\to\infty$. Moreover, if $\mu$ satisfies the $\alpha$-cap condition with parameters $L$ and $\eta$, then with probability at least $1-4N^{-2}$, 
\begin{align}
&F_\mu(A_N)-F_{\mu}(A_\star)\lesssim \left(\frac{\log N}{N}\right)^{1/\alpha}&N>N(\eta),\label{rt}
\end{align}
where $N(\eta)$ is a non-increasing function in $\eta$ and the implicit constant in \eqref{rt} depends only on $\mu$. 
In particular, if $\mu$ is supported on a convex set with positive density, $F_\mu(A_N)-F_\mu(A_\star)=\mathcal O((\log N/N)^{1/d})$ $\mu$-a.s.. 
\end{thm} 

The proof of \cref{t:rate} is given in \cref{s:pf-bd}.  The asymptotic rate in \eqref{rt} mainly depends on the $\alpha$ in the $\alpha$-cap definition, which, roughly speaking, depends on the geometry of the support and the density near the boundary of the support of $\mu$. In \cref{m-variance}, we discuss the  rate of convergence for densities with very large variance. 

In \cref{s:ExUnitDisk}, we demonstrate the consistency of the archetype analysis problem for the example when $d=2$ and $\mu$ is a uniform measure on the Euclidean unit disk. In this case, we prove that archetype points are the vertices of regular $k$-polygons inscribed in the unit disk. Furthermore, in \cref{s:NumEx}, we perform numerical experiments for the discrete problem \eqref{e:arch} with $k=3$ to demonstrate the convergence, as $N\rightarrow \infty$, of the archetype triangle to a regular triangle inscribed in a unit disk.

\subsection{Consistency for measures with non-compact support} 
We now consider the consistency problem when the probability measure, $\mu$, has non-compact support. Here we have that 
$\co(X_1) \subseteq \co(X_{2}) \subseteq \cdots$ and $\co(X_N)$ is a.s. unbounded as $N\to \infty$.  
In this case, it is clear that there can be no limiting problem for the archetype pointset $A_N$ as $N \to \infty$; the problem, as stated, is inconsistent. 
Consequently, we must modify \eqref{e:arch} or \eqref{e:arch2} to obtain a consistency result. 

When the probability measure $\nu$ has non-compact support, we add a `penalty term' to the energy that prevents the archetype points from tending to infinity and consider the objective function 
\begin{align}\label{pro:unbounded}
F_{\nu,\alpha}(A) &= \left( F_\nu^2(A) + \alpha V(A) \right)^{1/2}= \left(\int_{\R^d} d^2\left( x, \co (A) \right) \  d \nu(x) + \frac{\alpha}{k} \sum_{\ell \in [k] } \|a_\ell - \bar{a}\|_2^2\right)^{1/2}.
\end{align} 
 Here, $\alpha>0$ is a fixed parameter, 
$V(A) = \frac{1}{k} \sum_{\ell \in [k]} \left\|a_\ell - \bar a \right\|_2^2$  
is the variance of the archetype points, and $\bar a = \frac{1}{k} \sum_{\ell \in [k]} a_\ell$ is the mean of the archetype points. 
We think of $F_{\nu, \alpha}(A)$ as a trade-off between the  fidelity to the data and the size of the  archetype polytope.

\begin{rem}
In practice, datasets are always finite so that one can directly apply the archetypal analysis to them. Adding a regularization term can potentially make the algorithm more robust to unexpected outliers. However, the regularization parameter should be small, otherwise the solution will be close to the mean of the dataset (see \cref{t:ConsVarReg-alpha}), contradicting the original goal of summarizing datasets by their extreme patterns. 
\end{rem}

 \begin{rem}
Other regularization terms like the $k$-means type of penalties considered in \cite{wu2017prototypal} could be discussed similarly under our methodology.  Another penalty term one might consider would be $V(A) = \vol(\co (A))$. However, in this case, it is possible for $\co(A)$ to be a degenerate (lower dimensional) set  to avoid the penalization.     
\end{rem}

We consider the variance-regularized archetype problem, 
\begin{subequations}
\label{e:arch3}
\begin{align}
\min_{ A \in \{ \mathbb R^d \}^k } \ & F_{\nu,\alpha}(A)\\
 \textrm{s.t.}  \ &  A \subset \co(\supp(\nu)). 
\end{align} 
\end{subequations} 
For fixed $N$ and square-integrable $\nu$, the following result establishes the existence of minimizers of \eqref{e:arch3}.

\begin{thm}\label{existence:var-reg}
Let $\alpha>0$ be a fixed number. Suppose that $\nu$ is square-integrable, \ie, $\int_{\R^d}\|x\|_2^2\ d\nu(x)<\infty$. Then the minimization problem \eqref{e:arch3} admits at least one minimizer with optimal value 
$F_{\nu, \alpha}^\star \leq (\int_{\mathbb R^d} \|x - \bar x\|_2^2 \ d \nu(x))^{1/2}$,
where $\bar x  = \int_{\mathbb R^d} x \ d\nu(x)$ is the mean. 
 \end{thm}

Taking $\nu =\mu_N$ and $\mu$ in \cref{existence:var-reg} deduces the existence of minimizers for both the discrete and continuous archetype problems with the variance regularization. 

We next state the consistency result for the variance-regularized archetype problem. Since the support of the distribution may be unbounded, a technical condition is needed on the distribution and discussed further below. 

\begin{thm} \label{t:ConsVarReg}
Fix $k \in \N$. 
Let $\mu$ be a probability measure on $\R^d$ which is square-integrable. 
Let $x_1, x_2, \ldots $ be iid samples from $\mu$ and denote by $X_N = \{ x_i\}_{i \in [N]}$ the first $N$ samples. 
Let $A_N$ be the archetype pointset solving \eqref{e:arch3} with $\nu = \mu_N$ each $N$. 
Suppose that $\mu$-a.s., for any $r>0$, there exists some $N(r) \in\N$ such that $B(r)\subset \co(X_{N(r)})$. 
Then, $A_N$ has a convergent subsequence $A_{N_m}$ (in the $d_{2, \infty}$ sense) whose limit lies in the solution set of \eqref{e:arch2} $\mu$-a.s..
If \eqref{e:arch3} has a unique solution $A_\star$, then $A_N$ converge to $A_\star$ in $d_{2, \infty}$ $\mu$-a.s..
\end{thm}

\begin{rem}
The technical assumption on $\mu$ states that $\mu$-a.s., for any $r>0$, there exists some $N(r) \in\N$ such that $B(r)\subset \co(X_{N(r)})$. The same proof applies if all the minimizers of the continuous problem are in the interior of $\co(\supp(\mu))$. In particular, the theorem applies to any $\mu$ defined on $\R^d$ with strictly positive density. 
\end{rem}

Note that as $\alpha \to \infty$, problem \eqref{e:arch3} reduces to the following problem 
\begin{subequations}
\label{e:arch4}
\begin{align}
\min_{ a \in \R^d } \ & \left(\int_{\R^d} d^2(x, a)\ d\nu(x)\right)^{1/2}\\
 \textrm{s.t.}  \ &  a \in \co(\supp(\nu)), 
\end{align} 
\end{subequations} 
which has a unique minimizer $\bar{x} = \int_{\R^d} x\ d\nu(x)$. Therefore, one should expect that solutions of \eqref{e:arch3}, $A_{\star,\alpha}$, shrink to $A_{\star, \infty} = \{\bar{x}\}^k$ as $\alpha$ tends to infinity. The following theorem quantifies this observation by giving an upper bound on the convergence rate. 

\begin{thm} \label{t:ConsVarReg-alpha}
Fix $k \in \N$. 
Let $\nu$ be a probability measure on $\R^d$ which is square-integrable. 
Let $A_{\star, \alpha}$ be a solution to the problem \eqref{e:arch3}. Let $\bar{x}=\int_{\R^d} x\ d\nu(x)$ and define $A_{\star, \infty} = \{\bar{x}\}^k$. 
For sufficiently large $\alpha$,  
\begin{align*}
d_{2, \infty}(A_{\star, \alpha}, A_{\star, \infty})\leq 8k^{1/2}\alpha^{-1/4}\left(\int_{\R^d}\|x-\bar{x}\|_2^2\ d\nu(x)\right)^{1/2}. 
\end{align*}
\end{thm}

In \cref{s:NumEx}, we perform several numerical experiments for data sampled from the normal distribution, an annular distribution, and a Gaussian mixture distribution to study the solution to problem \eqref{e:arch3}. Furthermore, in the experiments, we study how the area of $\co(A)$ depends on $\alpha$ and observe that $\co(A)$ always  contains $\bar{x}$ for  all initializations and samples of the data.

\subsection{A brief review of consistency results for other unsupervised learning methods} \label{s:PrevWork}
In \cref{t:bdd consistency} and \cref{t:ConsVarReg}, we state consistency results for archetypal analysis. Here, we briefly review related work proving consistency for other unsupervised methods. 

A related and well-known unsupervised method for data summary is  $k$-means clustering, which finds the $k$ points in the space that minimize the within class variance of the dataset. 
In the seminal work, Pollard established a strong consistency property for the theoretic $k$-means solutions by assuming the data points are independently sampled from a fixed distribution \cite{Pollard_1981}. 
An asymptotic normality result on the convergence rate was later proven in \cite{pollard1982central} using a functional central limit theorem. 
The rich theory behind the $k$-means clustering boils down to the convenient form of the objective function as well as the fact that the admissible set is data-independent, and allows generalization to other related problems. For instance, Sun {\it et al.} extended the consistency result in \cite{Pollard_1981} to the regularized $k$-means clustering, where a group LASSO/adaptive LASSO term is considered to balance the trade-off between model fitting and sparsity \cite{Sun_2012}.  
Consistency results concerning other unsupervised clustering methods are also worth noting. For example, Hartigan  provided a weak notion of consistency for Linkage algorithms,  proving that the algorithm can obtain a couple of high-density regions \cite{Hartigan_1981}. 
Finally, there are many consistency and convergence results for the graph Laplaican on geometric graphs
\cite{Bousquet,LafonThesis,BelkinUniform,Hein,Singer,HeinMore,calder2019improved}, 
as well as graph-based methods for data analysis based on the graph Laplacian, including 
spectral clustering  \cite{von_Luxburg_2008,trillos2018variational},
Cheeger and ratio graph cuts \cite{trillos2016consistency}, 
Dirichlet partitions \cite{Reeb2016}, 
and the PageRank algorithm \cite{Yuan2020}.

\subsection{Outline} 
In \cref{s:pf-bd}, we discuss the consistency problem for distributions with compact support; in particular, we prove \cref{t:existence}, \cref{t:bdd consistency}, and \cref{t:rate}. 
In \cref{s:ExUnitDisk}, we consider finding the archetype points for the uniform distribution on the unit disk. 
In \cref{s:UnboundedSupport}, we consider distributions with unbounded support and prove \cref{existence:var-reg}, \cref{t:ConsVarReg} and \cref{t:ConsVarReg-alpha}. 
In \cref{s:NumEx}, we present a variety of numerical examples that strongly support the analysis. 
We conclude in \cref{s:Disc} with a discussion.

\section{Proof of Theorems~\ref{t:existence}, \ref{t:bdd consistency}, and \ref{t:rate} for probability measures with compact support}  
\label{s:pf-bd}
Before proving the main results on the consistency of archetype problems for probability measures with bounded support, we discuss a few properties of the $d_{2, \infty}$ distance, defined in \cref{dist:points}, which will be used throughout the convergence analysis. To see that the $d_{2, \infty}$ distance is well-defined, it suffices to check the triangle inequality holds, as the non-negativity, identity and symmetry are obvious. For any $\{a_\ell\}_{\ell\in [k]}, \{b_\ell\}_{\ell\in [k]}, \{c_\ell\}_{\ell\in [k]}\in\{\R^d\}^k$, there exist some $\sigma_1, \sigma_2\in S[k]$ such that 
\begin{align*}
d_{2, \infty}(\{a_\ell\}_{\ell\in [k]}, \{c_\ell\}_{\ell\in [k]}) &= \max_{\ell\in [k]}\|a_{\sigma_1(\ell)}-c_\ell \|_2\\
d_{2, \infty}(\{b_\ell\}_{\ell\in [k]}, \{c_\ell\}_{\ell\in [k]}) &= \max_{\ell\in [k]}\|b_{\sigma_2(\ell)}-c_\ell \|_2.
\end{align*}
By definition,
\begin{align*}
d_{2,\infty}(\{a_\ell\}_{\ell\in [k]}, \{b_\ell\}_{\ell\in [k]})&\leq \max_{\ell\in [k]}\|a_{\sigma_1(\ell)}-b_{\sigma_2(\ell)} \|_2\\
&\leq \max_{\ell\in [k]}\|a_{\sigma_1(\ell)}-c_\ell \|_2 + \max_{\ell\in [k]}\|b_{\sigma_2(\ell)}-c_\ell \|_2\\
& = d_{2, \infty}(\{a_\ell\}_{\ell\in [k]}, \{c_\ell\}_{\ell\in [k]}) + d_{2, \infty}(\{b_\ell\}_{\ell\in [k]}, \{c_\ell\}_{\ell\in [k]}).
\end{align*}
It is clear that any bounded closed subset of $\{\R^d\}^k$ is compact in the topology induced by $d_{2, \infty}$. 

Another commonly used metric to measure the distance between two closed sets in a metric space $(\mathcal{M}, d)$ is the Hausdorff distance. For any two closed sets $X, Y\subset\mathcal{M}$, the Hausdorff distance between them is given by 
\begin{align}
d_H(X, Y):= \max \left\{\,\sup _{x\in X}\inf _{y\in Y}d(x,y),\,\sup _{y\in Y}\inf _{x\in X}d(x,y)\,\right\}.\label{haus}
\end{align}
It is easy to check that $d_{2, \infty}$ is stronger than the Hausdorff distance for sets consisting of $k$ points in $\R^d$. In particular, for $\{a_\ell\}_{\ell\in [k]}, \{b_\ell\}_{\ell\in [k]}\in\{\R^d\}^k$, 
\begin{align}
d_H(\{a_\ell\}_{\ell\in [k]}, \{b_\ell\}_{\ell\in [k]})\leq d_{2,\infty}(\{a_\ell\}_{\ell\in [k]}, \{b_\ell\}_{\ell\in [k]}).\label{H-2-infty}
 \end{align}
In general, $d_{2, \infty}$ and $d_{H}$ are not exactly the same. For example, consider $A = \{(0,0), (1, 0), (2,\\0) \}$ and $B=\{(0.5, 0), (2, 0.1), (2, -0.1)\}$ in $\{\R^2\}^3$. It is easy to check that $d_H(A, B)=0.5<\sqrt{1.01}=d_{2,\infty}(A, B)$.

We will first establish the existence of solutions to \eqref{e:arch2g} as stated in \cref{t:existence}. We use the following lemma about the objective function, $F_\nu\colon \{\R^d\}^k \to \R$, defined in \eqref{e:arch2g}. 
\begin{lem}\label{lm:con}
For any probability measure $\nu$ on $\R^d$ which is square-integrable, the objective function defined by
\begin{align*}
F_\nu( \{ a_\ell\}_{\ell \in [k]} )=\left(\int_{\R^d}d^2\left( x, \co\left( \{a_\ell \}_{\ell \in [k]} \right)\right)  d \nu(x)\right)^{1/2}
\end{align*}
is a $1$-Lipschitz continuous function from the metric space $(\{\R^d\}^k, d_{2, \infty})$ to $\R$, where $d_{2, \infty}$ is defined in \cref{dist:points}. 
\end{lem}

\begin{proof}
Square-integrability guarantees that $F_\nu$ is finite for all $\{ a_\ell\}_{\ell \in [k]}\in\{\R^d\}^k$. 
The rest of the proof follows from direct computation. 
For any 
$ \{ a_\ell\}_{\ell \in [k]},   \{ b_\ell\}_{\ell \in [k]} \in\{\R^d\}^k$,  
we compute
\begin{align*}
|F_\nu(  \{ a_\ell\}_{\ell \in [k]} )-F_\nu(  \{ b_\ell\}_{\ell \in [k]} )|
&\leq \left(\int_{\R^d}|d(x, \co(\{a_j\}_{j\in [k]}))-d(x, \co(\{b_j\}_{j\in [k]}))|^2\ d\nu(x)\right)^{1/2} \\
&\leq  \left(\int_{\R^d}|d_H(\co(\{a_j\}_{j\in [k]}), \co(\{b_j\}_{j\in [k]}))|^2\ d\nu(x)\right)^{1/2} \\
 &=   d_H(\co(\{a_j\}_{j\in [k]}), \co(\{b_j\}_{j\in [k]})) \\
&\leq   d_H(\{a_j\}_{j\in [k]}, \{b_j\}_{j\in [k]}) \\
& \stackrel{\eqref{H-2-infty}}{\leq}  d_{2, \infty}(\{a_j\}_{j\in [k]}, \{b_j\}_{j\in [k]}).
\end{align*} 
\end{proof}

\begin{proof}[Proof of \cref{t:existence}]
Note that the admissible set $\{\text{supp}(\mu)\}^k\subset \{\R^d\}^k$ is closed in $d_{2, \infty}$ and has bounded diameter, therefore is compact. \cref{t:existence} is an immediate consequence of \cref{lm:con} together with the fact that a continuous function achieves its infimum on a compact set. This proves the existence of minimizers. The second part follows by noting that $\bar{x}\in\co(\supp{(\nu)})$.  
\end{proof}

We will next prove the consistency result stated in \cref{t:bdd consistency}. 
\begin{proof}[Proof of \cref{t:bdd consistency}]
Let $\calA$ be the solution set of \eqref{e:arch2}. It is clear that $\calA$ is non-empty by \cref{t:existence}. Our goal is to show that $\{A_N\}_{N\in\N}$ has a convergent subsequence whose limit lies in $\calA$. Since $\supp(\mu)$ is compact, we can apply the Bolzano--Weierstrass theorem to the $k$ elements of $\{A_N\}_{N\in\N}$ with a diagonal argument to find a convergent subsequence $\{A_{N_m}\}_{m\in\N}$ of $\{A_N\}_{N\in\N}$ such that $A_{N_m}\rightarrow A_\star$ in $d_{2, \infty}$ for some $A_\star \subset \co(\supp(\mu))$, \ie, $d_H(A_{N_m}, A_\star)\to 0$ as $m\to\infty$. 
We claim $A_\star\in\calA$ and this will be shown via a triangle-inequality type of argument.

Compared to \eqref{e:arch2}, the discrete problem \eqref{e:arch} has not only discrete objective function but also discrete constraints. To bridge the gap, we introduce an intermediate discrete minimization problem with the objective function in \eqref{e:arch} but the constraint in \eqref{e:arch2}:
\begin{align}\label{e:archetypal5}
\min_{ \{ a_\ell\}_{\ell \in [k]}} \ & F_{\mu_N}(A)\\
 \textrm{s.t.}  \ &  a_\ell\in \co(\supp(\mu)), \quad \ell \in{[k]}
\nonumber
\end{align} 
and denote a solution to \eqref{e:archetypal5} by $A_N'$ for each $N$. 
Similar to as in \cref{thm:bd}, we can take $A_N'$ to be located on the boundary of $\co(\supp(\mu))$.
Below, it is also convenient to define the set of points consisting of the convex projection of the points in $A_N'$ onto $\co(\supp(\mu_N))$,
\begin{align}
A_N'' = \text{Proj}_{\co(\supp(\mu_N))}(A_N').\label{proj}
\end{align}
Note that since $A_N' \subset  \overline{ \mathbb R^d \setminus \co(\supp(\mu_N))}$, we have that $A_N'' \subset \partial(\co(\supp(\mu_N)))$. 

The rest of the proof can be summarized as two steps: 
(i) first show that $d_H(A'_{N_m}, A''_{N_m})\rightarrow 0$ as $m\rightarrow\infty$ $\mu$-a.s. and 
(ii) then use it as a bridge to obtain that $A_\star$ is an optimal solution for \eqref{e:arch2}. 

To show (i), observe that $\co(\{x_j\}_{j\in [N]})\subseteq \co(\{x_j\}_{j\in [N+1]})\subseteq \co(\supp(\mu))$.
Since $\mu$ has compact support, $K:=\lim_{N\rightarrow\infty}\co(\{x_j\}_{j\in [N]})$ exists and is a convex set contained in $\co(\supp(\mu))$. It follows from a contradiction argument that $d_H(K,\co(\supp(\mu)))=0$ $\mu$-a.s.. If not, there would exist a point $x\in \co(\supp(\mu))$ such that $d(x, K)>0$. Since $x$ can be written as a convex combination of finitely many points in $\supp(\mu)$, say $x^{(1)}, \cdots, x^{(r)}$ for some $r>0$, at least one of them has positive distance to $K$. Without loss of generality we assume $d(x^{(1)}, K)>0$. This implies that the sequence $\{x_j\}$ does not intersect with a ball centered at $x^{(1)}$ with radius $\frac{1}{2}d(x^{(1)},K)$, \ie, $\{x_j\}\cap B(x^{(1)},\frac{1}{2}d(x^{(1)},K))=\varnothing$. But this happens on a $\mu$-null set: 
\begin{align*}
\mu\left(\{x_j\}\cap B(x^{(1)},\frac{1}{2}d(x^{(1)},K))=\varnothing\right)\leq\lim_{N\rightarrow\infty}(1-\mu(B(x^{(1)}, \frac{1}{2}d(x^{(1)},K))))^N=0,
\end{align*}
where $\mu(B(x^{(1)}, \frac{1}{2}d(x^{(1)},K)))>0$ because $x^{(1)}\in\supp(\mu)$. Therefore, $d_H(K, \co(\supp(\mu))) \\=0$. The desired result follows by observing that 
\begin{align}
\label{dist:conv}
d_H(A'_{N_m}, A''_{N_m})\leq d_H(\co\{x_j\}_{j\in [N]}, \co(\supp(\mu)))\rightarrow 0.
\end{align} 

For step (ii), we note that \eqref{dist:conv} together with $A_{N_m}\rightarrow A_\star$ and $\mu_{N_m}\rightharpoonup\mu$ $\mu$-a.s., implies that for any $A\in\calA$ and $\e>0$, there exists a sufficiently large $M$ such that for $m>M$,  
\begin{align}
d_H(A_{N_m}, A_\star)&\leq\e\label{11}\\
d_H(A'_{N_m}, A''_{N_m})&<\e\label{22}\\
\left(\int_{\R^d}d^2(x, A_\star)\ d\mu(x)\right)^{1/2}&\leq\left(\int_{\R^d}d^2(x, A_\star)\ d\mu_{N_m}(x)\right)^{1/2}+\e\label{33}\\
\left(\int_{\R^d}d^2(x, A)\ d\mu_{N_m}(x)\right)^{1/2}&\leq\left(\int_{\R^d}d^2(x, A)\ d\mu(x)\right)^{1/2}+\e\label{44}.  
\end{align}  
Therefore, for $m>M$, 
\begin{align*}
&\ \ \ \left(\int_{\R^d}d^2(x, A_\star)\ d\mu(x)\right)^{1/2}\\
&\stackrel{\eqref{33}}{\leq}\left(\int_{\R^d}d^2(x, A_\star)\ d\mu_{N_m}(x)\right)^{1/2}+\e\\
&\stackrel{\text{pf. Lem}\ \ref{lm:con}}{\leq}\left(\int_{\R^d}d^2(x, A_{N_m})\ d\mu_{N_m}(x)\right)^{1/2}+\left(\int_{\R^d}d^2_H(A_\star, A_{N_m})\ d\mu_{N_m}(x)\right)^{1/2}+\e\\
&\stackrel{\text{def. $A_{N_m}$}, \ \eqref{11}}{\leq}\left(\int_{\R^d}d^2(x, A''_{N_m})\ d\mu_{N_m}(x)\right)^{1/2}+2\e\\
&\stackrel{\text{pf. Lem}\ \ref{lm:con}}{\leq}\left(\int_{\R^d}d^2(x, A'_{N_m})\ d\mu_{N_m}(x)\right)^{1/2}+\left(\int_{\R^d}d_H^2(A''_{N_m},A'_{N_m})\ d\mu_{N_m}(x)\right)^{1/2}+2\e\\
&\stackrel{\text{def. $A'_{N_m}$}, \ \eqref{22}}{\leq}\left(\int_{\R^d}d^2(x, A)\ d\mu_{N_m}(x)\right)^{1/2}+3\e\\
&\stackrel{\eqref{44}}{\leq}\left(\int_{\R^d}d^2(x, A)\ d\mu(x)\right)^{1/2}+4\e. 
\end{align*}
Setting $\e\rightarrow 0$ yields 
\begin{align*}
\int_{\R^d}d^2(x, A_\star)\ d\mu(x)\leq\int_{\R^d}d^2(x, A)\ d\mu(x). 
\end{align*}
This finishes the proof of the first part of the theorem. The second part of the theorem follows by noting that whenever $\calA$ contains only one element, $A_\star$ does not depend on the choice of the subsequence $\{A_{N_m}\}$. A moment thought on the compactness of $\supp(\mu)$ reveals that the result in the first part holds for the whole sequence.   
\end{proof}

For the proof of \cref{t:rate}, we will need a uniform bound on the convergence rate on the empirical measure of $\mu$ in a class of test functions, which is given by the following lemmas. 
\begin{lem}[Dudley's inequality]\label{ddl}
Let $(X_t)_{t\in T}$ be a random process on a metric space $(T, d)$ with sub-gaussian increments:
\begin{align*}
&\| X_t -X_s\|_{\Psi_2}\leq K d(t,s),&\forall s, t\in T,
\end{align*}
where $\|\cdot\|_{\Psi_2}$ is the sub-gaussian norm and $K>0$ is a constant. 
Then, for every $u \geq 0$, the event
\begin{align*}
\sup_{t,s\in T} |X_t -X_s|\leq CK\left(\int_0^\infty\sqrt{\log\mathcal N(T,d,\e)}d\e + u\cdot \textrm{diam}(T)\right)
\end{align*}
holds with probability at least $1 -2 \exp(-u^2)$, where $\mathcal N(T,d,\e)$ is the $\e$-covering number of $T$ under metric $d$. 
\end{lem}

The proof of \cref{ddl} can be found in \cite[Theorem 8.16]{Vershynin_2018}. 
\begin{lem}\label{dd}
Under the same conditions as \cref{t:bdd consistency},
\begin{align*}
\sup_{A\in\{\co(\supp(\mu))\}^k} \left|\frac{1}{N}\sum_{i\in [N]}d^2(x_i, \co(A))-\int_{\R^d}d^2(x, A) d\mu\right| 
\ \lesssim \ 
\left(\frac{\log N}{N}\right)^{1/2} 
\quad 
 \mu\text{-} a.s.
\end{align*} 
\end{lem}
\begin{proof}
For fixed $N$, consider the random process
\begin{align*}
Y_A = N^{-1}\sum_{i\in [N]}d^2(x_i, \co(A))-\int_{\R^d}d^2(x, A) d\mu,
\end{align*}
where the index set is equipped with the metric $d_{2,\infty}$. One can verify that $\{d^2(x_i, \co(A))-d^2(x_i, \co(B))\}_{i\in [N]}$ are i.i.d. random variables which are bounded by $2Dd_{2,\infty}(A, B)$ for $A, B\in\{\co(\supp(\mu))\}^k$, where $D$ is the diameter of $\co(\supp(\mu))$. 
It follows from a similar computation as in \cite[Section 8.2.2]{Vershynin_2018} that 
\begin{align*}
\|Y_A-Y_B\|_{\Psi_2}\lesssim \frac{D}{\sqrt{N}}d_{2,\infty}(A, B).
\end{align*} 
Therefore, by \cref{ddl}, with probability at least $1-2\exp(-u^2)$,
\begin{align}
\sup_{A\in\{\co(\supp(\mu)\}^k}|Y_A|\lesssim \frac{D}{\sqrt{N}}\left(\int_0^\infty\sqrt{\log\mathcal N(\{\co(\supp(\mu))\}^k, d_{2,\infty}, \epsilon)} \ d\epsilon + kDu\right)+|Y_B|,\label{dudley}
\end{align}
where $B\in \{\co(\supp(\mu)\}^k$ is some fixed element. Note that 
\begin{align}
\mathcal N(\{\co(\supp(\mu))\}^k, d_{2,\infty}, \epsilon)&\leq \left(\mathcal N(\co(\supp(\mu)), \|\cdot\|_2, \epsilon)\right)^k\nonumber\\
&\leq \left(\mathcal N(B_z(D), \|\cdot\|_2, \epsilon/2)\right)^k\leq\left(\frac{6D}{\epsilon}\right)^{dk}\label{covering},
\end{align} 
where $z\in\co(\supp(\mu))$, $B_z(D)=\{x\colon \|x-z\|_2\leq D\}$, and the last inequality follows from estimates on the Euclidean covering number of $\ell_2$ balls, see \cite{Vershynin_2018}.
Plugging \eqref{covering} into \eqref{dudley} and setting $u=\sqrt{2\log N}$ yields that with probability at $1-2N^{-2}$,  
\begin{align}
\sup_{A\in\{\co(\supp(\mu)\}^k}|Y_A|\lesssim \left(\frac{\log N}{N}\right)^{1/2}+|Y_B|.\label{ull}
\end{align}
On the other hand, $Y_B$ is the sum of i.i.d. centered bounded random variables, by Hoeffding's inequality, $|Y_B|\lesssim (\log N/N)^{1/2}$ with probability at least $1-N^{2}$.
An application of the Borel--Cantelli lemma finishes the proof. 
\end{proof}

\begin{proof}[Proof of \cref{t:rate}]
Fix $N>0$. 
Adopting the same notation as in the previous proof, we let $A_\star$ and $A_N$ denote two archetype pointsets for the continuum and discrete problems, respectively. 
Moreover, by \cref{thm:bd}, we may assume that they are on the boundary of the convex hull of their respective support. 
Similar to \eqref{proj}, define $A'_\star$ as the projection of $A_\star$ to $\co(\text{supp}(\mu_N))$. It is easy to check that  
\begin{align}
d_H(A_\star, A_\star')\leq d_H(\co(\text{supp}(\mu)), \co(\text{supp}(\mu_N))). 
\end{align} 
Without loss of generality, 
we assume that $F_{\mu_N}(A_N)\geq F_\mu(A_\star)$. 
Repeating the same trick in the proof of \cref{t:bdd consistency}, 
\begin{align}
& F_\mu(A_N) - F_{\mu}(A_\star) \\
\leq&\ |F_\mu(A_N)-F_{\mu_N}(A_N)|+F_{\mu_N}(A_N)-F_\mu(A_\star)\nonumber\\
\leq& \ |F_\mu(A_N)-F_{\mu_N}(A_N)|+F_{\mu_N}(A_\star')-F_\mu(A_\star)\nonumber\\
\leq& \ |F_\mu(A_N)-F_{\mu_N}(A_N)|+|F_{\mu_N}(A_\star)-F_\mu(A_\star)|+|F_{\mu_N}(A'_\star)-F_{\mu_N}(A_\star)|\nonumber\\
\leq& \ \underbrace{|F_\mu(A_N)-F_{\mu_N}(A_N)|+|F_{\mu_N}(A_\star)-F_\mu(A_\star)|}_{(i)}+\underbrace{d_H(\co(\text{supp}(\mu)), \co(\text{supp}(\mu_N)))}_{(ii)}\to 0\nonumber
\end{align}
as $N\to\infty$, where $(i)\to 0$ due to \cref{dd} and $(ii)\to 0 $ due to the proof of \cref{t:bdd consistency}. 
This finishes the first part of the proof.  

To obtain the convergence rate, note that \eqref{ull} combined with the triangle inequality implies that $(i)\lesssim (\log N/N)^{1/2}$ with probability at least $1-3N^{2}$.  
Bounding $(ii)$ for an arbitrary $\mu$ is intractable in general. However, under the $\alpha$-cap condition (with parameters $L$ and $\eta$), it has been established in \cite[Theorem 1]{Brunel_2019} that there exists a sufficiently large integer $N(\eta)$ (non-increasing in $\eta$), such that for $N>N(\eta)$, with $\mu$ at least $1-\delta$, 
\begin{align}
d_H(\co(\text{supp}(\mu)), \co(\text{supp}(\mu_N)))\lesssim\left(\frac{d\log (1/\delta)}{NL}\right)^{1/\alpha},\label{213}
\end{align}
Setting $\delta = N^{-2}$ and noting $\alpha\geq d\geq 2$ when $\mu$ has a density completes the proof. 
\end{proof}

\begin{rem}\label{m-variance}
A very large variance\footnote{Note that variance is defined for real-valued random variables. For a random vector $X\in\R^d$, we consider $\E \left[ \|X-\E[X] \|_2^2 \right]$ as its (total) variance, which is the trace of its variance-covariance matrix.} may  accelerate the convergence in \eqref{rt}. To see this, for fixed set $D \subset \mathbb R^d$,  let 
\begin{align*}
\mathscr M=\{\nu \colon  \text{$\nu$ is a probability measure on $\R^d$ such that $\supp(\nu)\subset D$}\}.
\end{align*} 
Consider the elements in $\mathscr M$ that maximize the variance:
\begin{align} \label{VV}
\mathcal V = \arg\max_{\nu\in\mathscr M}\E_\nu \left[ \|x-\E_\nu[x] \|_2^2 \right].
\end{align}
We claim that every $\mu\in V$ is supported on at most $d+1$ affinely independent points in $D$; for a proof, see \cref{appendix2}. Therefore, any element in $\mathscr M$ with variance approximating the variance of $\mu$ should concentrate on some finite set of affinely independent points. 
\end{rem}

\section{Example: uniform probability measure on the unit disk} \label{s:ExUnitDisk}
In this section, we consider the continuous archetypal analysis problem \eqref{e:arch2} for the uniform probability measure on the unit disk. 

\begin{prop}
Let $k\geq 3$,  $\D \subset \R^2$ be the Euclidean unit disk, and   $\mu$ the uniform probability measure on $\D$. 
The minimizers to \eqref{e:arch2} are extremal points for regular $k$-polygons inscribed in $\D$ and have squared objective value
\[
F_\mu^2 = k I \left( \frac{2 \pi}{k} \right), 
\]
where $I\colon [0,\pi] \to \mathbb R$ is defined in \eqref{e:I}. 
The minimizers are unique modulo rotation. 
\end{prop}

\begin{proof}
Let $R_k$ denote a regular $k$-polygon inscribed in the unit disc and $A_k$ its vertices. 
Using symmetry, the  squared objective value for $A_k$ can be simplified, 
\begin{equation} \label{e:F2Rk}
F_\mu^2(A_k) 
 = \int_{\D} d^2\left( x, R_k \right) \  d \mu(x)
 = \frac{1}{\pi}\int_{\D} d^2\left( x,R_k \right) \  d x 
 = \frac{k}{\pi} \int_{\sector_{\frac{2 \pi}{k} }} d^2\left( x, R_k \right) \  d x. 
\end{equation}
Here, $\sector_{\frac{2 \pi}{k} }$ denotes a sector of the disk with angle $\frac{2 \pi}{k}$. 

For future reference, it is useful to record a slightly more general calculation. Let 
\[
I(\alpha) = 
\frac{1}{\pi}\int_{\sector_\alpha}d^2(x, \bigtriangleup_\alpha)\ dx, 
\qquad  \qquad  \alpha\in [0, \pi].
\]
Here, denoting 
$A=(1,0)$,  
$O = (0,0)$, and 
$B=(\cos\alpha, \sin\alpha)$, 
we have used the notation that 
$\sector_\alpha$ is the sector of the unit disc $\sector AOB$ and 
$\bigtriangleup_\alpha$ is the triangle with vertices $A$, $O$, and $B$. 
The chord $\overline{AB}$ lies on the line
\[
\left\{ (x,y) \in \mathbb R^2 \colon \begin{pmatrix} \cos \frac{\alpha}{2} \\  \sin \frac{\alpha}{2} \end{pmatrix} \cdot \begin{pmatrix} x \\  y \end{pmatrix}  = \cos \frac{\alpha}{2} \right\}, 
\]
so $\overline{AB}$ can be explicitly parameterized by 
\[\overline{AB} = \left\{ (r \cos \theta , r \sin \theta ) \colon  r = \frac{\cos \frac \alpha 2 }{ \cos ( \theta - \frac \alpha 2 )} , \theta \in [0,\alpha] \right\}. 
\]
For a point, $(r \cos \theta, r \sin \theta)  \in  \sector_\alpha \setminus \bigtriangleup_\alpha$, we compute 
\begin{align*}
d^2 \left( (r \cos \theta, r \sin \theta) ,   \bigtriangleup_\alpha \right) 
&= d^2 \left(  (r \cos \theta, r \sin \theta) , \overline{AB} \right) \\
 &=  \left( \begin{pmatrix} \cos \frac{\alpha}{2} \\  \sin \frac{\alpha}{2} \end{pmatrix} \cdot \begin{pmatrix} r \cos \theta \\  r \sin \theta \end{pmatrix}  - \cos \frac{\alpha}{2}  \right)^2 \\
 &= \left( r \cos \left( \theta - \frac \alpha 2 \right) - \cos \frac \alpha 2 \right)^2. 
\end{align*}
For $\alpha \in [0,\pi]$, we compute 
\begin{subequations} \label{e:I}
\begin{align}
I(\alpha) &= \frac{1}{\pi}\int_{\sector_\alpha}d^2(x, \bigtriangleup_\alpha)\ dx  \\
&=  \frac{2}{\pi}  \int_0^{\frac{\alpha}{2}} \int_{\frac{\cos\frac{\alpha}{2}}{\cos(\frac{\alpha}{2}-\theta)}}^1 
\left( r \cos \left( \theta - \frac \alpha 2 \right) - \cos \frac \alpha 2 \right)^2 \ r dr d\theta \\
& = \frac{1}{2 \pi}\left(\frac{\alpha}{4}-\frac{13}{12}\sin\alpha+ \alpha\cos^2\frac{\alpha}{2}-\frac{1}{3}\sin\frac{\alpha}{2}\cos^3\frac{\alpha}{2}\right),
\end{align}
\end{subequations}

\begin{figure}[t]
\begin{center}
\includegraphics[width=.4\textwidth]{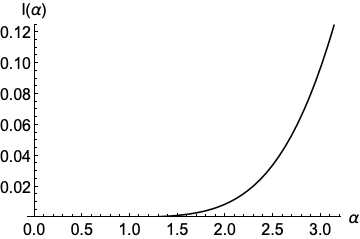}
\caption{The function $I \colon [0,\pi] \to \mathbb R$ defined in \eqref{e:I}.}
\label{f:Ifun}
\end{center}
\end{figure}

It can easily be checked that
\begin{itemize}
\item $I(0) = 0$ and $I(\pi) = \frac{1}{8}$, 
\item $I(\alpha)$ is an increasing positive function on $[0, \pi]$, and
\item $I(\alpha)$ is strictly convex on $[0, \pi]$.
\end{itemize}

From \eqref{e:F2Rk}, we obtain that $F_\mu^2(A_k) = k I \left(\frac{2 \pi}{ k} \right)$, which is a decreasing function in $k$. 

We now consider the optimization problem \eqref{e:arch2}. We first note that all the boundary points of $\D$ are extreme points, which has infinite cardinality. In this particular case, it is easy to check that for each $k\in\N$, the minimizers to \eqref{e:arch2} must lie on $\partial\D$, so their convex hull must be a cyclic polygon inscribed in the unit disc. Let $P_\alpha$ be a cyclic polygon inscribed in the unit disc with central angles $\alpha_1$, \ldots, $\alpha_k $ where $\alpha_j \in [0,2 \pi]$, $j \in [k]$ and $\sum \alpha_j = 2 \pi$. We first argue that if $P_\alpha$ were to be optimal, then it could not be contained in a  half disc, $H$, and therefore, we may assume that $\alpha_j \in [0, \pi]$, $j \in [k]$. This follows from the monotonicity property (if $P_\alpha \subset H$, then $F_\mu(P_\alpha) > F_\mu(H)$) and the fact that $F_\mu(H) > F_\mu(A_3)$. To see this last inequality, we compute 
\begin{align*}
F_\mu^2(H) = \int_{\mathbb D} d^2(x,H) \ d\mu (x)  = \frac{1}{\pi} \int_{\mathbb D}  \int_{\mathbb D} d^2(x,H) \ dx 
= I (\pi) 
= \frac{1}{8}
\end{align*}
and $F_\mu^2(A_3) = 3 I\left( \frac{2 \pi}{3} \right) \approx 0.035$.  

The optimization problem in  \eqref{e:arch2} then reduces to 
\begin{align*}
\min \ & \sum_{j=1}^k I( \alpha_j) \\
\textrm{s.t.} \ & \alpha_j \in [0,\pi]  \\
& \sum_{j=1}^k \alpha_j  = 2 \pi.
\end{align*}
This is a convex optimization problem with convex constraints. By Jensen's inequality, the unique solution is given by $\alpha_1 = \cdots = \alpha_k = \frac{2 \pi}{k} $, as desired. 
\end{proof}

\section{Proof of Theorems~\ref{existence:var-reg}, \ref{t:ConsVarReg}, and \ref{t:ConsVarReg-alpha} for probability measures with unbounded support}  
\label{s:UnboundedSupport}

We will establish existence of solutions to  \eqref{e:arch3} as stated in \cref{existence:var-reg}. We use the following lemma about the objective function, $F_{\nu,\alpha}\colon \{\R^d\}^k \to \R$, defined in \eqref{e:arch3}. 

\begin{lem}\label{var-reg:Lip}
For any square-integrable $\nu$, $F_{\nu,\alpha}$ is a $(1+2\sqrt{\alpha})$-Lipschitz continuous function from $(\{\R^d\}^k, d_{2, \infty})$ to $\R$.  
\end{lem}
\begin{proof}
For any $A_1=\{a^{(1)}_\ell\}_{\ell\in[k]}, A_2=\{a^{(2)}_\ell\}_{\ell\in[k]}\in\{\R^d\}^k$, let $\sigma$ be the permutation of $[k]$ such that 
\begin{align*}
d_{2,\infty}(A_1, A_2)=\max_{\ell\in [k]}\|a^{(1)}_\ell-a^{(2)}_{\sigma(\ell)}\|_2. 
\end{align*}
First applying the triangle inequality and then the elementary inequality $\sqrt{a+b}\leq\sqrt{a}+\sqrt{b}$, 
\begin{align*}
&|F_{\nu,\alpha}(A_1)-F_{\nu,\alpha}(A_2)|\\
\leq&\ \left|F_\nu(A_1)-F_\nu(A_2)\right|+\sqrt{\alpha}\left|\left(\frac{1}{k}\sum_{\ell\in [k]}\|a^{(1)}_\ell-\bar{a}^{(1)}\|_2^2\right)^{1/2}-\left(\frac{1}{k}\sum_{\ell\in [k]}\|a^{(2)}_\ell-\bar{a}^{(2)}\|_2^2\right)^{1/2}\right|.
\end{align*}
By \cref{lm:con}, the first term on the right-hand side is bounded by $d_{2, \infty}(A_1, A_2)$. The second term is further bounded using the triangle inequality as follows:
\begin{align*}
&\sqrt{\alpha}\left|\left(\frac{1}{k}\sum_{\ell\in [k]}\|a^{(1)}_\ell-\bar{a}^{(1)}\|_2^2\right)^{1/2}-\left(\frac{1}{k}\sum_{\ell\in [k]}\|a^{(2)}_\ell-\bar{a}^{(2)}\|_2^2\right)^{1/2}\right|\\
\leq &\ \sqrt{\alpha}\left(\left(\frac{1}{k}\sum_{\ell\in [k]}\|a^{(1)}_\ell-a^{(2)}_{\sigma(\ell)}\|_2^2\right)^{1/2}+\left(\frac{1}{k}\sum_{\ell\in [k]}\|\bar a^{(1)}-\bar a^{(2)}\|_2^2\right)^{1/2}\right)\\
\leq &\ 2\sqrt{\alpha}\ d_{2,\infty}(A_1, A_2). 
\end{align*}  
Putting the estimates together completes the proof.  
\end{proof}

\begin{proof}[Proof of \cref{existence:var-reg}]
Denote by $F^{\star}_{\nu, \alpha} \geq 0$ as the infimum objective value. Square-integrability of $\nu$ ensures that $F^{\star}_{\nu, \alpha}<\infty$. Let $A_m = \{a_\ell^{(m)}\}_{\ell\in [k]}$ be a non-increasing minimizing sequence such that for $m\geq 1$, 
\begin{align*}
2F^{\star}_{\nu, \alpha}\geq F_{\nu, \alpha}(A_1)\geq F_{\nu, \alpha}(A_m)\geq F_{\nu, \alpha}(A_{m+1}).
\end{align*}
We first observe that the diameter of $\{A_m\}_{m\in\N}$ is uniformly bounded. Let $r_1=4\sqrt{\frac{k}{\alpha}}F^{\star}_{\nu, \alpha}$. For any set $A=\{a_\ell\}_{\ell\in [k]}$ with $\text{diam}(A)> r_1$, 
\begin{align*}
F_{\nu, \alpha}(A)\geq\sqrt{\frac{\alpha}{k}\sum_{\ell\in [k]}\|a_\ell-\bar{a}\|_2^2}\geq\sqrt{\frac{\alpha}{k}}\frac{\text{diam}(A)}{2}>2F^{\star}_{\nu, \alpha}\geq F_{\nu, \alpha}(A_1). 
\end{align*}
This implies that $\sup_{m}\text{diam}(A_m)\leq r_1$. We next show that $\{A_m\}_{m\in\N}$ cannot escape to infinity as $m$ increases, which is equivalent to showing that $\{a^{(m)}_1\}_{m\in\N}$ is uniformly bounded. 
To see this, note that most mass of $\nu$ is concentrated in a ball centered at the origin. 
In particular, we can find a ball $B(r_2)$ with radius $r_2>0$ such that $\nu(B(r_2))\geq 1/2$.  We claim that $\{a^{(m)}_1\}_{m\in\N}$ is contained in a ball of radius $r_3:=r_1+r_2+2\sqrt{2}F^{\star}_{\nu, \alpha}$. Suppose $\|a^{(m)}_1\|_2>r_3$ for some $m$. Then the distance between any $x\in B(r_2)$ and $A_m$ is greater than $r_3-r_1-r_2=2\sqrt{2}F^{\star}_{\nu, \alpha}$.  Therefore,
\begin{align*}
F_{\nu, \alpha}(A_m)\geq\left(\int_{B(r_2)} d^2(x, A_m)\ d\nu(x)\right)^{1/2}>\left(\int_{B(r_2)} 8{(F^{\star}_{\nu, \alpha})}^2\ d\nu(x)\right)^{1/2} \geq 2F^{\star}_{\nu, \alpha}\geq F_{\nu, \alpha}(A_1), 
\end{align*}
contradicting the hypothesis that $\{A_m\}_{m\in\N}$ is a decreasing minimizing sequence.  Hence we have proved that $\{A_m\}_{m\in\N}$ is contained in $B(r_1+r_3)$, \ie,  $\{A_m\}_{m\in\N}$ is uniformly bounded. Applying the Bolzano--Weierstrass theorem componentwise yields that $\{A_m\}_{m\in\N}$ admits a convergent subsequence in the $d_{2, \infty}$ metric, say $A_{m_n}\rightarrow A_\star$ as $n\rightarrow\infty$. Since $\co(\supp(\nu))$ is closed and $A_m\subset\co(\supp(\nu))$ for every $m$, $A_\star\subset\co(\supp(\nu))$. It follows from the fact that $F_{\nu,\alpha}$ is continuous with respect to $d_{2, \infty}$-metric that
\begin{align*}
F_{\nu, \alpha}(A_\star)=\lim_{n\rightarrow\infty}F_{\nu,\alpha}(A_{m_n})=F^{\star}_{\nu, \alpha}. 
\end{align*}
Hence, $A_\star$ is a minimizer of problem \eqref{e:arch3}. 
The inequality $F_{\nu, \alpha}(A_\star)\leq (\int_{\mathbb R^d} \|x - \bar x\|_2^2 \ d \nu(x))^{1/2}$ follows from evaluating $F_{\nu, \alpha}$ at $A = \{\bar{x}\}^k$. 
\end{proof}

We will next prove the consistency result stated in \cref{t:ConsVarReg}. 

\begin{proof}[Proof of \cref{t:ConsVarReg}.]
Let $\calA$ be the solution set of \eqref{e:arch3} with $\nu=\mu$. Let $A_N$ be a minimizer to problem \eqref{e:arch3} with $\nu=\mu_N$ and $A\in\calA$. Since $\mu_N\rightharpoonup\mu$ $\mu$-a.s., it is possible to find a centered ball of sufficiently large radius $r_2$ such that 
\begin{align*}
\min \left\{\sup_{N\in\N}\mu_N(B(r_2)), \mu(B(r_2))\right\}\geq \frac{1}{2}.
\end{align*}
It follows from the proof of \cref{existence:var-reg} that 
\begin{align*}
A_N&\subset B\left((8k^{1/2}\alpha^{-1/2}+2\sqrt{2})F^{\star}_{\mu_N, \alpha}+r_2\right)\\
A&\subset  B\left((8k^{1/2}\alpha^{-1/2}+2\sqrt{2})F^{\star}_{\mu, \alpha}+r_2\right).
\end{align*}
To show that $\{A_N\}_{N\in\N}$ admits a convergent subsequence in $d_{2,\infty}$, it suffices to show that $F_{\mu_N, \alpha}(A_N)$ is uniformly bounded in $N$. To see this, note that for every $N$, 
\begin{align*}
F_{\mu_N, \alpha}(A_N)\leq F_{\mu_N, \alpha}(\{x_1\}^k) = \frac{1}{N}\sum_{i\in [N]}(x_i-x_1)^2\rightarrow\int_{\R^d}\|x-x_1\|_2^2\ d\mu(x)<\infty. 
\end{align*}
Hence, for
\begin{align*}
R:= \left( 8k^{1/2}\alpha^{-1/2} + 2 \sqrt{2} \right) \max\left\{\sup_{N\in\N}F^{\star}_{\mu_N, \alpha}, F^{\star}_{\mu, \alpha}\right\}+r_2<\infty,
\end{align*}
we have 
\begin{align*}
(\cup_{N\in\N}A_N)  \cup A \ \  \subset \ \  B(R)\cap\co(\supp(\mu)). 
\end{align*}
Since $B(R)\cap\co(\supp(\mu))$ is compact, by the Bolzano--Weierstrass theorem, $\{A_N\}_{N\in\N}$ has a convergent subsequence $\{A_{N_m}\}_{m\in\N}$ such that $A_{N_m}\rightarrow A_\star$ in $d_{2,\infty}$ and in the Hausdorff metric. It is clear that $A_\star\subset \co(\supp(\mu))$ thus is a natural candidate for the minimizer of the continuous problem. To show that it is, we use the fact that for any $\e>0$, there exists a sufficiently large $M$ such that the following statements hold true whenever $m>M$:
\begin{align}
A\subset B(R)&\subset \co(X_{N_m})\label{111}\\
d_H(A_{N_m}, A_\star)&\leq \e\label{222}\\
\int_{\R^d}d^2(x, A_\star)\ d\mu(x)&\leq \int_{\R^d}d^2(x, A_\star)\ d\mu_{N_m}(x) +\e\label{333}\\
\int_{\R^d}d^2(x, A)\ d\mu_{N_m}(x)&\leq\int_{\R^d}d^2(x, A)\ d\mu(x) +\e.
\label{444}
\end{align}
Therefore, for $m>M$, 
\begin{align*}
F_{\mu, \alpha}(A_\star) &= \left(\int_{\R^d}d^2(x, A_\star)\ d\mu(x) + \alpha V(A_\star)\right)^{1/2}\\
&\stackrel{\eqref{333}}{\leq} \left(\int_{\R^d}d^2(x, A_\star)\ d\mu_{N_m}(x) + \alpha V(A_\star)+\e\right)^{1/2}\\
&\leq \left(\int_{\R^d}d^2(x, A_\star)\ d\mu_{N_m}(x) + \alpha V(A_\star)\right)^{1/2}+\e^{1/2}\\
&\stackrel{\eqref{222}, \ \text{pf. Lem}\ \ref{var-reg:Lip}}{\leq} \left(\int_{\R^d}d^2(x, A_{N_m})\ d\mu_{N_m}(x) + \alpha V(A_{N_m})\right)^{1/2}+(1+2\sqrt{\alpha})\e+\e^{1/2}\\
&\stackrel{\text{def. $A_{N_m}$}, \ \eqref{111}}{\leq}  \left(\int_{\R^d}d^2(x, A)\ d\mu_{N_m}(x) + \alpha V(A)\right)^{1/2}+(1+2\sqrt{\alpha})\e+\e^{1/2}\\
& \stackrel{\eqref{444}}{\leq} F_{\mu, \alpha}(A)+(1+2\sqrt{\alpha})\e+2\e^{1/2}. 
\end{align*}
Taking $\e\rightarrow 0$ yields $F_{\mu, \alpha}(A_\star)\leq F_{\mu, \alpha}(A)$ which completes the first part of the theorem.  The second part of the theorem follows by noting that whenever $\calA$ contains only one element, $A_\star$ does not depend on the choice of the subsequence $\{A_{N_m}\}$, hence the whole sequence converges to $A_\star$. 
\end{proof}

\begin{proof}[Proof of \cref{t:ConsVarReg-alpha}.]
Let $A_{\star, \alpha}$ be a minimizer of \eqref{e:arch3} and $\bar{a}_\alpha = \int_{A_{\star, \alpha}} x \ d\nu(x)$   be the center of $A_{\star, \alpha}$. It follows from the proof of \cref{existence:var-reg} that 
\begin{align*}
\text{diam}(A_{\star, \alpha})\leq 4k^{1/2}\alpha^{-1/2}F^{\star}_{\nu, \alpha}\leq 4k^{1/2}\alpha^{-1/2}\left(\int_{\R^d}\|x- \bar{x}\|_2^2\ d\nu(x)\right)^{1/2}. 
\end{align*}
where the second inequality follows by evaluating $F_{\nu, \alpha}$ at $A=A_{\star, \infty}$. Therefore, 
\begin{align}
F^{\star}_{\nu, \alpha} &= F_{\nu, \alpha}(A_{\star, \alpha})\nonumber\\
&\geq\left(\int_{\R^d}d^2(x, A_{\star, \alpha})\ d\nu(x)\right)^{1/2}\nonumber\\
&\geq \left(\int_{\R^d}\|x- \bar{a}_\alpha\|_2^2\ d\nu(x)\right)^{1/2} -\text{diam}(A_{\star, \alpha})\nonumber\\
&\geq \left(\int_{\R^d}\|x- \bar{x}\|_2^2\ d\nu(x) + \|\bar{a}_\alpha-\bar{x}\|_2^2\right)^{1/2} - 4k^{1/2}\alpha^{-1/2}\left(\int_{\R^d}\|x- \bar{x}\|_2^2\ d\nu(x)\right)^{1/2}.\label{haha}
\end{align}
Since $F^{\star}_{\nu, \alpha}\leq(\int_{\R^d}\|x- \bar{x}\|_2^2\ d\nu(x))^{1/2}$,  $\|\bar{a}_\alpha-\bar{x}\|_2\rightarrow 0$ as $\alpha\rightarrow\infty$. Therefore, there exists a sufficiently large $T$ (which only depends on $\nu$ and $k$) such that 
\begin{align}
&\|\bar{a}_\alpha-\bar{x}\|^2_2\leq 3\int_{\R^d}\|x-\bar{x}\|_2^2\ d\nu(x)& \alpha>T.\label{???}
\end{align}
Inserting \eqref{???} this into \eqref{haha} and then applying the elementary inequality $\sqrt{x+y}\geq\sqrt{x}+\frac{y}{2\sqrt{x+y}}$ ($x,y>0$) yields
\begin{align*}
F^{\star}_{\nu, \alpha}&\geq\left(\int_{\R^d}\|x- \bar{x}\|_2^2\ d\nu(x)\right)^{1/2} +\frac{1}{2}\left(\int_{\R^d}\|x- \bar{x}\|_2^2\ d\nu(x)+\|\bar{a}_\alpha-\bar{x}\|_2^2\right)^{-1/2} \|\bar{a}_\alpha-\bar{x}\|_2^2\\
&\ \ \ \ -4k^{1/2}\alpha^{-1/2}\left(\int_{\R^d}\|x- \bar{x}\|_2^2\ d\nu(x)\right)^{1/2}\\
&\geq \left(\int_{\R^d}\|x- \bar{x}\|_2^2\ d\nu(x)\right)^{1/2} +\frac{1}{4}\left(\int_{\R^d}\|x- \bar{x}\|_2^2\ d\nu(x)\right)^{-1/2} \|\bar{a}_\alpha-\bar{x}\|_2^2\\
 &\ \ \ \ -4k^{1/2}\alpha^{-1/2}\left(\int_{\R^d}\|x- \bar{x}\|_2^2\ d\nu(x)\right)^{1/2}\\
 &\geq F^{\star}_{\nu, \alpha}+\frac{1}{4}\left(\int_{\R^d}\|x- \bar{x}\|_2^2\ d\nu(x)\right)^{-1/2} \|\bar{a}_\alpha-\bar{x}\|_2^2-4k^{1/2}\alpha^{-1/2}\left(\int_{\R^d}\|x- \bar{x}\|_2^2\ d\nu(x)\right)^{1/2}, 
\end{align*}
from which one deduces that 
\begin{align*}
\|\bar{a}_\alpha -\bar{x}\|_2\leq 4k^{1/4}\alpha^{-1/4}\left(\int_{\R^d}\|x- \bar{x}\|_2^2\ d\nu(x)\right)^{1/2}. 
\end{align*}
Consequently, 
\begin{align*}
d_{2, \infty}(A_{\star, \alpha}, A_{\star, \infty}) \leq \|\bar{a}_\alpha-\bar{x}\|_2 + \text{diam}(A_{\star, \alpha})\leq 8k^{1/2}\alpha^{-1/4}\left(\int_{\R^d}\|x- \bar{x}\|_2^2\ d\nu(x)\right)^{1/2},
\end{align*}
as desired. 
\end{proof}

\section{Numerical implementation and experiments} \label{s:NumEx}
In this section, we first discuss an algorithm and its numerical implementation for the archetypal analysis problem  for a distribution with bounded support \eqref{e:arch} and the modified method  for measures with non-compact support \eqref{pro:unbounded}. We then present the results of various numerical experiments that support and complement our analysis. 

\subsection{Numerical methods and their implementation}
\subsubsection{Numerical method and implementation for \eqref{e:arch}}\label{sec:imp1}

The archetypal analysis problem \eqref{e:arch} can be equivalently rewritten as 
\begin{subequations}
\label{e:archetypal}
\begin{align}
\min_{\mathcal A \in \mathbb R^{N \times k}, \ \mathcal B \in \mathbb R^{k \times N}} \ & \frac{1}{N} \| X - X \mathcal A  \mathcal B \|^2_F \\
\textrm{s.t.}, \ & \mathcal A _{ij}, \mathcal B_{ij} \geq 0 \\ 
& \mathcal A^t 1 = 1, \quad  \mathcal B^t 1 = 1
\end{align}
\end{subequations}
where $X = [x_1\mid \cdots \mid x_N ] \in \mathbb R^{d \times N}$, and $1$ denotes vector of  the all ones of compatible size. Furthermore, we denote  $\mathcal Z \colon = X\mathcal A = [a_1\mid \cdots \mid a_k ] \in \mathbb R^{d \times k}$.

Motivated by \cite{Cutler_1994,M_rup_2012},  we use an alternating algorithm to update $\mathcal Z$ (\ie, update $\mathcal A$) and $\mathcal B$, respectively. That is, starting from some initial $\mathcal Z^{(0)}$, we find a sequence of solutions, 
\[\mathcal B^{(0)}, \ \mathcal Z^{(1)}, \ \mathcal B^{(1)}, \cdots, \  \mathcal B^{(p)}, \ \mathcal Z^{(p+1)}, \  \cdots, \]  
where 
\begin{align}\label{prob:updateb}
\mathcal B^{(p)} = \arg \min_{\mathcal B} \ & \frac{1}{N} \| X -  \mathcal Z^{(p)}  \mathcal B \|^2_F \\
\textrm{s.t.}, \ & \mathcal B_{ij} \geq 0\ \textrm {and}  \  \mathcal B^t 1 = 1 \nonumber
\end{align}
and 
\begin{align}\label{prob:updatez}
\mathcal Z^{(p+1)} = \arg \min_{\mathcal Z} \ & \frac{1}{N} \| X -  \mathcal Z  \mathcal B^{(p)} \|^2_F \\
\textrm{s.t.}, \ & \mathcal A_{ij} \geq 0, \ \  \mathcal A^t 1 = 1, \ \textrm {and}  \ \mathcal Z = X \mathcal A \nonumber. 
\end{align}

We first use the Gauss--Seidel strategy to solve \eqref{prob:updatez} by updating the columns of $\mathcal Z$ (\ie, $a_i$, $i\in [k]$) sequentially. 
For $\ell\in [k]$, the $\ell$-th column of $\mathcal{Z}^{(p)}$ is updated by 
\begin{align}\label{prob:updatez1}
a_\ell^{(p+1)} = \arg \min_{a_\ell} \ & \frac{1}{N} \bigg|\bigg| \frac{\sum_{i=1}^N \mathcal{B}^{(p)}_{\ell i} \xi_i}{\sum_{i=1}^N {\mathcal{B}^{(p)}_{\ell i}}^2} -  a_\ell   \bigg|\bigg|^2_2 \\
\textrm{s.t.}, \ & \mathcal A_{ij} \geq 0, \ \  \mathcal A^t 1 = 1, \ \textrm {and}  \ a_\ell = \{X \mathcal A\}_\ell \nonumber
\end{align}
where $\{\cdot\}_\ell$ denotes the $\ell$-th column,  and 
\begin{align}
\xi_i =x_i - \sum_{s \neq \ell}^k \mathcal{B}^{(p)}_{si}a_s 
\qquad \qquad \forall \  i\in [N].
\label{rainy}
\end{align} 
The details of the derivation of \eqref{prob:updatez1} is given in \cref{appendix}.
It's easy to see that \eqref{prob:updateb} and \eqref{prob:updatez1} are summarized by the following general convex least squares problem:
\begin{align}\label{cls}
\min_{\omega \in  \mathcal S} \ &  \| u -  C \omega \|_2^2
\end{align}
where 
\[ \mathcal S = \{\omega = (\omega_1,\omega_2, \cdots, \omega_q)  \colon \omega_i \geq 0\ \textrm {and}  \  \sum_{i=1}^q \omega_i = 1\}\] 
is the unit simplex and  $u \in  \mathbb R^{n}$ and $C = [t_1 \ | \ t_2 \ | \cdots | \  t_q] \in \mathbb R^{n \times q} $ are given. 
\eqref{cls} is a standard convex programming problem which can be effectively solved by the projected gradient descent method, as summarized in \cref{alg1}.
 
\begin{algorithm}[t!]

{\bf Input}: {Let $\omega^0$ be the initialization, $\tau>0$} \\
{\bf Output}: { $\omega^n$ that approximately solves  \eqref{cls}.} \\
 Set $s=1$\; \\
 {\bf While} {\it not converged} {\bf do} \\
{\bf 1. Gradient Descent.} Solve the initial value problem until time $\tau$ with initial value given by $\omega^{s-1}$:
\begin{align}\label{eq:updatew}
\begin{cases}
\frac{d \omega}{dt} = -C^TC\omega +C^T u \\
\omega(0) = \omega^{s-1}.
\end{cases}
\end{align}
Let $\tilde \omega = \omega(\tau)$.\; \\
{\bf 2. Projection Step.} Project $\tilde \omega$ to $\mathcal S$ to obtain $\omega^s$. \; \\
Set $s = s+1$\;
\caption{The projected gradient descent method to approximately solve \eqref{cls}. } 
\label{alg1}
\end{algorithm}

When $n$ and $q$ are relatively small, the exact solution of step 1 in \cref{alg1} can be efficiently computed as follows. Assume that $C^TC$ has eigenvalue decomposition $V\Sigma V^T$ with eigenvalues $\sigma_i$, $i \in [q]$. The solution to \eqref{eq:updatew} is
\begin{equation}\label{sol:odesystem}
\omega(t) = VP
\end{equation}
where $P \in \mathbb R^{q\times 1}$ with entries 
\begin{equation}
P_i = \begin{cases} e^{-\sigma_i t}P_i^0 - \frac{1}{\sigma_i}e^{-\sigma_i t}(V^TC^Tu)_i+\frac{1}{\sigma_i}(V^TC^Tu)_i ,\ &  \  \textrm{if}  \ \ \sigma_i\neq 0, \\
P_i^0 +t (V^TC^Tu)_i , \ &  \  \textrm{if}  \ \ \sigma_i = 0 \\
 \end{cases}
\end{equation}
and $P^0 = V^T\omega(0)$. For the high-dimensional and large data set, one can simply use the forward Euler method to approximate the solution \cite{M_rup_2012}.

\medskip

Step 2 of \cref{alg1} can be efficiently solved by writing it as the optimization problem
\begin{subequations}\label{prob:projection}
\begin{align} 
\min_\omega \ & \frac{1}{2}\|\omega-\tilde \omega\|_2^2 \\ 
\textrm{s.t.} \ & \sum_{i=1}^n w_i = 1, \ w_i \geq 0
\end{align}
\end{subequations}
and using the  projection algorithm  proposed in \cite{Duchi_2008}, which we summarize in \cref{alg2}. 
\begin{algorithm}[t!]
 {\bf Input}: {A vector $\tilde \omega \in \mathbb R^n$} \\
 {\bf Output}: {$\omega$ solves \eqref{prob:projection}.} \\
 {\bf Step 1.} Sort $\tilde \omega$ into $\mu$: $\mu_1 \geq \mu_2 \geq \cdots \geq \mu_n$; \\
 {\bf Step 2.} Find $$\rho = \max \left\{ j \in [n]\colon \mu_j - \frac{1}{j} \left(\sum_{r=1}^j\mu_r -1\right)>0\right\} ;$$ \\
 {\bf Step 3.} Define $$\theta = \frac{1}{\rho} \left(\sum_{i=1}^\rho \mu_i -1 \right);$$
 {\bf Step 4.} Set $$\omega_i = \max\{\tilde{\omega}_i - \theta,0\}.$$ 
\caption{Algorithm for the projection onto the simplex \cite{Duchi_2008}.} 
\label{alg2}
\end{algorithm}

\subsubsection{Numerical method and implementation for \eqref{pro:unbounded}}
Similarly to \cref{sec:imp1}, we can equivalently write the variance-regularized archetypal analysis problem \eqref{pro:unbounded} in the following  form:
\begin{subequations}
\label{e:archetypal1}
\begin{align}
\min_{\mathcal A \in \mathbb R^{N \times k}, \ \mathcal B \in \mathbb R^{k \times N}} \ & \frac{1}{N} \| X - X \mathcal A  \mathcal B \|^2_F +\frac{\alpha}{k} \sum_{i=1}^k\left\|a_i - \frac{1}{k}\sum_{j=1}^k a_j \right\|_2^2 \\
\textrm{s.t.}, \ & \mathcal A _{ij}, \mathcal B_{ij} \geq 0 \\ 
& \mathcal A^t 1 = 1, \quad  \mathcal B^t 1 = 1
\end{align}
\end{subequations} 
where $\alpha>0$ is a parameter. Again, $X = [x_1\mid \cdots \mid x_k ] \in \mathbb R^{d \times N}$, $1$ denotes the all $1$ vector of compatible size, and $\mathcal Z \colon = X\mathcal A = [a_1\mid \cdots \mid a_k ] \in \mathbb R^{d \times k}$.

Again, we use an alternating algorithm to update $\mathcal Z$ (\ie, update $\mathcal A$) and $\mathcal B$. That is, starting from some initial $\mathcal Z^{(0)}$, we find a sequence of solutions, 
$ \mathcal B^{(0)}$, $\mathcal Z^{(1)}$, $\mathcal B^{(1)}$, $\cdots$, $\mathcal B^{(p)}$, $\mathcal Z^{(p+1)}$,  $\cdots$,  
where 
\begin{align}\label{prob:updateb1}
\mathcal B^{(p)} = \arg \min_{\mathcal B} \ & \frac{1}{N} \| X -  \mathcal Z^{(p)}  \mathcal B \|^2_F +\frac{\alpha}{k} \sum_{i=1}^k\left\|a_i^{(p)} - \frac{1}{k}\sum_{j=1}^k a_j^{(p)} \right\|_2^2\\
\textrm{s.t.}, \ & \mathcal B_{ij} \geq 0\ \textrm {and}  \  \mathcal B^t 1 = 1 \nonumber
\end{align}
and 
\begin{align}\label{prob:updatez2}
\mathcal Z^{(p+1)} = \arg \min_{\mathcal Z} \ & \frac{1}{N} \| X -  \mathcal Z  \mathcal B^{(p)} \|^2_F+\frac{\alpha}{k} \sum_{i=1}^k\left\|a_i^{(p)} - \frac{1}{k}\sum_{j=1}^k a_j^{(p)} \right\|_2^2 \\
\textrm{s.t.}, \ & \mathcal A_{ij} \geq 0, \ \  \mathcal A^t 1 = 1, \ \textrm {and}  \ \mathcal Z = X \mathcal A \nonumber. 
\end{align}
We note that \eqref{prob:updateb1} is exactly same as \eqref{prob:updateb} when $\mathcal Z^{(p)}$ is fixed because the additional term is independent of  $\mathcal B$. As for \eqref{prob:updatez2}, we continue to use the Gauss--Seidel strategy. For $\ell\in [k]$, the $\ell$-th column of $\mathcal{Z}^{(p)}$ is updated by

\begin{align}\label{prob:updatez3}
a_\ell^{(p+1)} =\arg \min_{a_\ell} \ &  \bigg|\bigg| \frac{\frac{1}{N}\sum_{i=1}^N \mathcal{B}^{(p)}_{\ell i} \xi_i +\frac{\alpha}{k^2}\sum_{s\neq \ell}^k a_s }{\frac{1}{N}\sum_{i=1}^N {\mathcal{B}^{(p)}_{\ell i}}^2+\frac{\alpha(k-1)}{k^2}} -  a_\ell   \bigg|\bigg|^2_2 \\
\textrm{s.t.}, \ & \mathcal A_{ij} \geq 0, \ \  \mathcal A^t 1 = 1, \ \textrm {and}  \ a_\ell = \{X \mathcal A\}_\ell \nonumber
\end{align}
where $\{\cdot\}_\ell$ denotes the $\ell$-th column and $\xi_i$ is the same as \eqref{rainy}.  
The details of this derivation are given in \cref{appendix}.  Again, both problems (\ie, \eqref{prob:updateb1} and \eqref{prob:updatez2}) can be reduced to solving \eqref{cls} which can be efficiently solved via \cref{alg1,alg2}.

\begin{rem}
In \eqref{sol:odesystem}, as we mentioned, the exact solution is only possible when $n$ and $q$ are relatively small. It would be very expensive for high-dimensional and large data. However, if $q \gg n$ when $n$ is relatively small, to accelerate this procedure, we choose a convenient representation based on Carath\'eodory's theorem, which states that every point in a compact convex set in $\R^n$ can be written as a convex combination of $n+1$ extreme points. This suffices to set $X$ in \eqref{prob:updatez1} as extreme points of $\partial(\co(X_N))$, and the resulting $q$ in the computation is dramatically reduced. 
\end{rem}

\subsection{Numerical experiments}\label{sec:num}
In this section, we illustrate the performance of the algorithm and verify the above consistency results via several computational experiments. 
We implemented the algorithms in MATLAB. All reported results were obtained on a laptop with a 2.7GHz Intel Core i5 processor and 8GB of RAM.  In all experiments, we set $\tau = 0.5$. We choose the stopping criteria for the iteration to be $\|\mathcal Z^{(p+1)}-\mathcal Z^{(p)}\|_F^2 < tol$ and, in all experiments, set $tol = 1\times 10^{-6}$.

\subsubsection{Example 1: Uniform distribution in a disk} \label{sec:uniform}
We first consider the case where the distribution is the uniform distribution in a unit disc. That is, each data point $(x,y)$ is randomly generated by 
\[\begin{cases} x = \sqrt{r} \cos (2\pi\theta) \\
y = \sqrt{r} \sin (2\pi\theta)
\end{cases}
\]
where $r$ and $\theta$ are both random numbers generated from the uniform distribution on $[0,1]$ (\ie,  $U([0,1])$).

In this experiment, we study, when $k$ is fixed, how the convex hull of archetype points, $\co\left(\{a_\ell \}_{\ell \in [k]} \right)$, changes as the number of data points, $N$, increases. From the consistency result in \cref{t:bdd consistency} and the analytical result from \cref{s:ExUnitDisk} for this problem, we know that the optimal polytope should converge to the regular one with unit-length vertices. To measure this convergence, we report the maximum and minimum values of the angles at extreme points. 
The reported results were obtained by repeating the same experiment $100$ times and taking their mean for each value of $N$. 
\cref{fig:uniform} shows the change of angles as $N$ increases 
from $N=4$ to $13$ with $\Delta N = 1$, 
from $N=13$ to $333$ with $\Delta N= 10$, and 
from $N=333$ to $30,033$ with $\Delta N = 300$,
where $\Delta N$ is the step size. 
We observe that, as $N$ increases,  the maximum angle and the minimum angle converge to the same value, indicating that $\co\left(\{a_\ell \}_{\ell \in [k]} \right)$ converges to an equilateral triangle. 
The right figure in \cref{fig:uniform} is the plot of a minimizer for $30,033$ random data points. 
The red curve is the boundary of the convex hull of the archetype points and the magenta curve is the boundary of the convex hull of random data points.

\begin{figure}[ht]
\centering
\includegraphics[width = 0.40\textwidth,clip,trim=2cm 1cm 3cm 1cm]{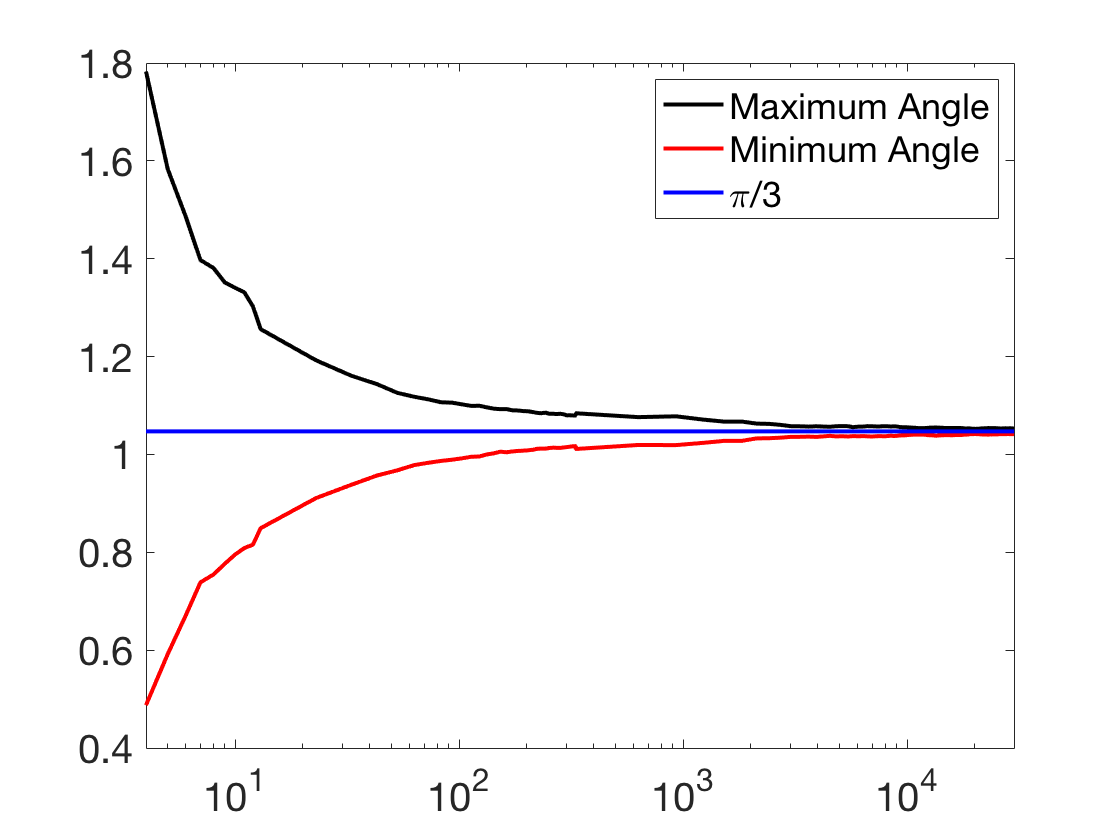}
\includegraphics[width = 0.40\textwidth,clip,trim=2cm 0cm 3cm 1cm]{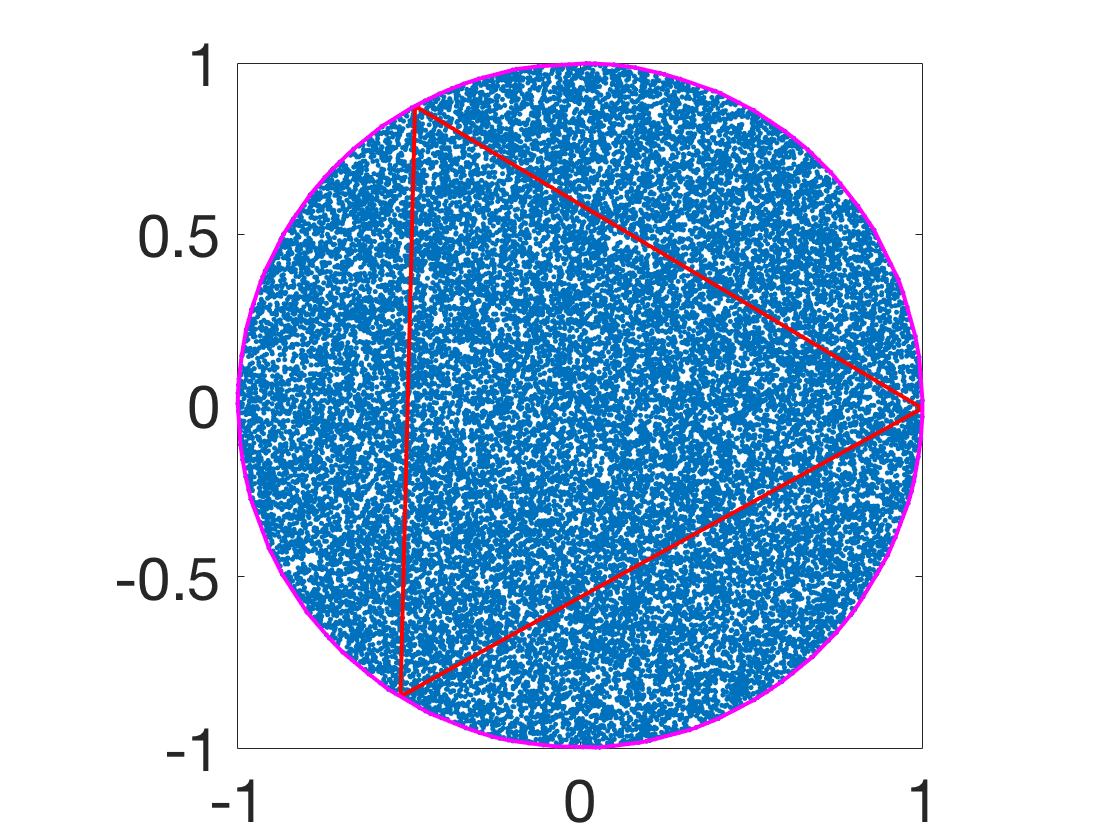}
\caption{{\bf (left)} The maximum and minimum angles of the convex polytope as $N$ increases from $4$ to $30,033$. This result is obtained by averaging over $100$ numerical experiments.  {\bf (right)} A minimizer for $30,033$ random data points. The red curve is the boundary of the convex hull of the archetype points and the magenta curve is the boundary of the convex hull of random data points.  See \cref{sec:uniform}.} \label{fig:uniform}
\end{figure}

\subsubsection{Example 2: Normal distribution} \label{sec:unbounded1}
In this example, we consider the case where the support of the distribution is non-compact. In this case, to obtain a solution with bounded support in the limit as $N\to \infty$, a penalty term is added to the objective functional as in \eqref{pro:unbounded}.

In this experiment, for a fixed $k =3$, we study the behavior of the solution as $\alpha$ increases.   We first generate $N= 30,000$ random points from a standard normal distribution $\mathcal N\left(0,10 \cdot I\right)$ where  $I$ is the identity matrix. 
In \cref{fig:normal1}, we plot the initialization and the solution to the archetypal analysis problem for $\alpha = 0$. 
We observe that all archetype points are located on the boundary of the convex hull of the data set, as expected by \cref{thm:bd}. Here, we note that, when $\alpha=0$, as $N\rightarrow \infty$, the convex polytope should grow to infinity in all directions. However, for a fixed $N$, we have a bounded data set and thus observe similar behavior to that in \cref{sec:uniform}. In this case with $30,000$ random data points, the area of the final $\co(A)$ is $233$. When $\alpha$ is slightly increased to $0.1$, the area of $\co(A)$ decreases to about $39$.  
We further study the dependence of the area of $\co(A)$ on $\alpha$ in the next experiment. 

In this experiment, we gradually increase the value of $\alpha$ and find the corresponding archetype points, $A$, by solving \eqref{pro:unbounded}. 
In \cref{fig:normal2}, we plot the area of $\co(A)$ for various values of $\alpha$. 
Here, we let $\alpha$ vary from $0.1$ to $5$ with step size $0.1$. 
We also plot the convex polytope, $\co(A)$, for some values of $\alpha$. 
We observe that $\co(A)$ shrinks as $\alpha$ increases, which is consistent with intuition. 
Also, for this isotropic distribution, we observe that the solution is always close to a regular triangle. 
In each subfigure in \cref{fig:normal2}, we indicate the mean of the distribution $\begin{pmatrix} 0 \\ 0 \end{pmatrix}$ with a black point for comparison with \cref{t:ConsVarReg-alpha}.  In all following numerical experiments, \ie, initializations and data samples $X_N$, we observe that the mean of the distribution is contained in $\co(A)$.

We expect the cost function of a minimizer $A_\star$ to be balanced in the sense that $F_\nu^2(A_\star) $ and $\alpha V(A_\star)$ are on the same order, \ie, 
\[
F_\nu^2(A_\star) \sim \alpha V(A_\star) \sim (F_{\nu,\alpha}^\star)^2. 
\] 
Since we also have that $V(A_\star) \sim \textrm{diam}(A_\star)^2$, where $\textrm{diam}(A)$ denotes the diameter of $\co(A)$, we obtain 
\[
\textrm{diam}(A_\star) \sim \frac{F_\nu^2(A_\star)}{\alpha}. 
\]
The scaling in \cref{fig:normal2} is consistent with this argument.

\begin{figure}[ht]
\centering
\includegraphics[width = 0.35\textwidth,clip,trim=6cm 1cm 6cm 2cm]{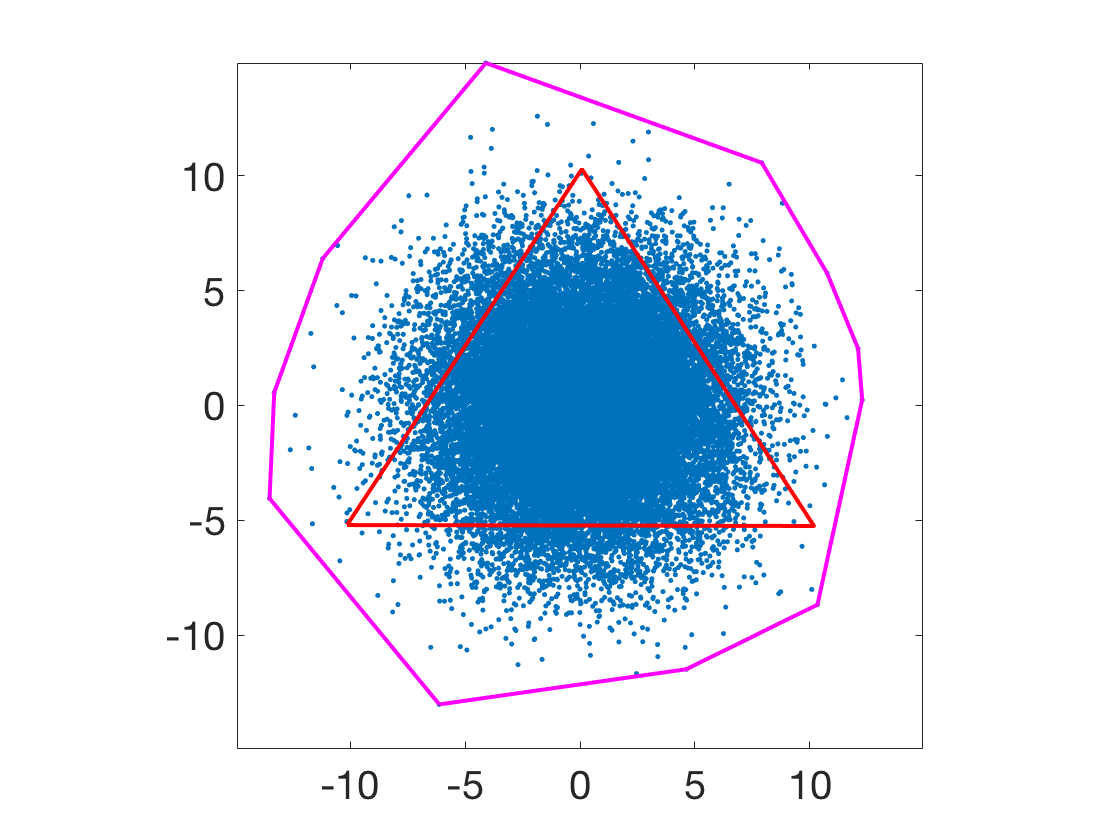}
\includegraphics[width = 0.35\textwidth,clip,trim=6cm 1cm 6cm 2cm]{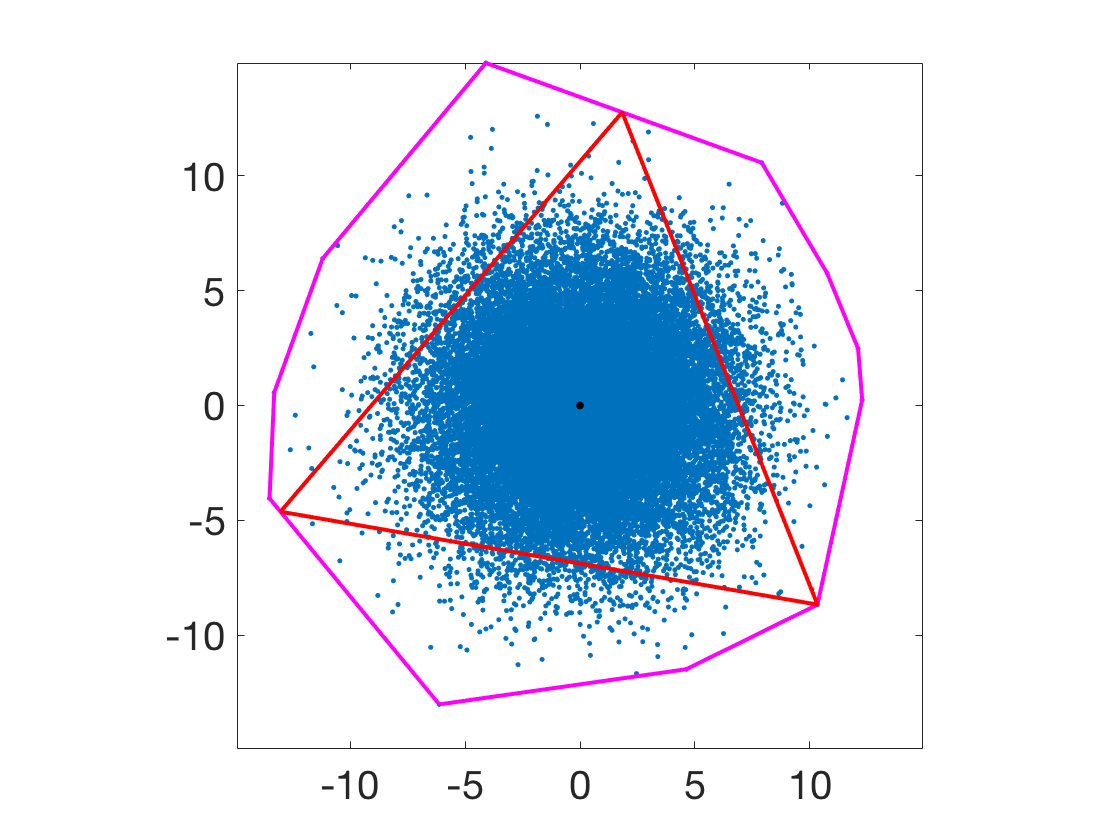}
\caption{{\bf (left)} $N=30,000$ random points sampled from a normal distribution. The red curve represents the initialization of the archetype polygon. {\bf (right)} The optimal solution of $A$ when $\alpha =0$. See \cref{sec:unbounded1}.} \label{fig:normal1}
\end{figure}

\begin{figure}[ht]
\centering
\includegraphics[width = 0.75\textwidth,clip,trim = 10cm 0cm 8cm 0cm]{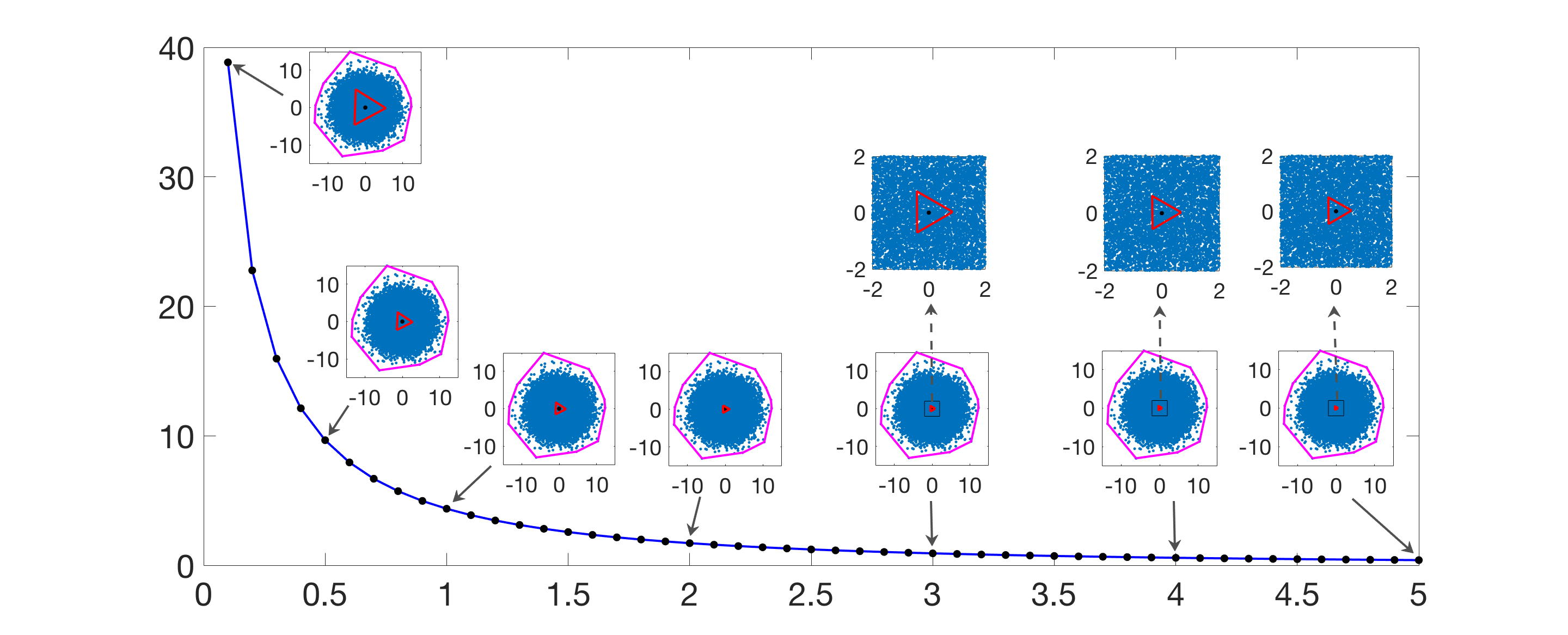}
\caption{The area of $\co(A)$ as we vary $\alpha$ with $N = 30,000$ random data points generated from a normal distribution.  See \cref{sec:unbounded1}.} 
\label{fig:normal2}
\end{figure}

\subsubsection{Example 3: Annular distribution} \label{sec:unbounded2}
In this experiment, we consider the case when the distribution is an annular distribution where $N$ random points $(x_i,y_i)$, $i\in[N]$ are generated by
\begin{equation} \label{eq:bimodalrandom}
\begin{cases} x_i =  \sqrt{|r_i|} \cos (2\pi\theta_i) \\
y_i = \sqrt{|r_i|}  \sin (2\pi\theta_i).
\end{cases}
\end{equation}
Here $\theta_i$ is a random number generated from the uniform distribution in $[0,1]$ (\ie,  $U([0,1])$) and $r_i$ is a random number generated from the normal distribution $\mathcal N(25, 50)$ (See \cref{fig:bimodal1} for an example with $N=30,000$). Note that the random points are concentrated around the circle with radius 5. 

\begin{figure}[ht!]
\centering
\includegraphics[width = 0.35\textwidth,clip,trim=6cm 1cm 6cm 2cm]{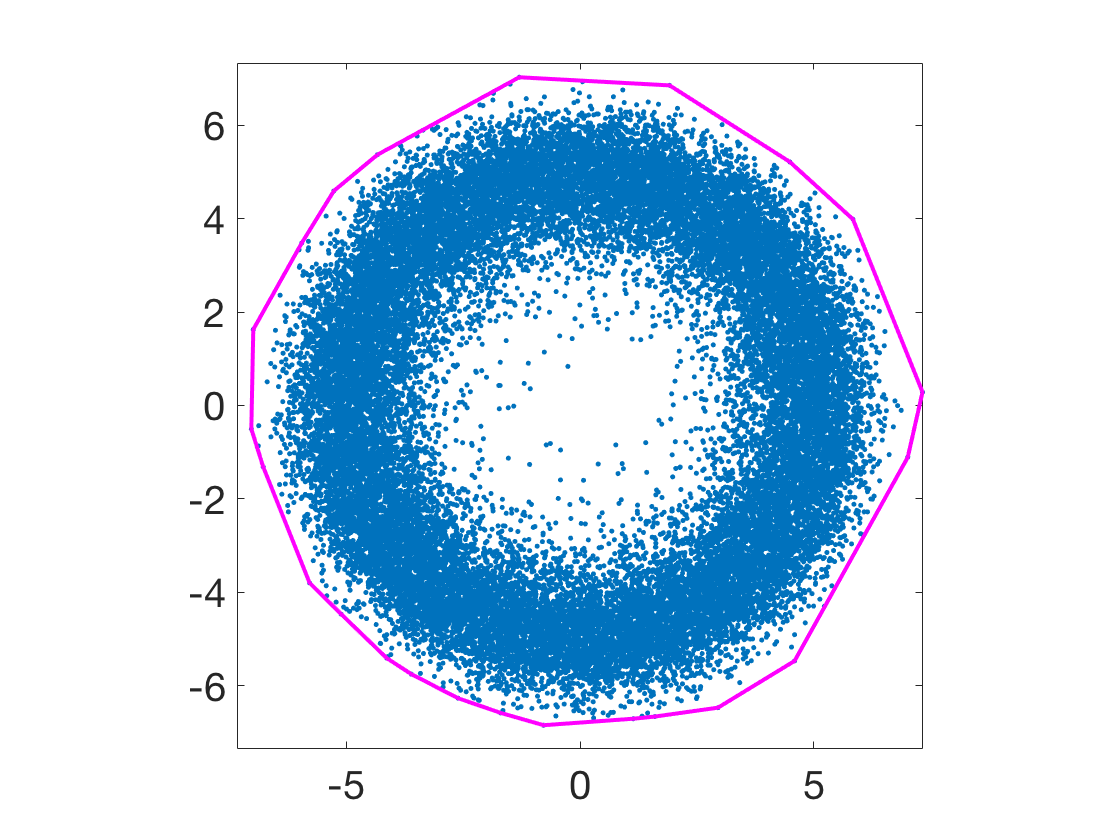}
\includegraphics[width = 0.35\textwidth,clip,trim=6cm 1cm 6cm 2cm]{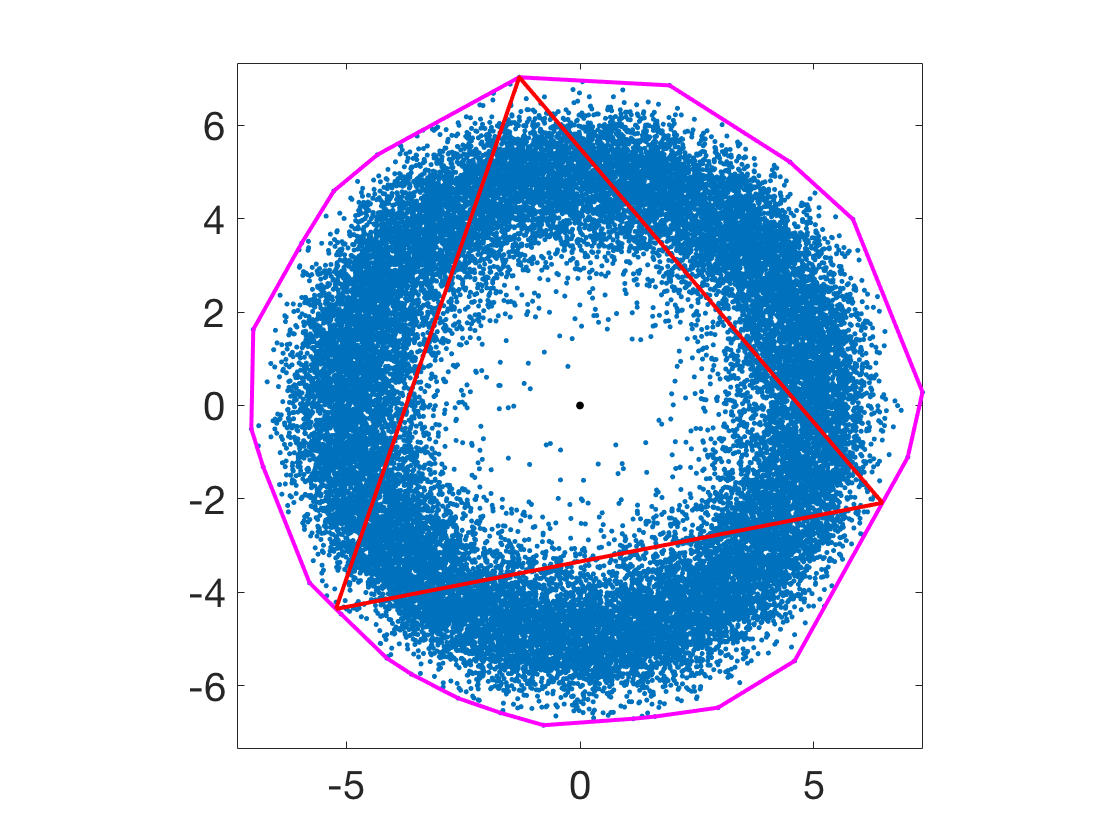}
\caption{{\bf (left)} $N=30,000$ random points generated from an annular distribution as in \eqref{eq:bimodalrandom}. The boundary of the convex hull is indicated. {\bf (right)} The optimal solution of \eqref{e:arch3} with $\alpha =0$. See \cref{sec:unbounded2}.} 
\label{fig:bimodal1}
\end{figure}

In \cref{fig:bimodal2}, we plot the area of $\co(A)$ as we vary $\alpha$. We observe that even for  random points generated from this annular distribution, as $\alpha \rightarrow 0$, $\co(A)$ shrinks to a regular triangle containing the mean of the distribution. 

\begin{figure}[ht]
\centering
\includegraphics[width = 0.75\textwidth,clip,trim = 10cm 0cm 8cm 0cm]{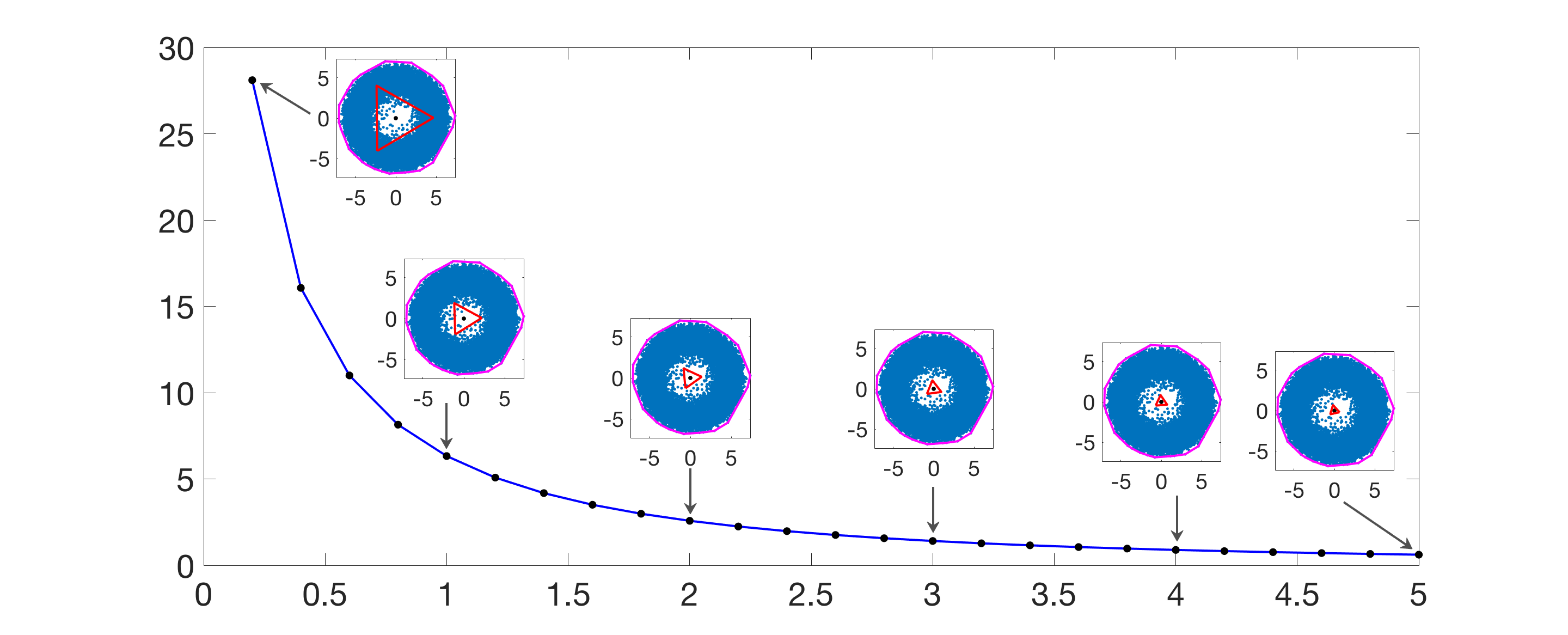}
\caption{The area of $\co(A)$ as we vary  $\alpha$ for $N = 30,000$ random points generated from an annular distribution.  See \cref{sec:unbounded2}.} \label{fig:bimodal2}
\end{figure}

To study how the algorithm depends on the initialization, we fix $\alpha = 2$ and choose several different initializations for $\{a_\ell\}_{\ell \in [3]}$. \cref{fig:bimodal3} displays selected snapshots of the iterations for different initializations.  The final  $\co(A)$ is more or less independent of the initialization, which indicates the robustness of the algorithm.
 
\begin{figure}[ht]
\centering
\includegraphics[width = 0.12\textwidth,clip,trim=6cm 0cm 8cm 0cm]{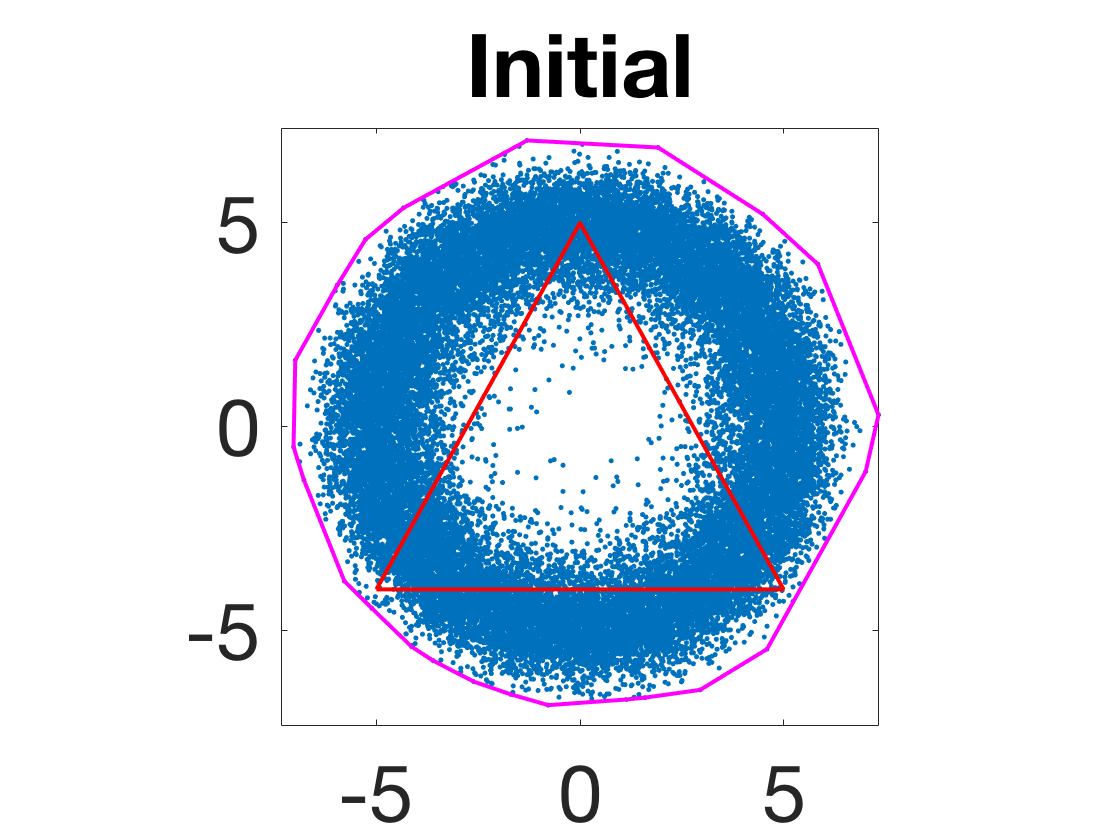}
\includegraphics[width = 0.12\textwidth,clip,trim=6cm 0cm 8cm 0cm]{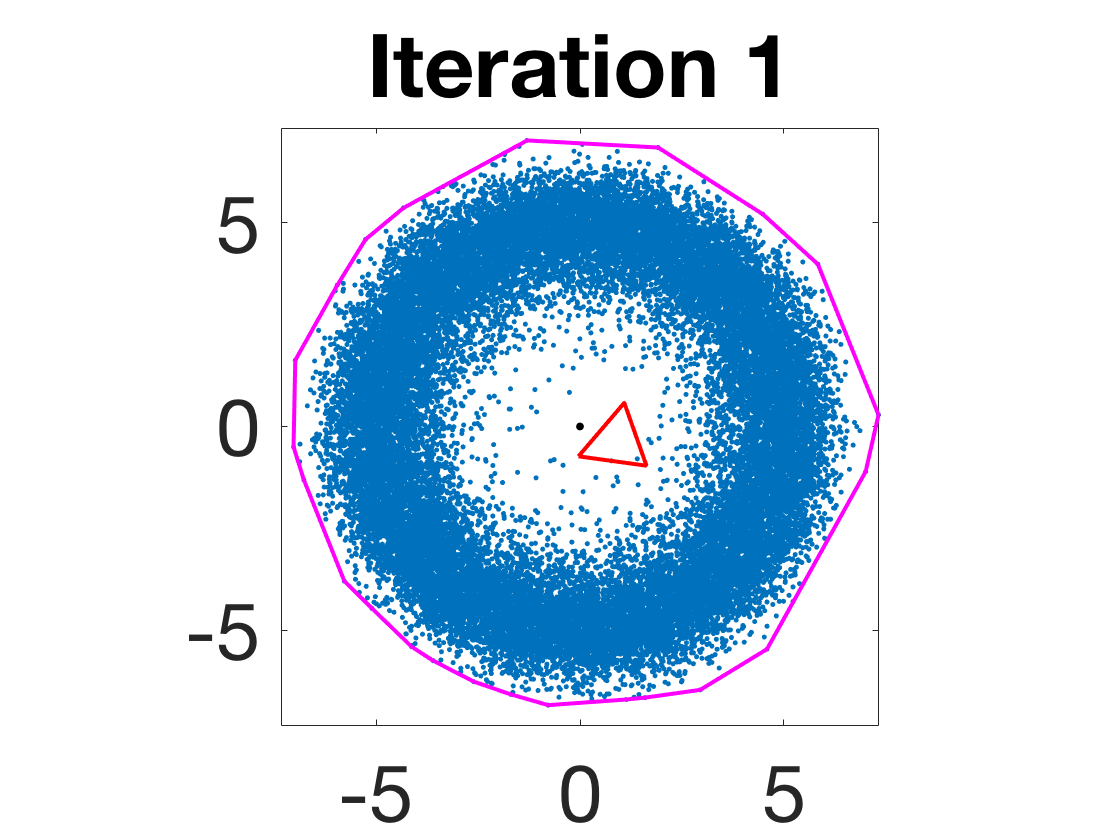}
\includegraphics[width = 0.12\textwidth,clip,trim=6cm 0cm 8cm 0cm]{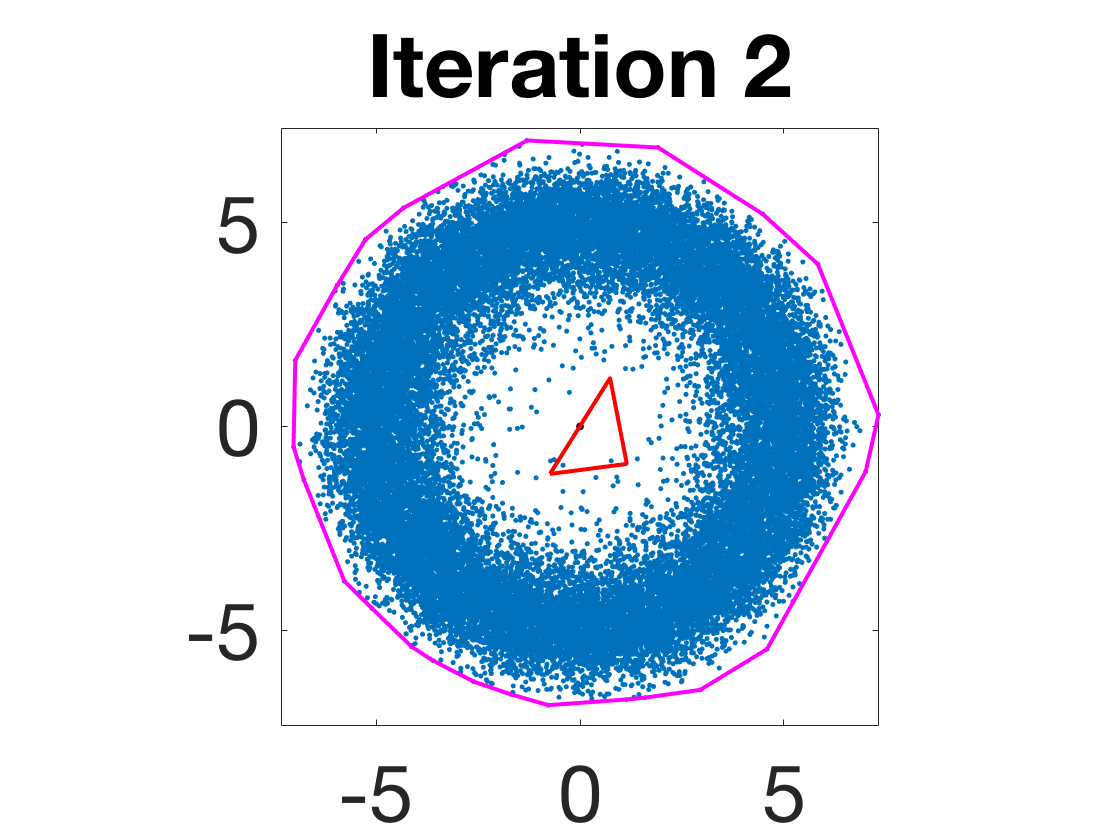}
\includegraphics[width = 0.12\textwidth,clip,trim=6cm 0cm 8cm 0cm]{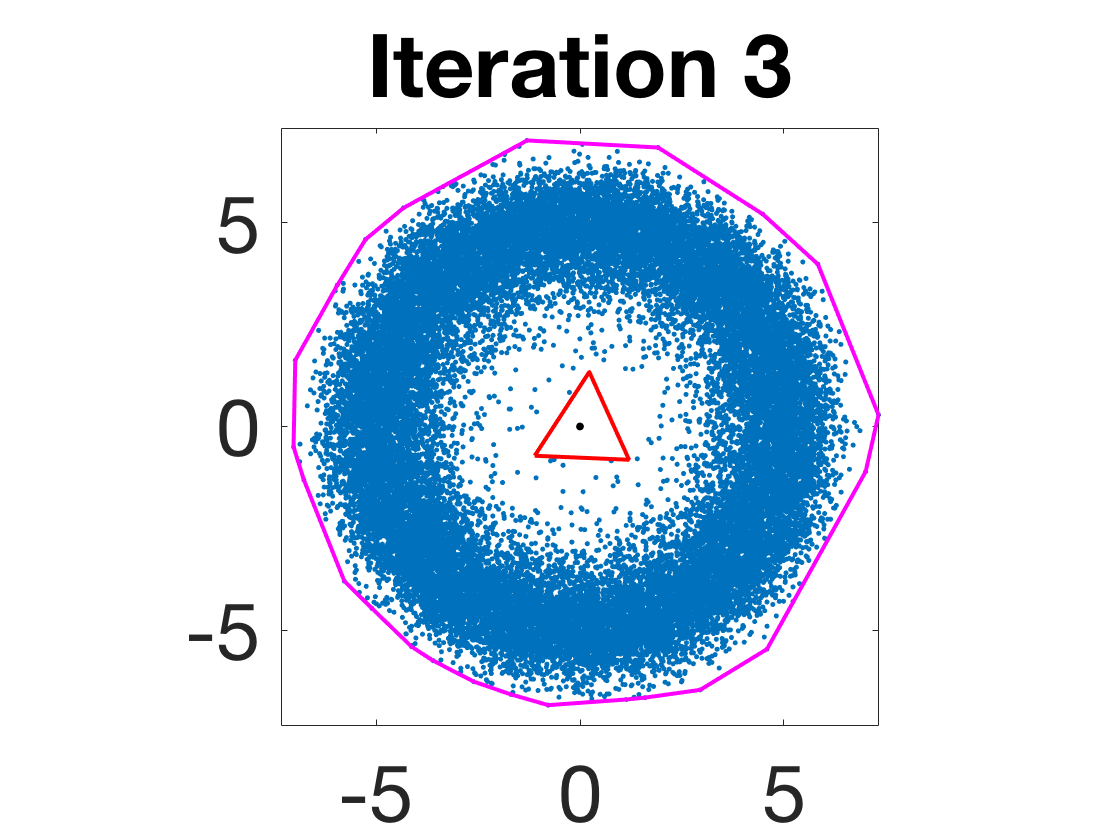}
\includegraphics[width = 0.12\textwidth,clip,trim=6cm 0cm 8cm 0cm]{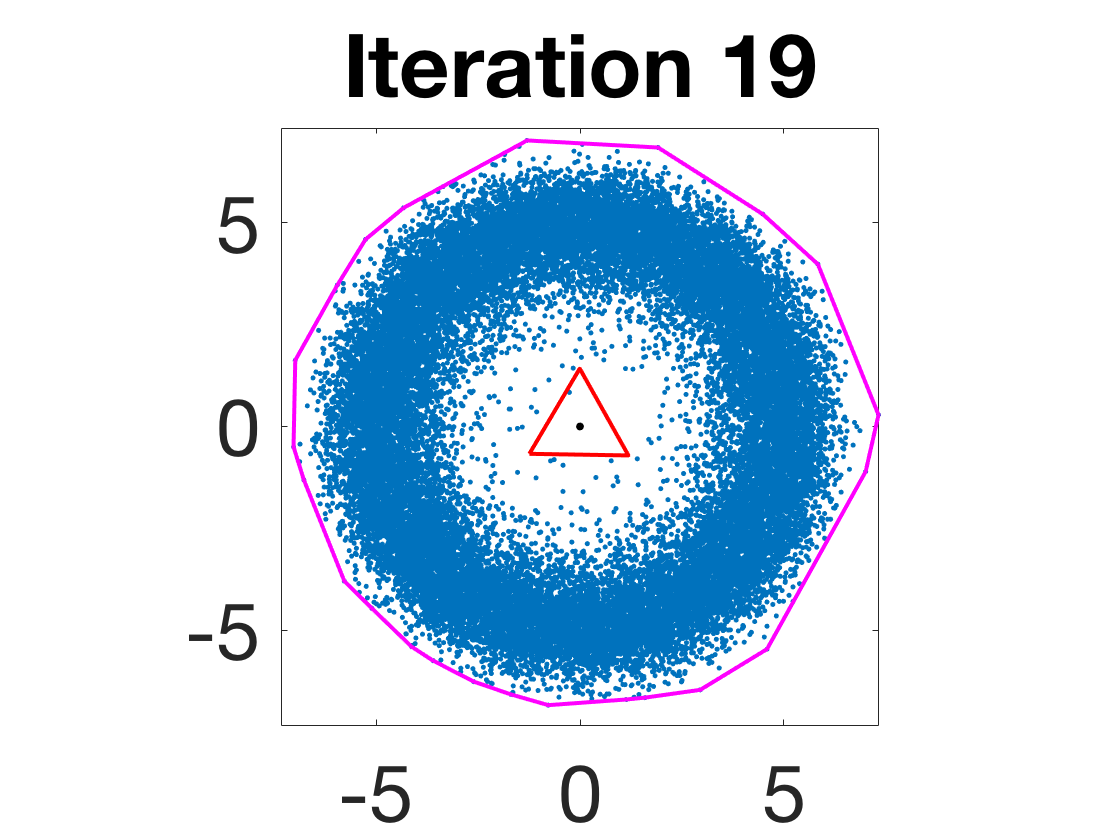}
\includegraphics[width = 0.12\textwidth,clip,trim=6cm 0cm 8cm 0cm]{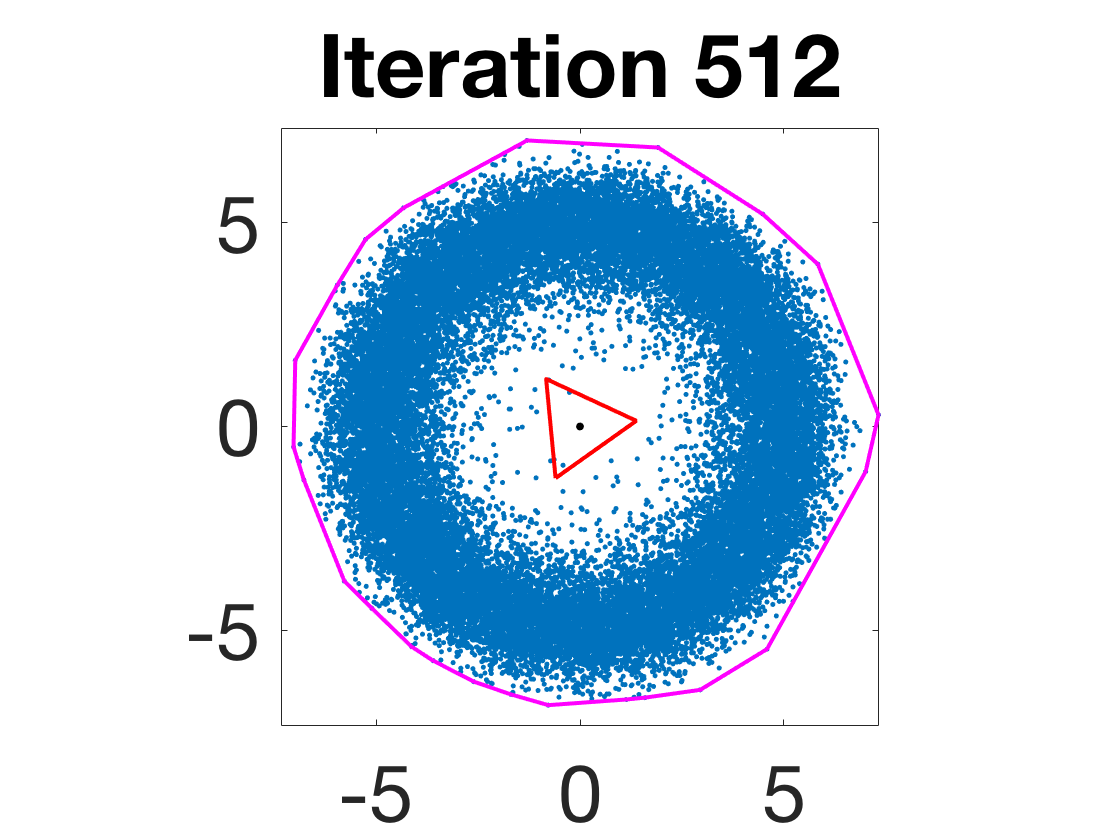}\\
\includegraphics[width = 0.12\textwidth,clip,trim=6cm 0cm 8cm 0cm]{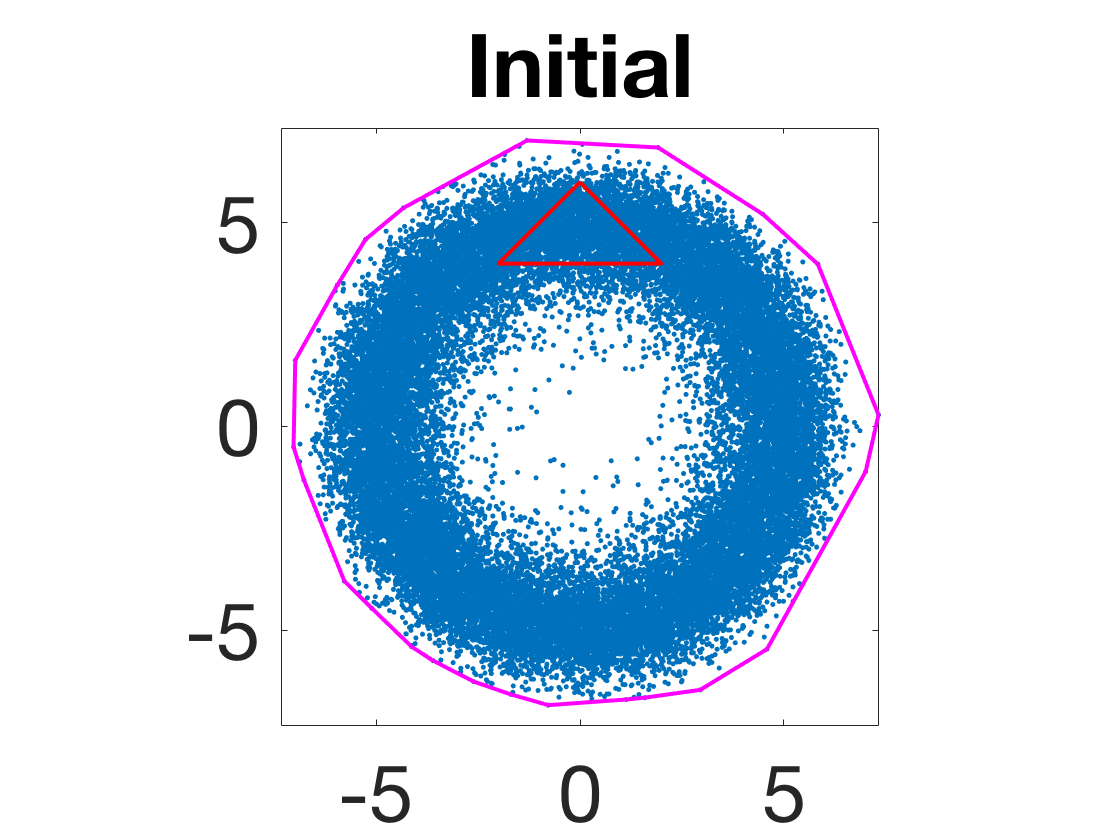}
\includegraphics[width = 0.12\textwidth,clip,trim=6cm 0cm 8cm 0cm]{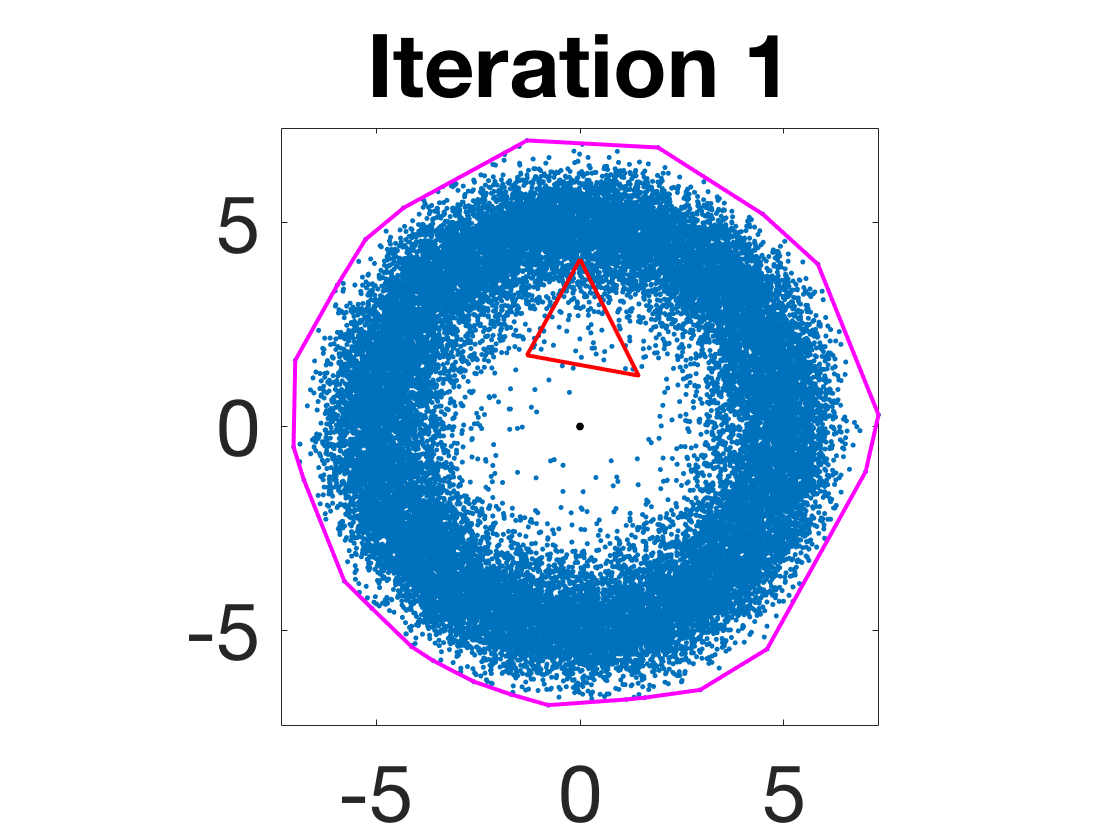}
\includegraphics[width = 0.12\textwidth,clip,trim=6cm 0cm 8cm 0cm]{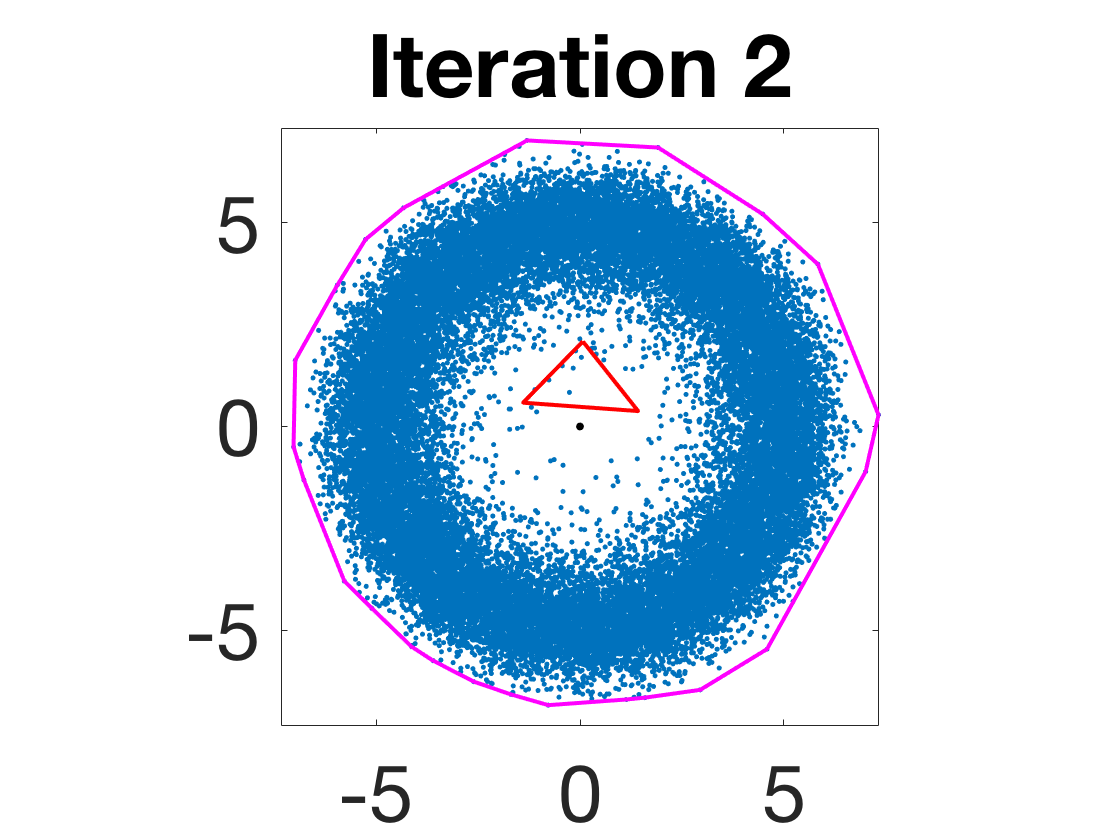}
\includegraphics[width = 0.12\textwidth,clip,trim=6cm 0cm 8cm 0cm]{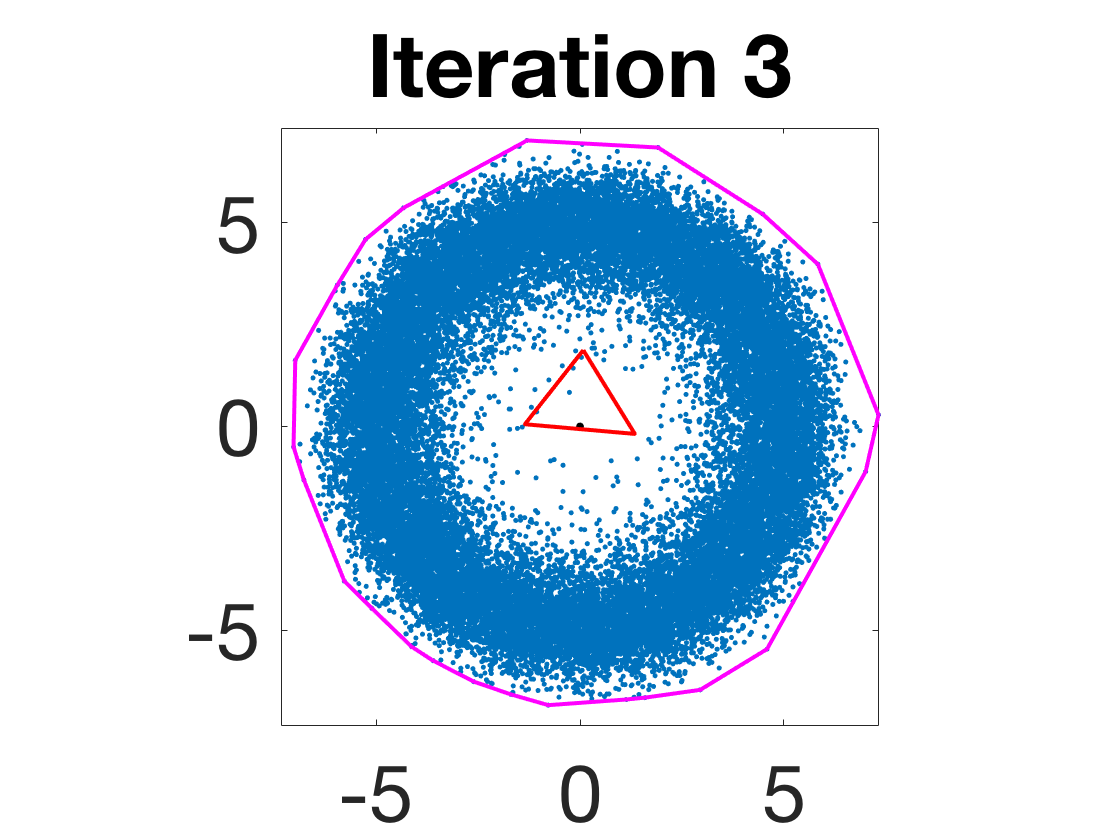}
\includegraphics[width = 0.12\textwidth,clip,trim=6cm 0cm 8cm 0cm]{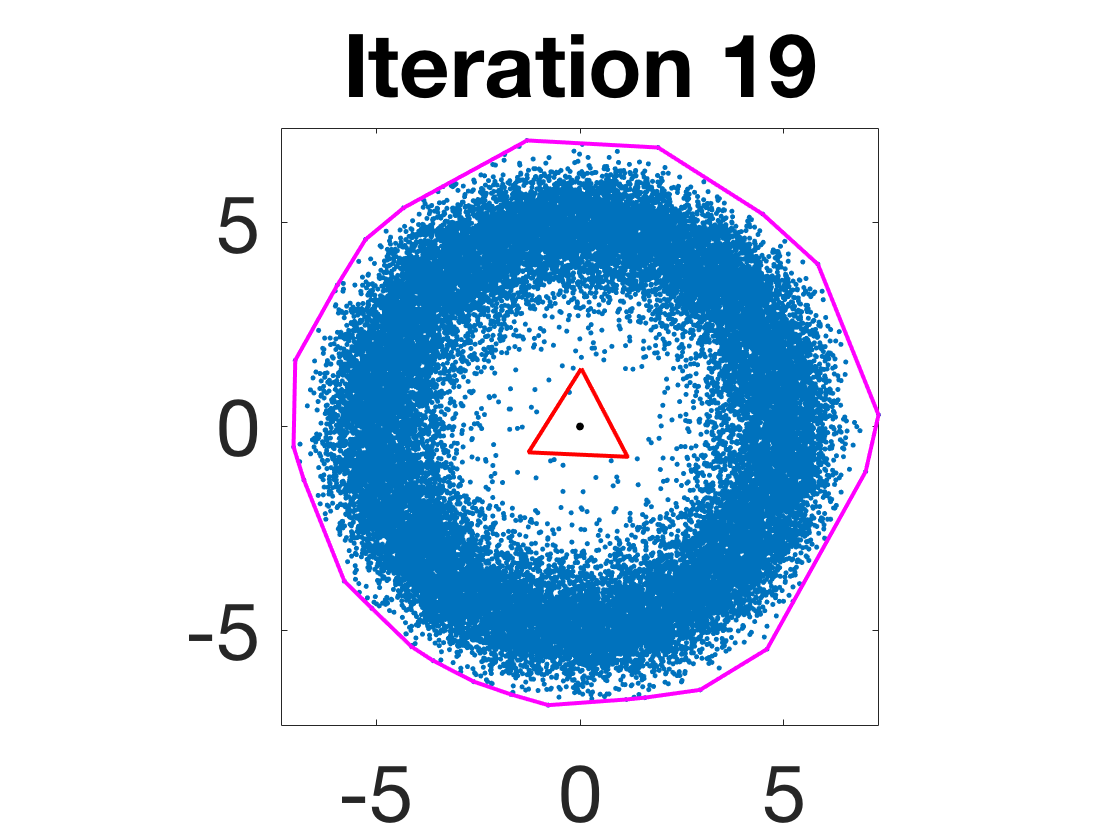}
\includegraphics[width = 0.12\textwidth,clip,trim=6cm 0cm 8cm 0cm]{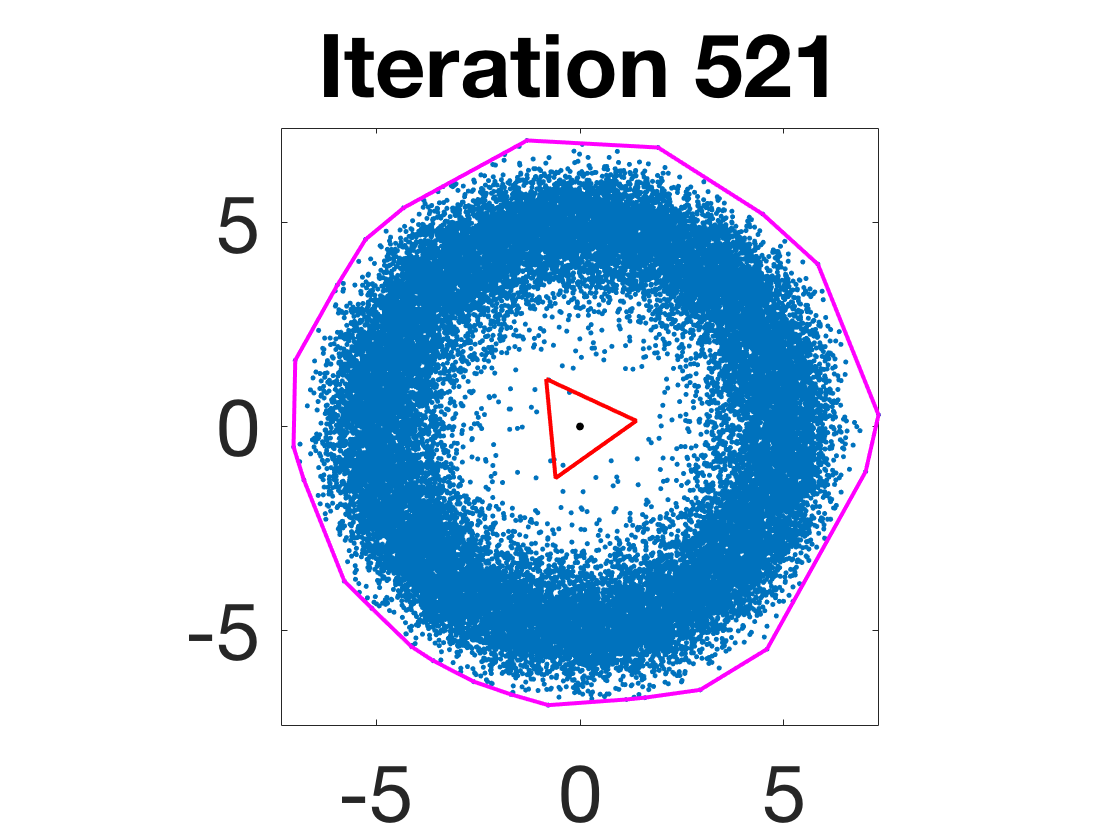}\\
\includegraphics[width = 0.12\textwidth,clip,trim=6cm 0cm 8cm 0cm]{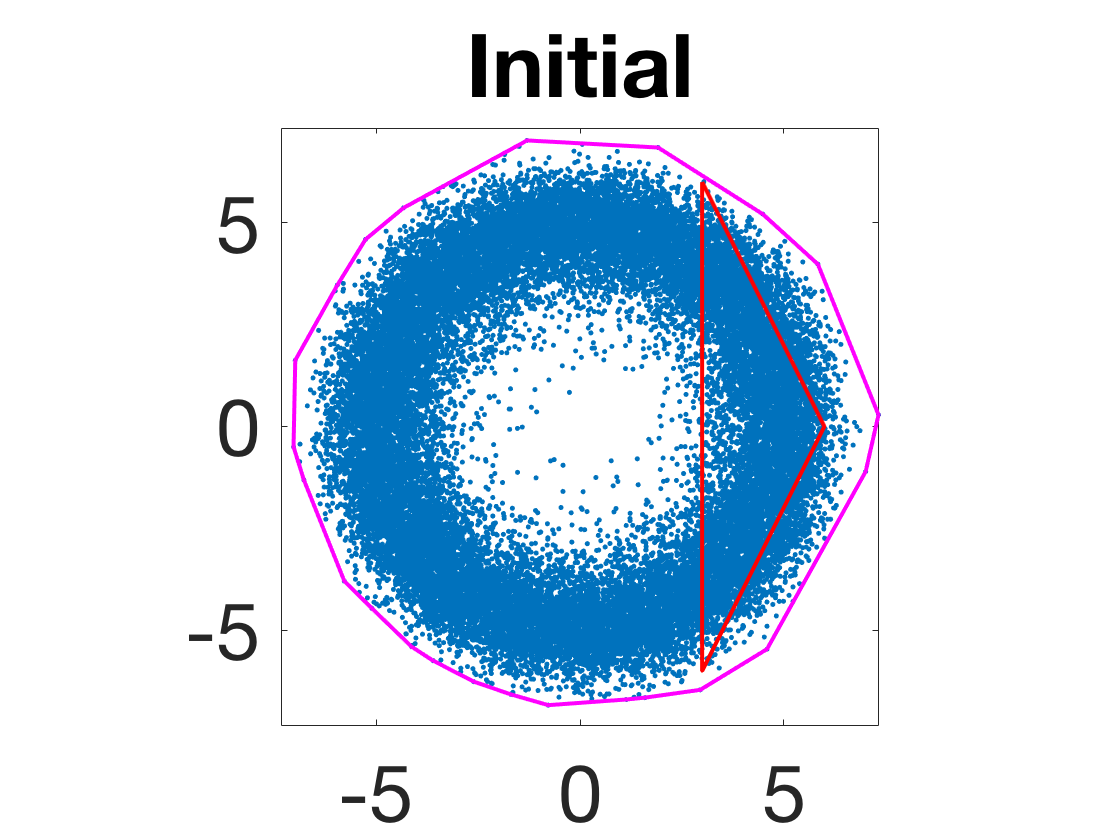}
\includegraphics[width = 0.12\textwidth,clip,trim=6cm 0cm 8cm 0cm]{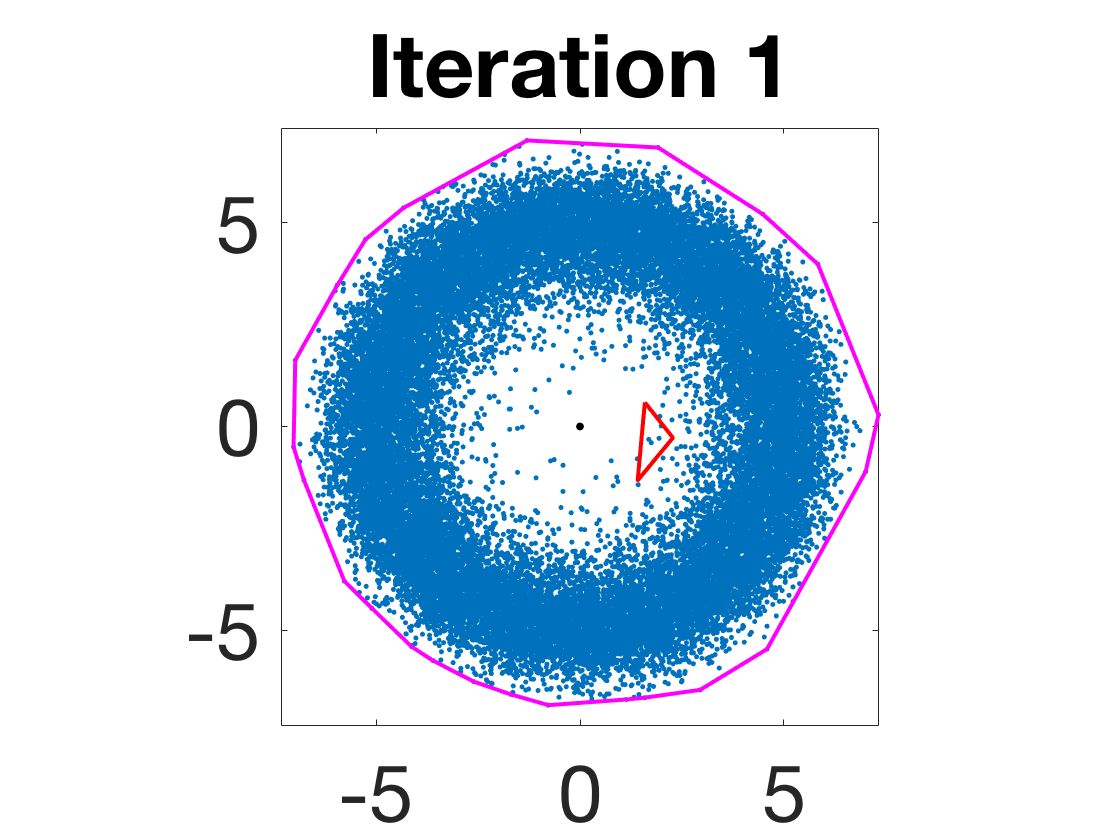}
\includegraphics[width = 0.12\textwidth,clip,trim=6cm 0cm 8cm 0cm]{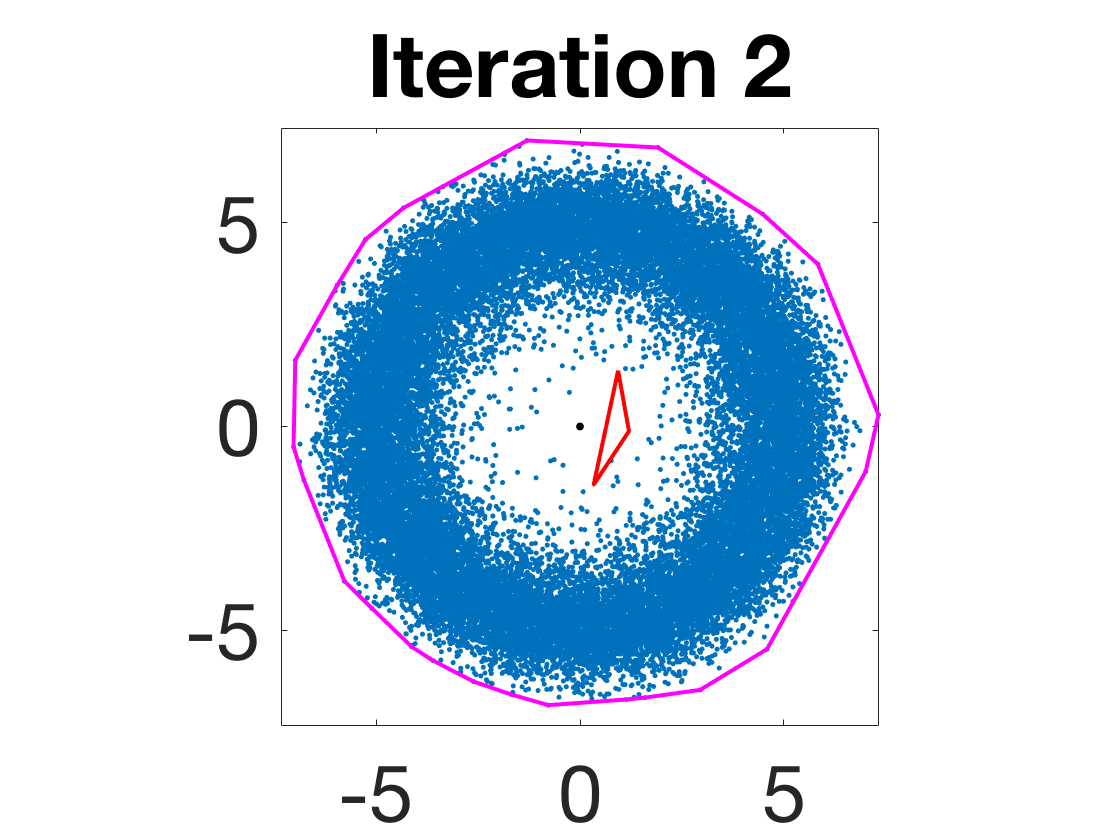}
\includegraphics[width = 0.12\textwidth,clip,trim=6cm 0cm 8cm 0cm]{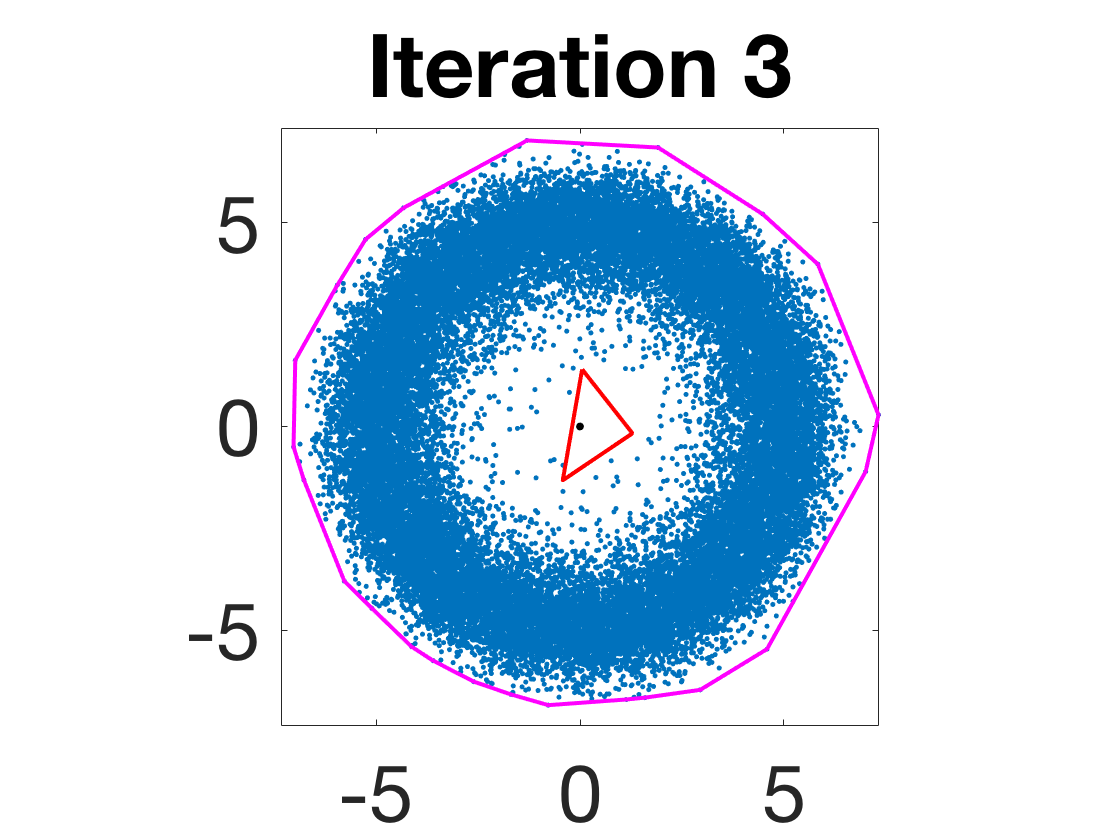}
\includegraphics[width = 0.12\textwidth,clip,trim=6cm 0cm 8cm 0cm]{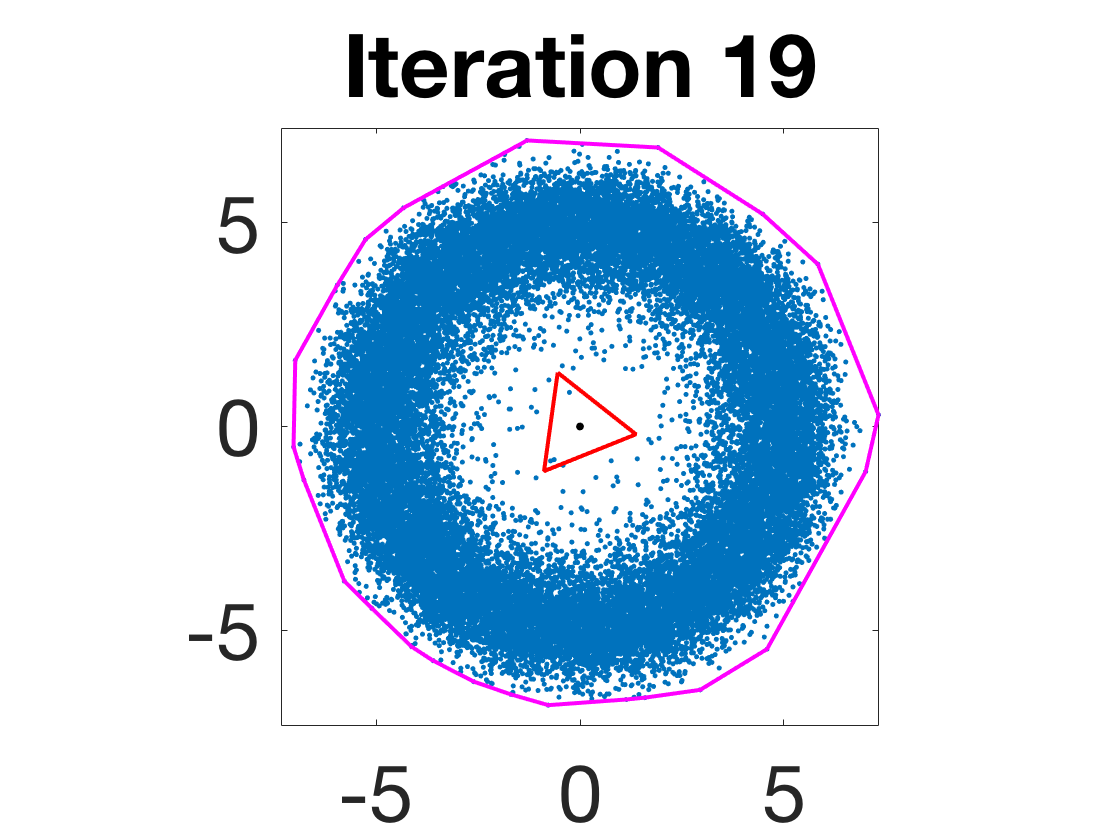}
\includegraphics[width = 0.12\textwidth,clip,trim=6cm 0cm 8cm 0cm]{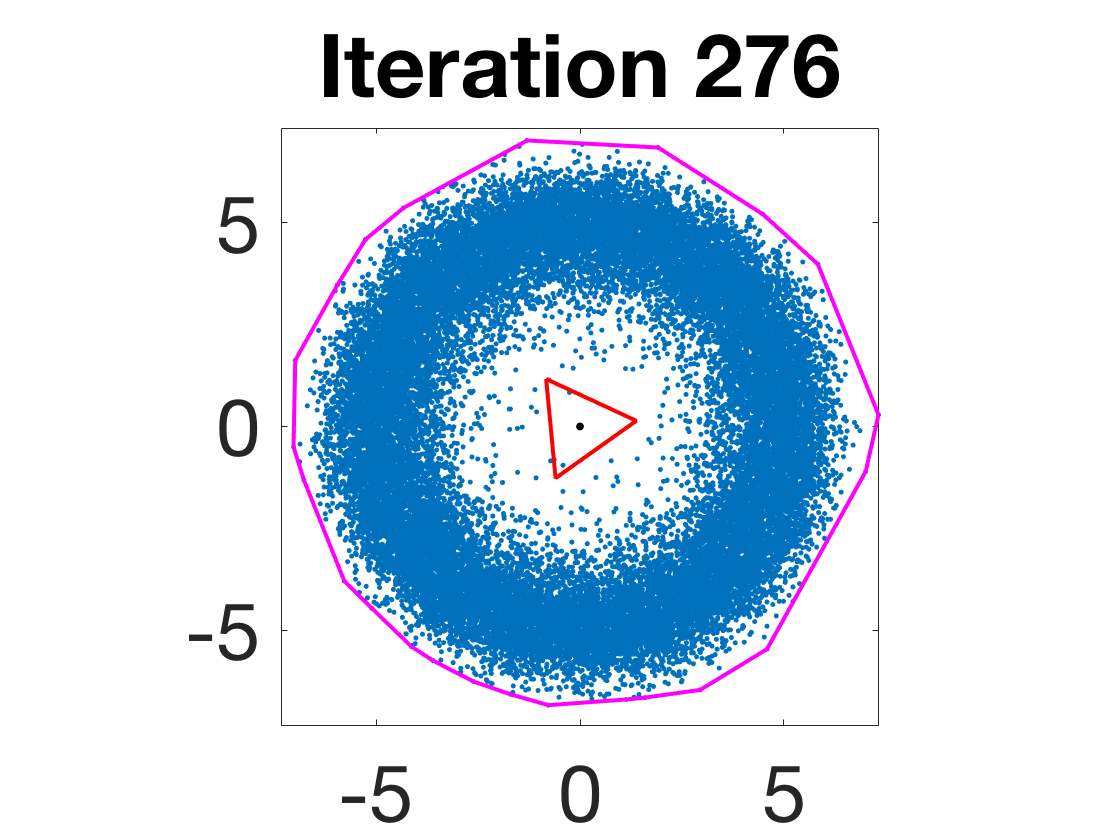}
\caption{Snapshots of the iterates for different initializations. Last column is the final iterate. See \cref{sec:unbounded2}.} \label{fig:bimodal3}
\end{figure}

\subsubsection{Example 4: A Gaussian mixture model} \label{sec:unbounded3} 
In this experiment, we consider an asymmetric distribution where $N = 30,000$ random data points are generated from the mixed Gaussian model with two normal distributions $\mathcal N(\mu_1,\Sigma_1)$ and $\mathcal N(\mu_2,\Sigma_2)$ with $\mu_1 = \begin{pmatrix} 8 \\ 8  \end{pmatrix}$, $\mu_2 = \begin{pmatrix}-8 \\ -8 \end{pmatrix}$, $\Sigma_1 = 9 \cdot I $, and varying  $\Sigma_2$. We note that the mean of this bimodal distribution is $\begin{pmatrix} 0 \\ 0  \end{pmatrix}$. In this experiment, $\alpha =2$.
In the first two rows of \cref{fig:bimodal4}, we display  snapshots of iterations for the data randomly generated from the Gaussian mixture model with $\Sigma_2 = 4 \cdot I $ for two different initializations. We observe that the iterates converge to a triangle containing the data mean in just a few iterations. In the last two rows of \cref{fig:bimodal4}, we change $\Sigma_2 $ to $1 \cdot I $ and $0.25 \cdot I $  to increase the ``anisotropy''. After about 20 iterations, the solution also converges to a triangle containing the mean $\begin{pmatrix} 0 \\ 0 \end{pmatrix}$. In this experiment, we note that, because the ``anisotropy'' of the random data, $\co(A)$ is not a regular triangle, but contains the mean. This numerically supports the claim \cref{t:ConsVarReg-alpha} that penalization makes the archetype points tend towards the mean.

\begin{figure}[ht!]
\centering
\includegraphics[width = 0.12\textwidth,clip,trim=2.5cm 0cm 3.5cm 0cm]{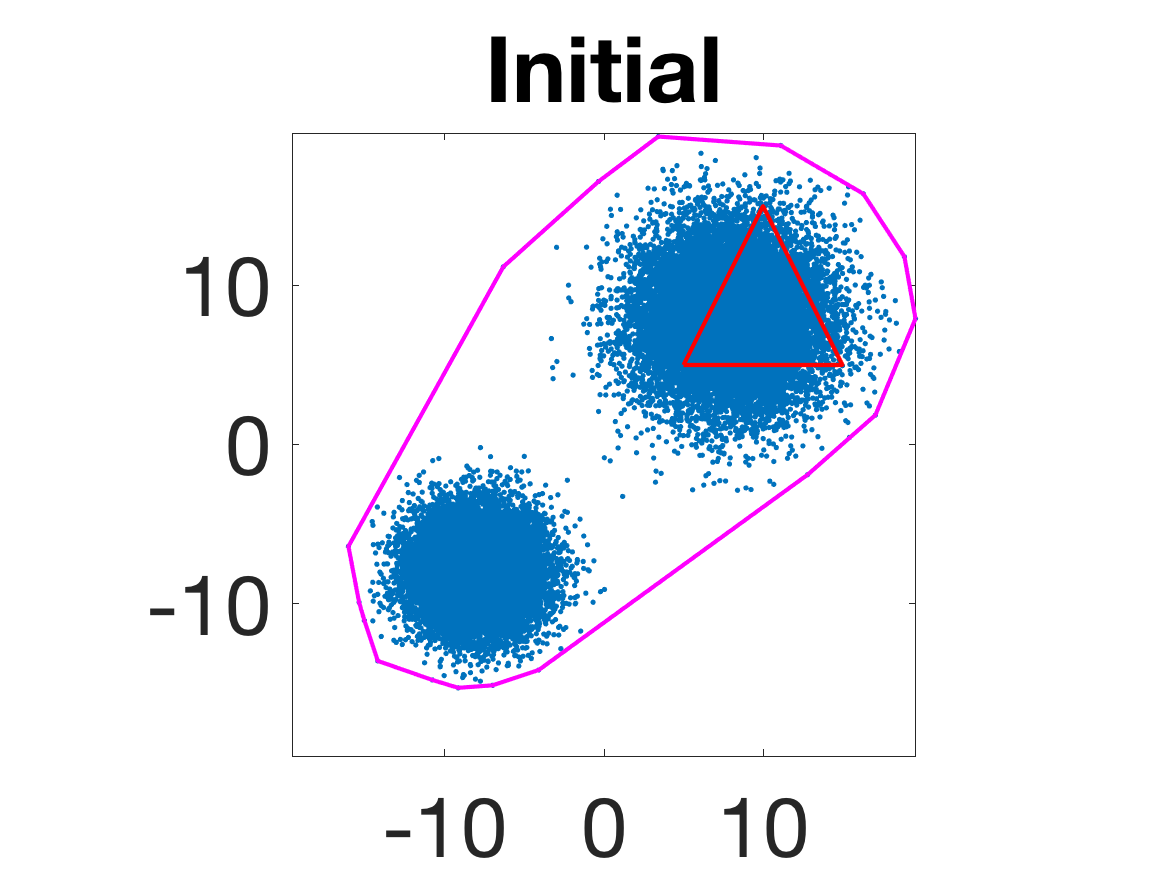}
\includegraphics[width = 0.12\textwidth,clip,trim=2.5cm 0cm 3.5cm 0cm]{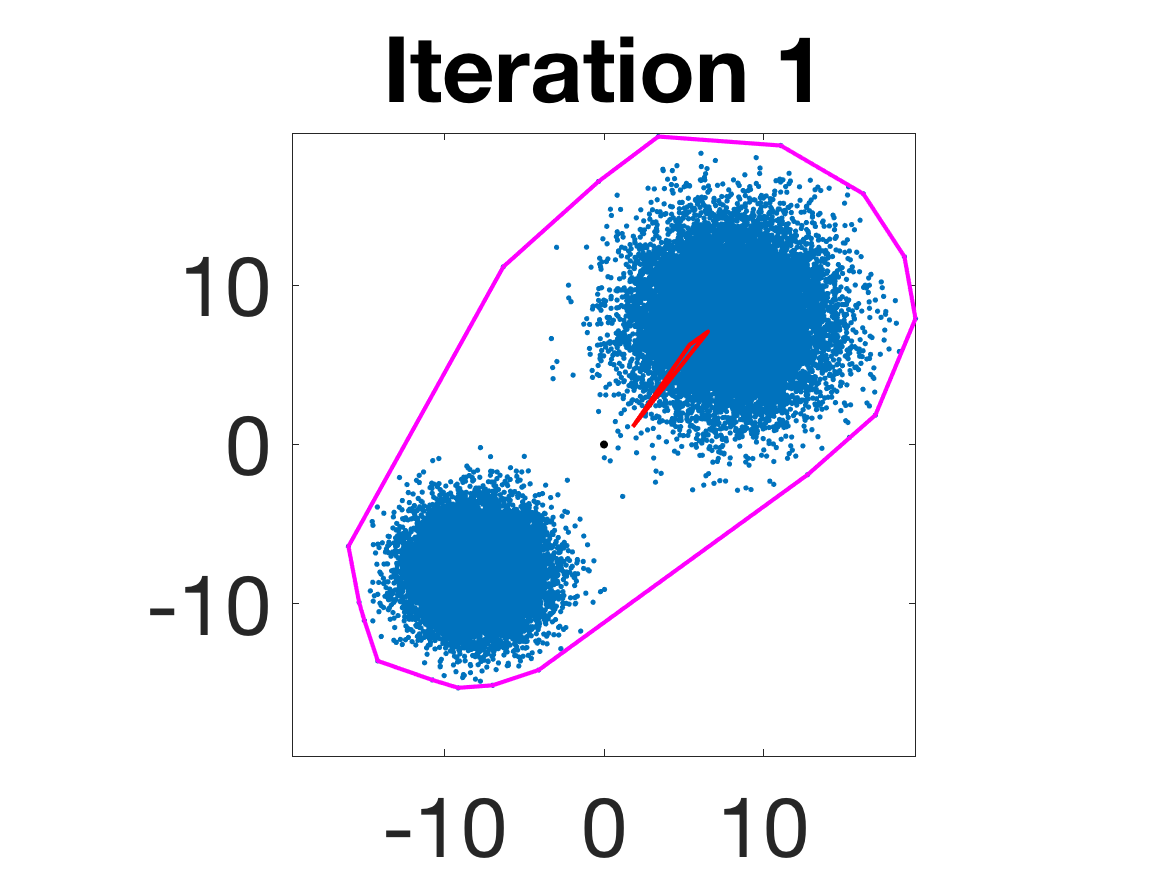}
\includegraphics[width = 0.12\textwidth,clip,trim=2.5cm 0cm 3.5cm 0cm]{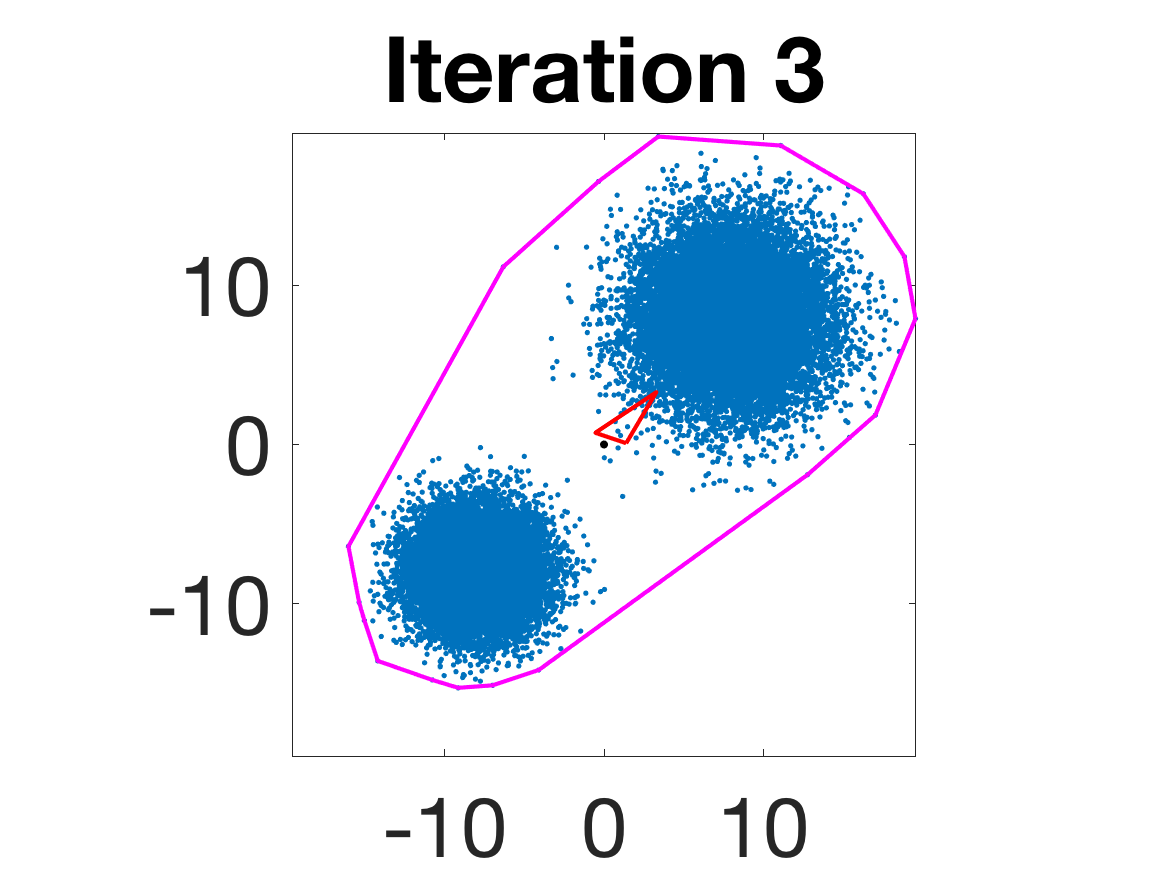}
\includegraphics[width = 0.12\textwidth,clip,trim=2.5cm 0cm 3.5cm 0cm]{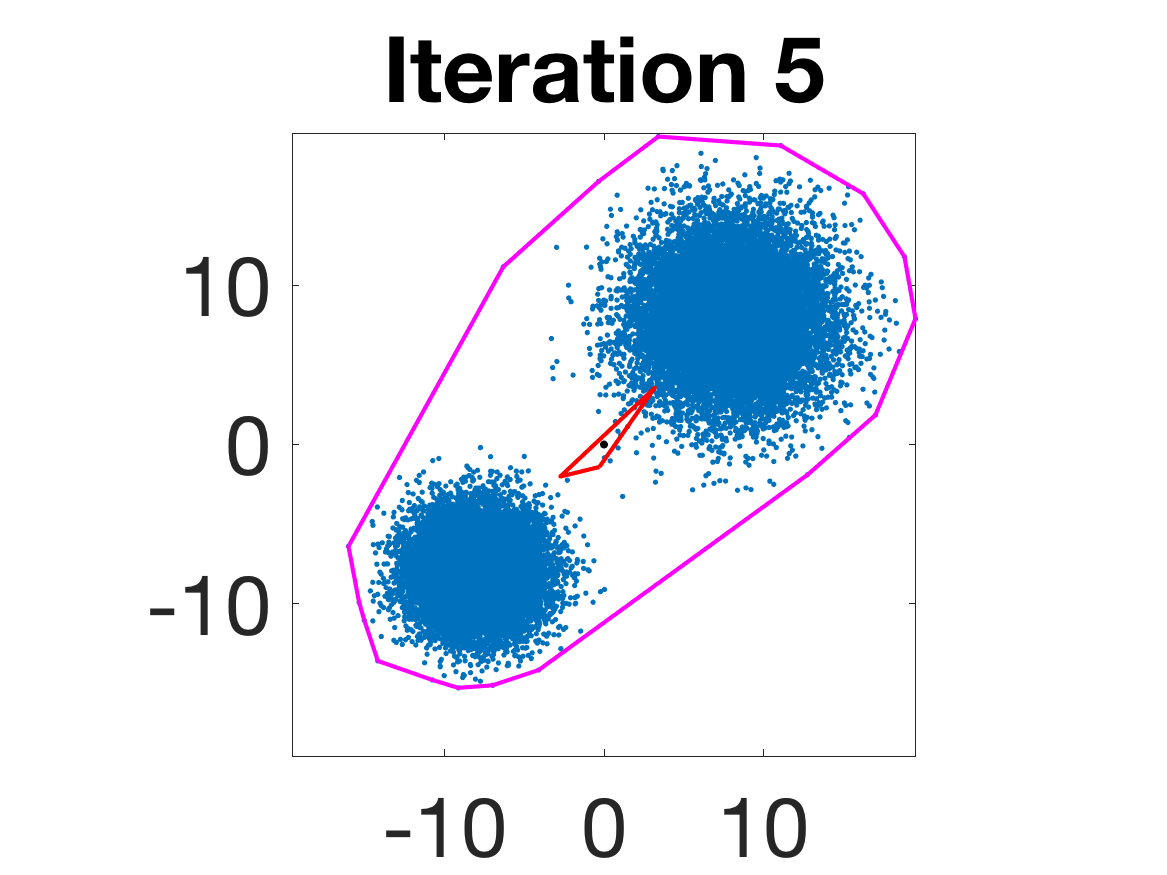}
\includegraphics[width = 0.12\textwidth,clip,trim=2.5cm 0cm 3.5cm 0cm]{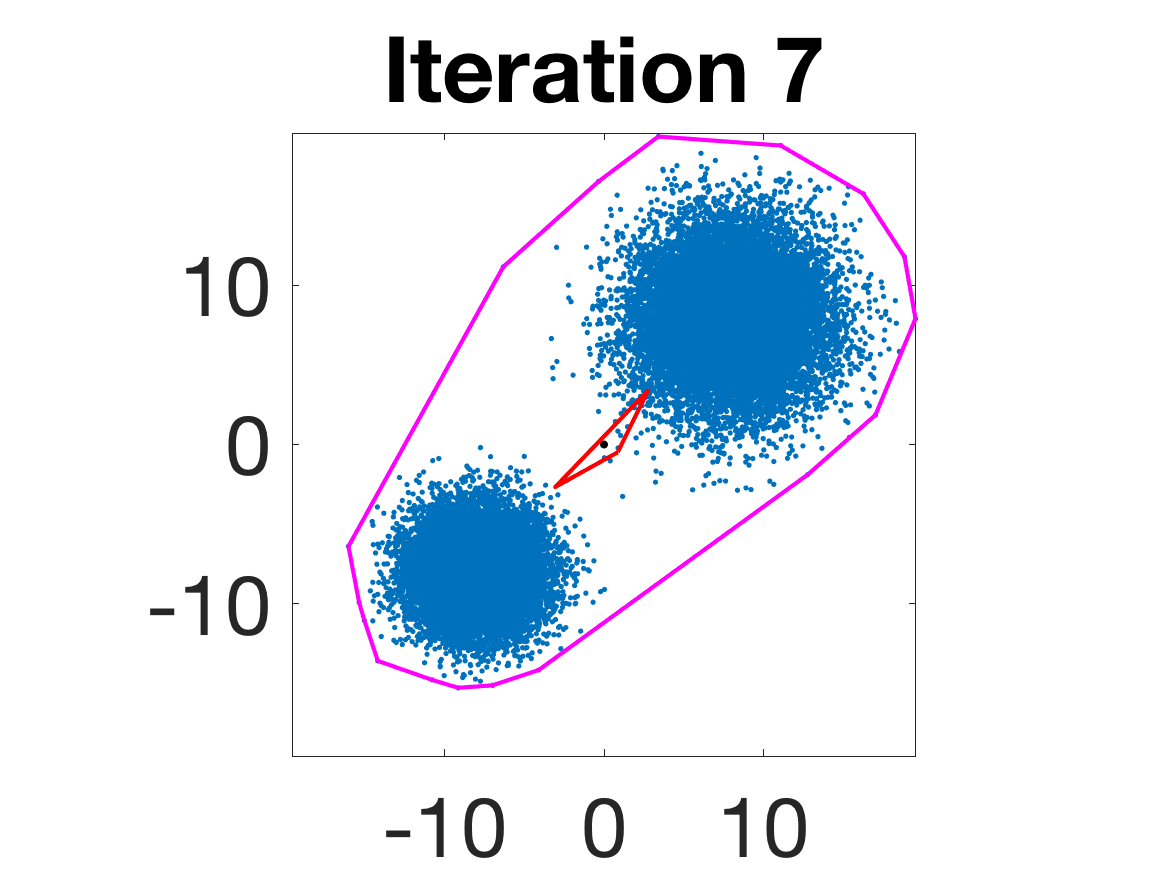}
\includegraphics[width = 0.12\textwidth,clip,trim=2.5cm 0cm 3.5cm 0cm]{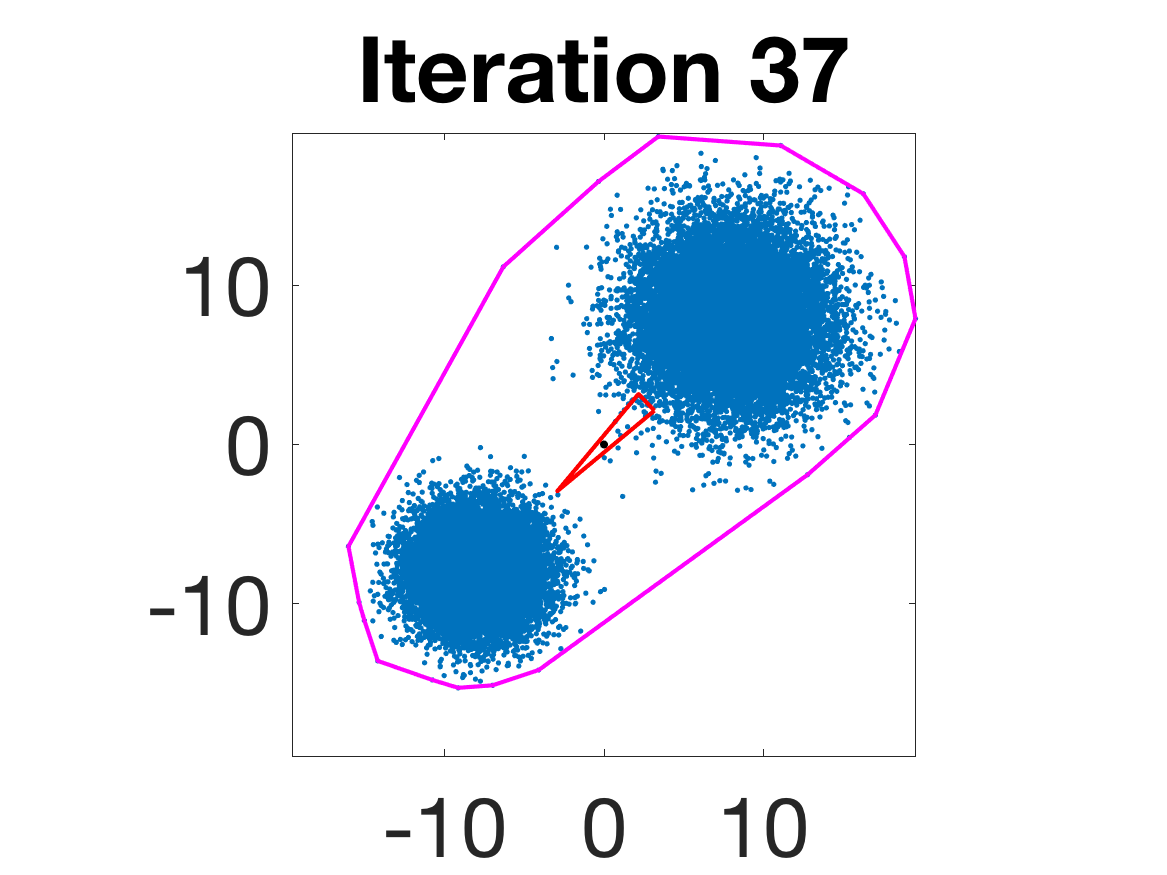}\\
\includegraphics[width = 0.12\textwidth,clip,trim=2.5cm 0cm 3.5cm 0cm]{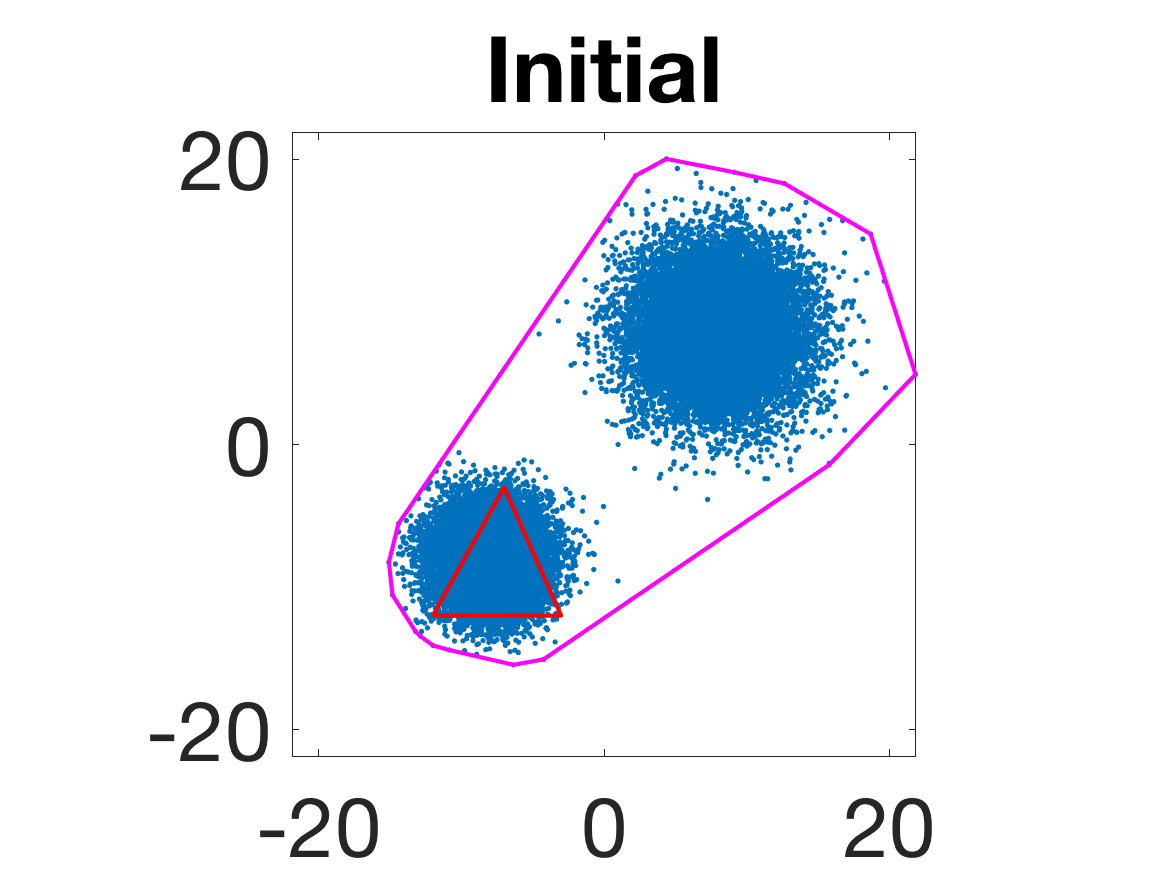}
\includegraphics[width = 0.12\textwidth,clip,trim=2.5cm 0cm 3.5cm 0cm]{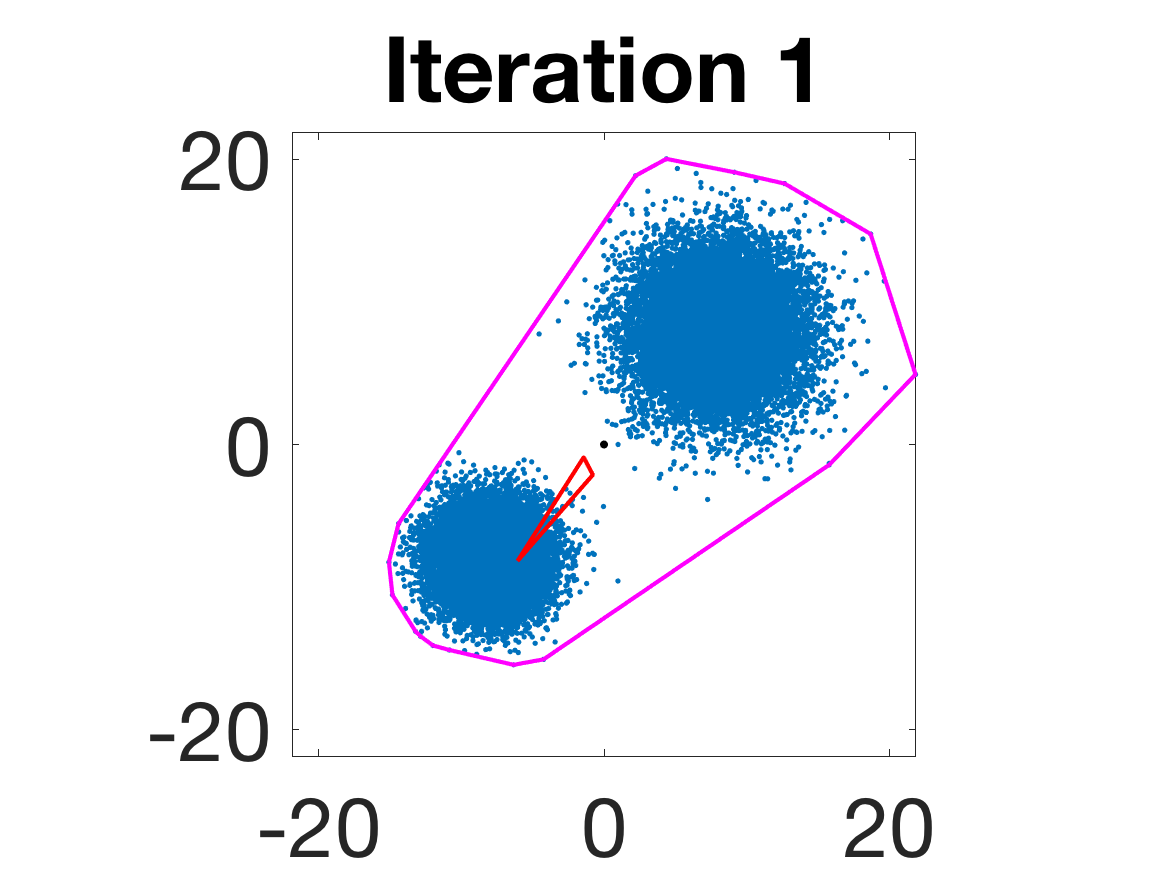}
\includegraphics[width = 0.12\textwidth,clip,trim=2.5cm 0cm 3.5cm 0cm]{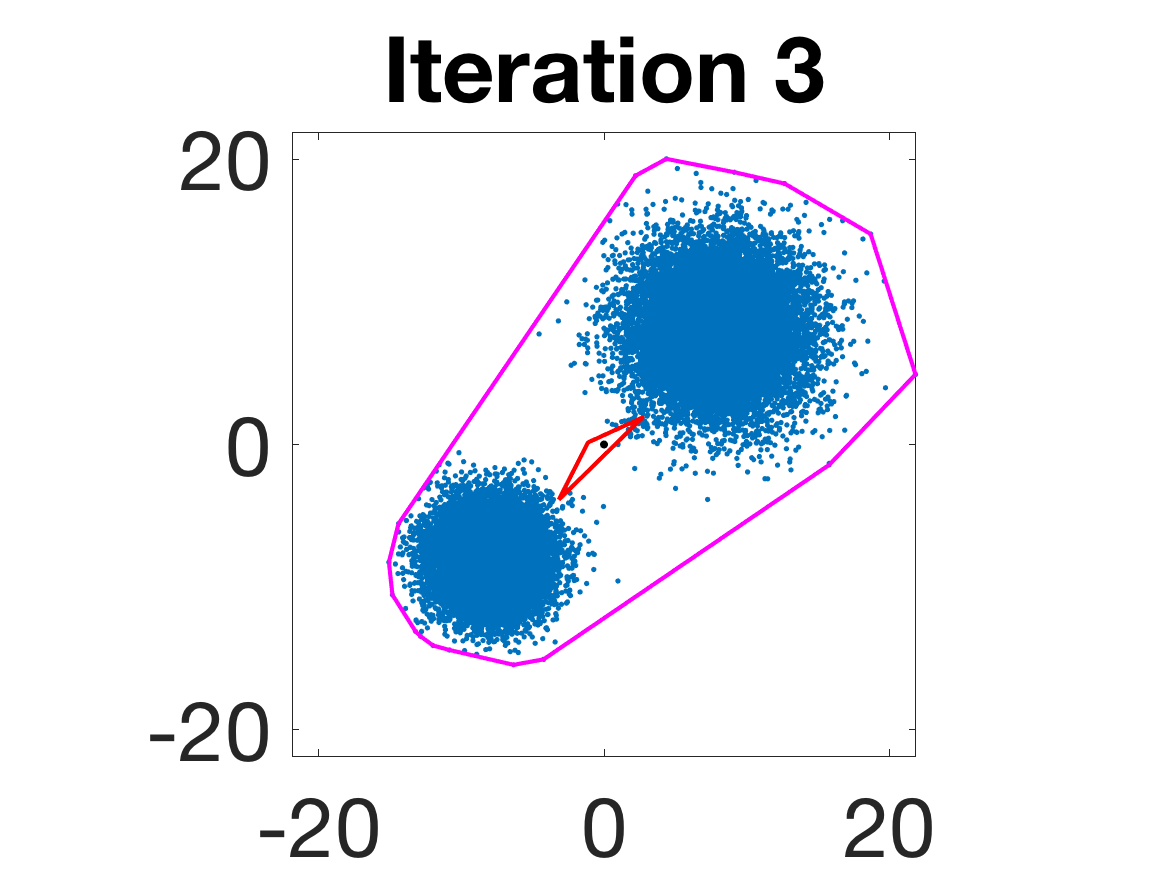}
\includegraphics[width = 0.12\textwidth,clip,trim=2.5cm 0cm 3.5cm 0cm]{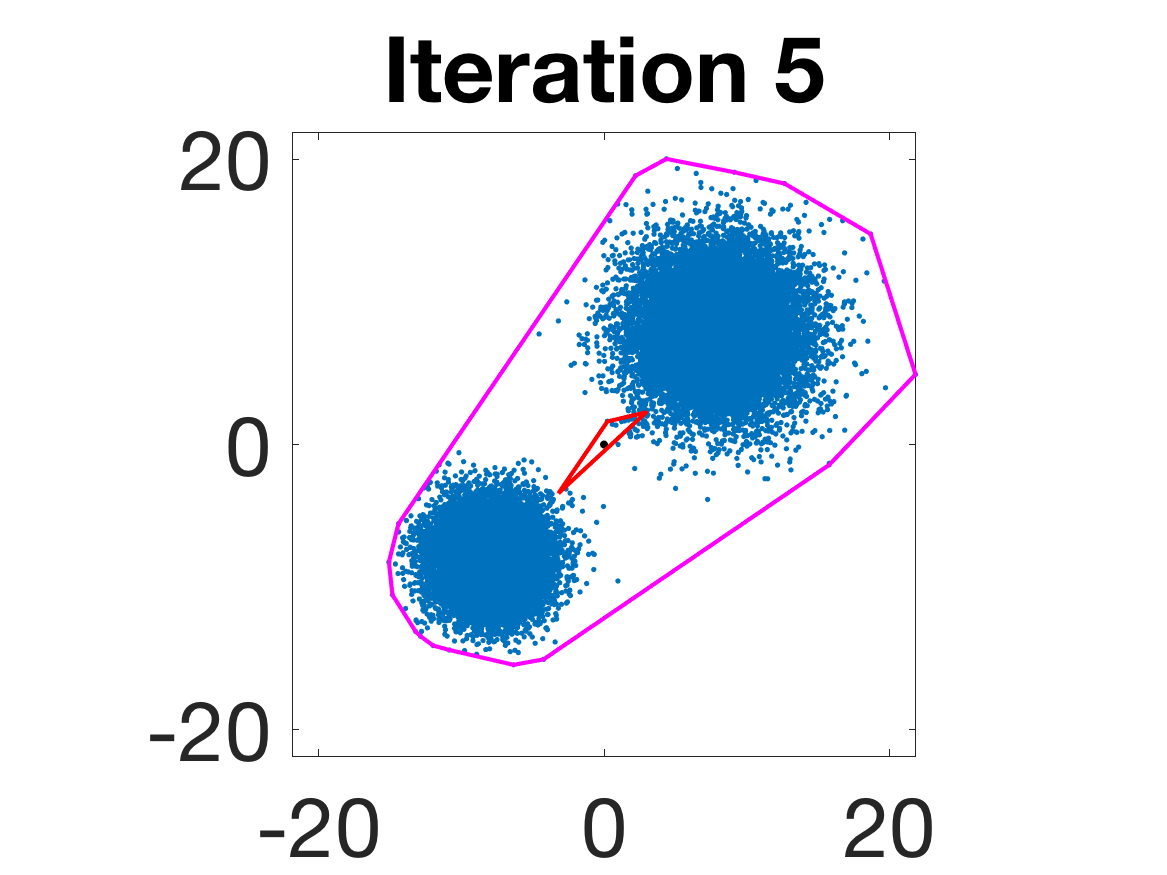}
\includegraphics[width = 0.12\textwidth,clip,trim=2.5cm 0cm 3.5cm 0cm]{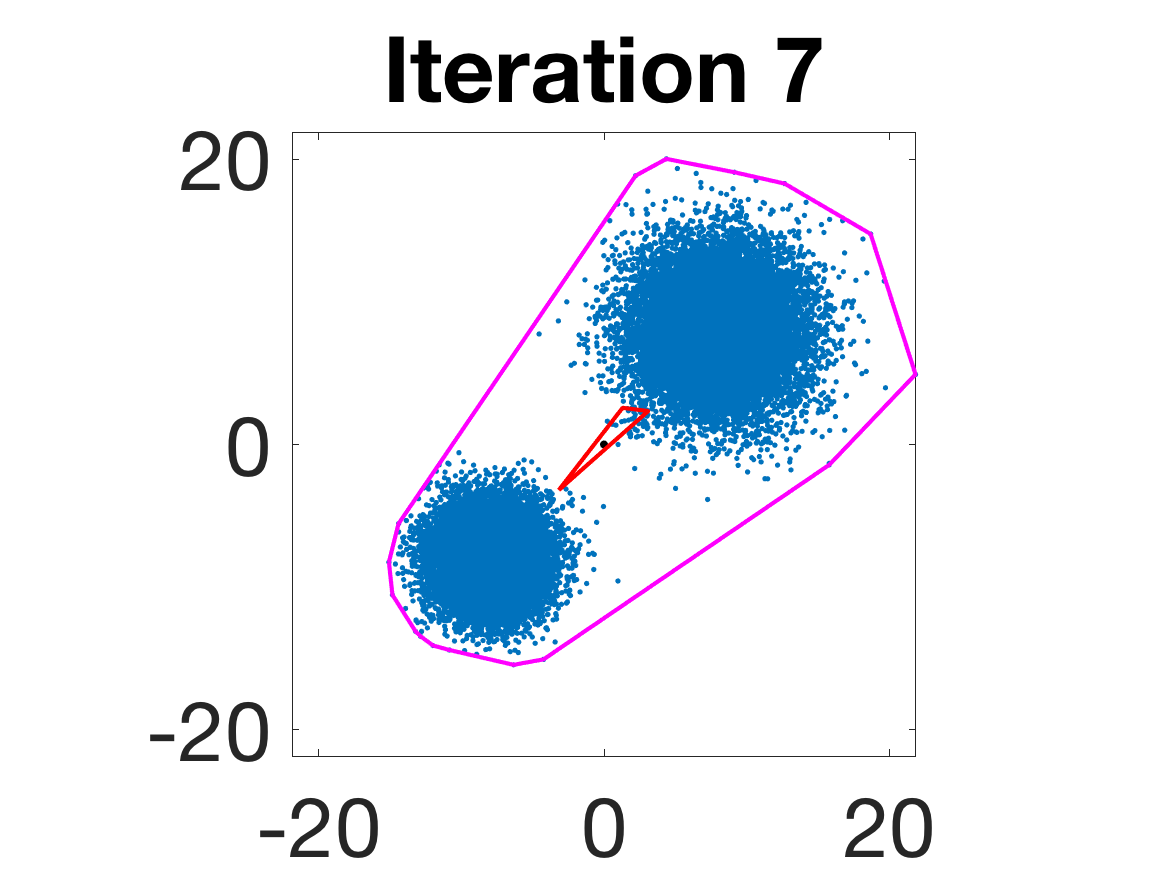}
\includegraphics[width = 0.12\textwidth,clip,trim=2.5cm 0cm 3.5cm 0cm]{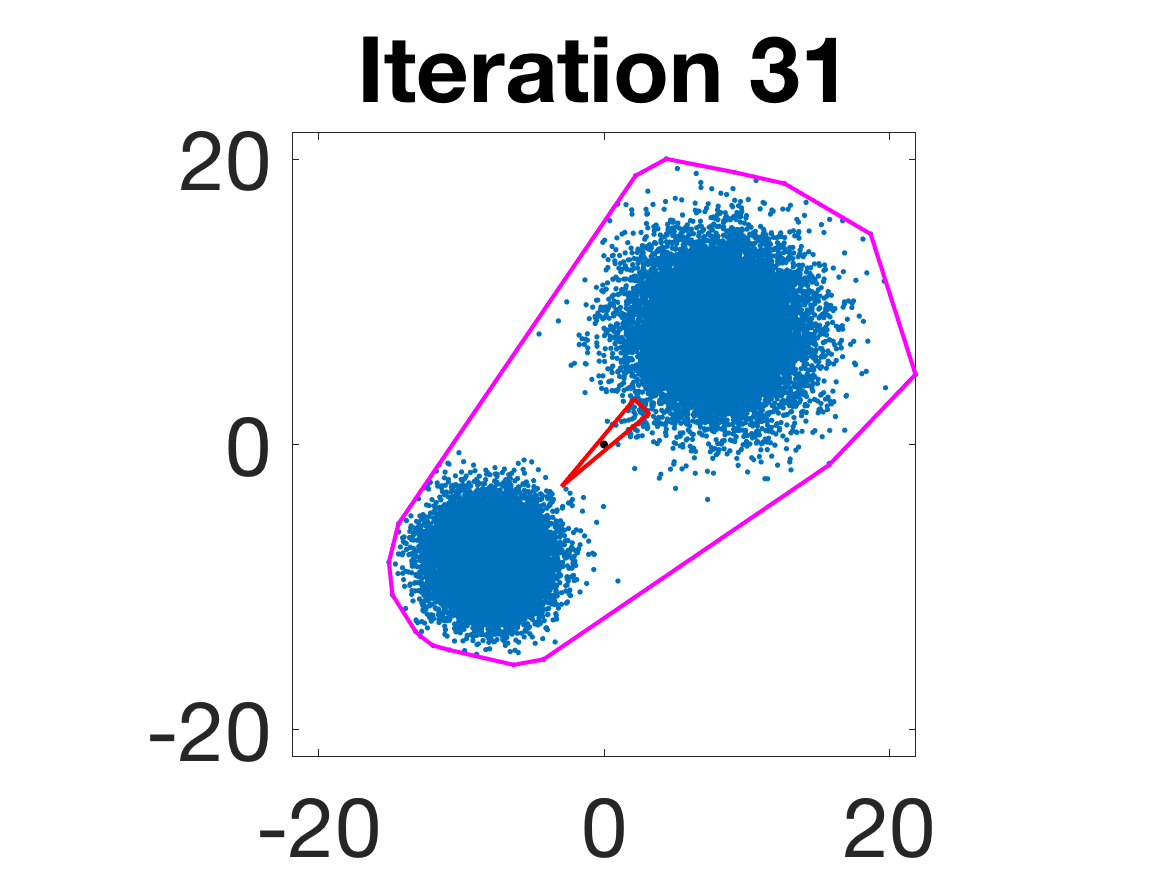}\\
\includegraphics[width = 0.12\textwidth,clip,trim=2.5cm 0cm 3.5cm 0cm]{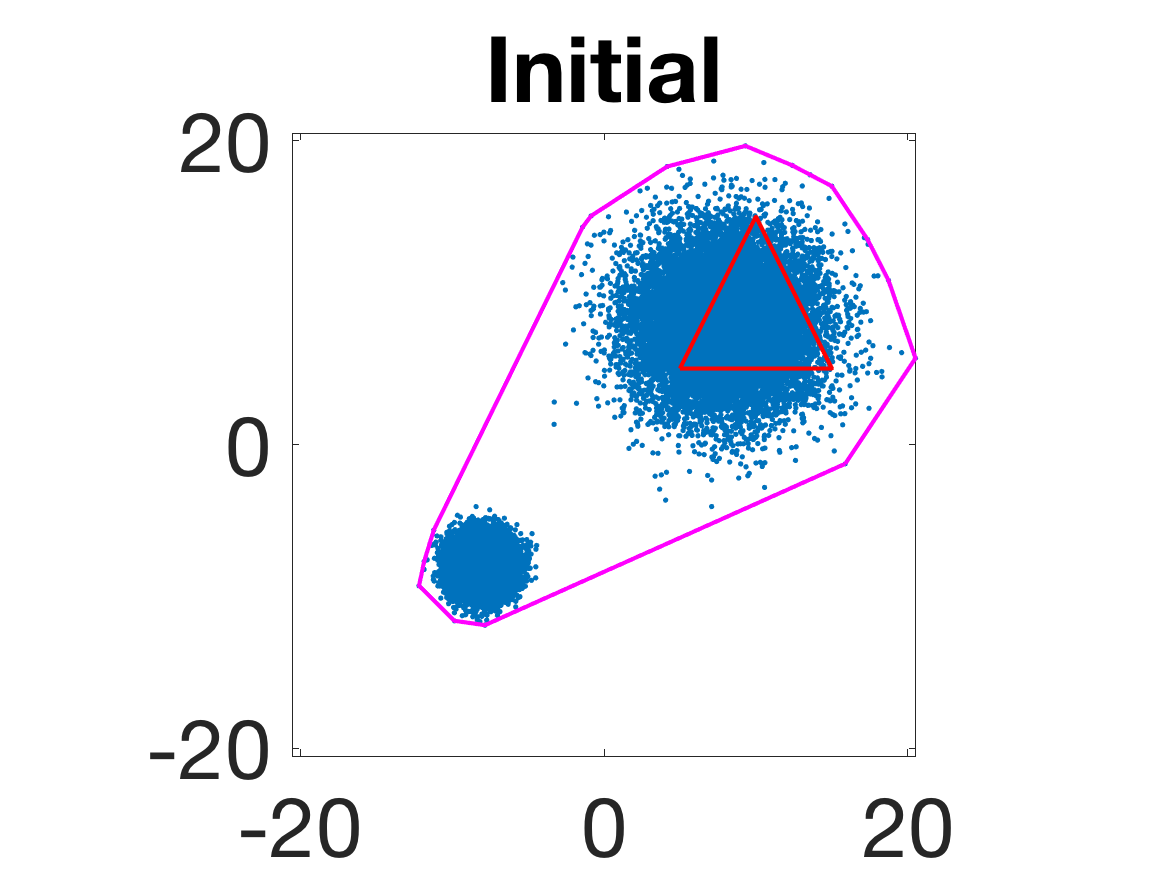}
\includegraphics[width = 0.12\textwidth,clip,trim=2.5cm 0cm 3.5cm 0cm]{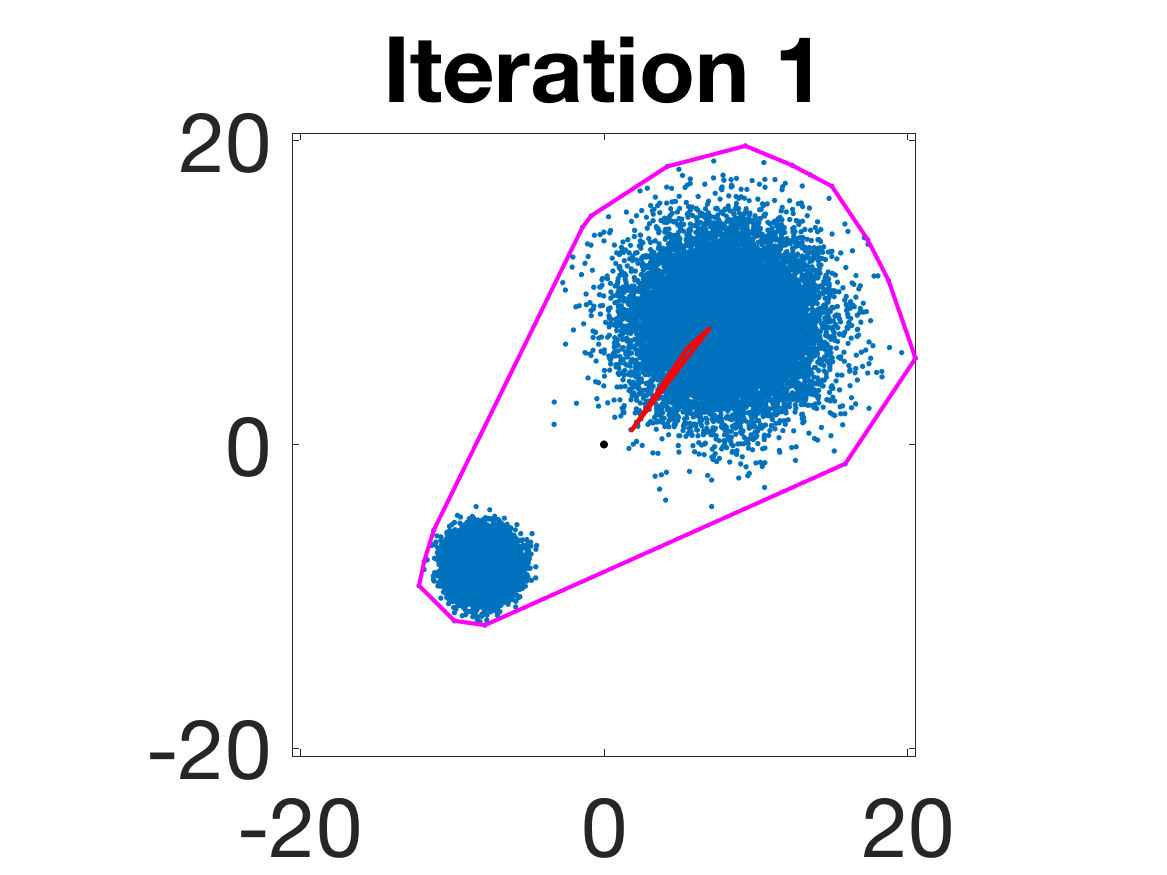}
\includegraphics[width = 0.12\textwidth,clip,trim=2.5cm 0cm 3.5cm 0cm]{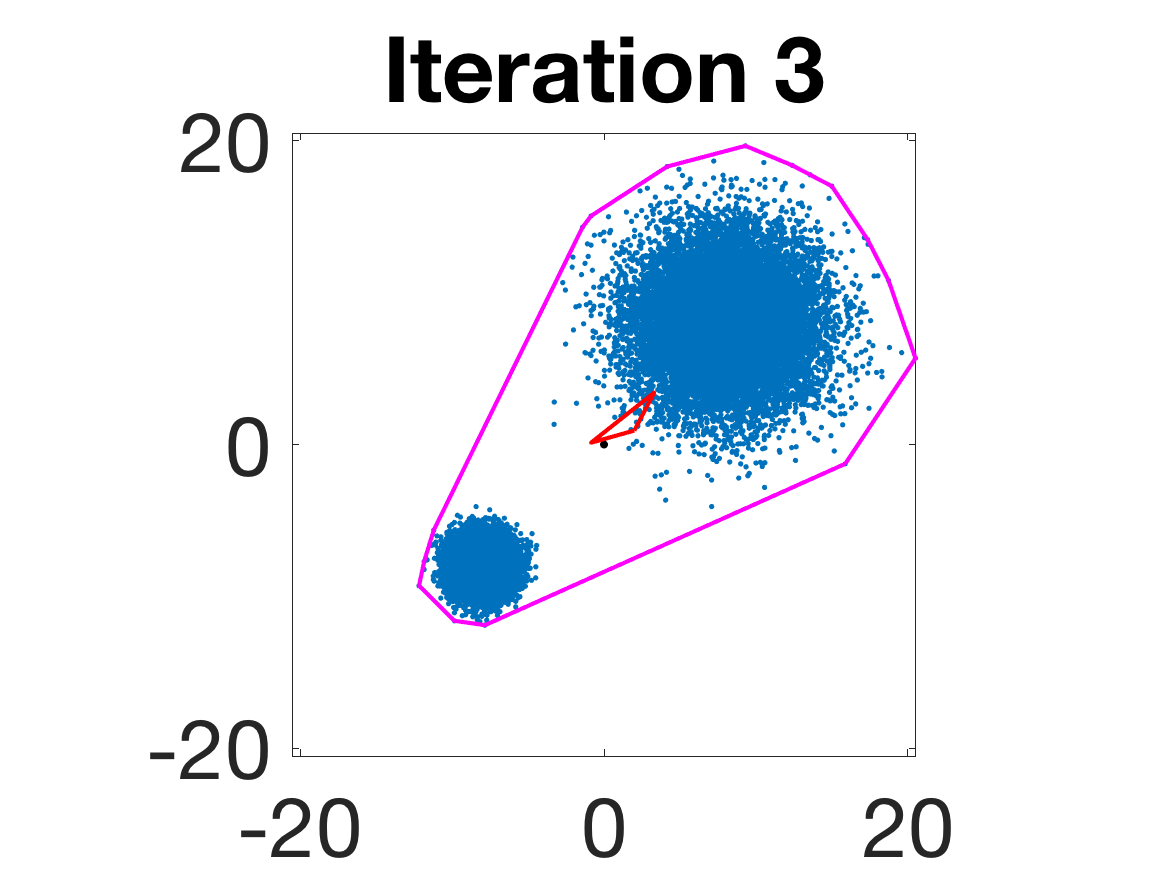}
\includegraphics[width = 0.12\textwidth,clip,trim=2.5cm 0cm 3.5cm 0cm]{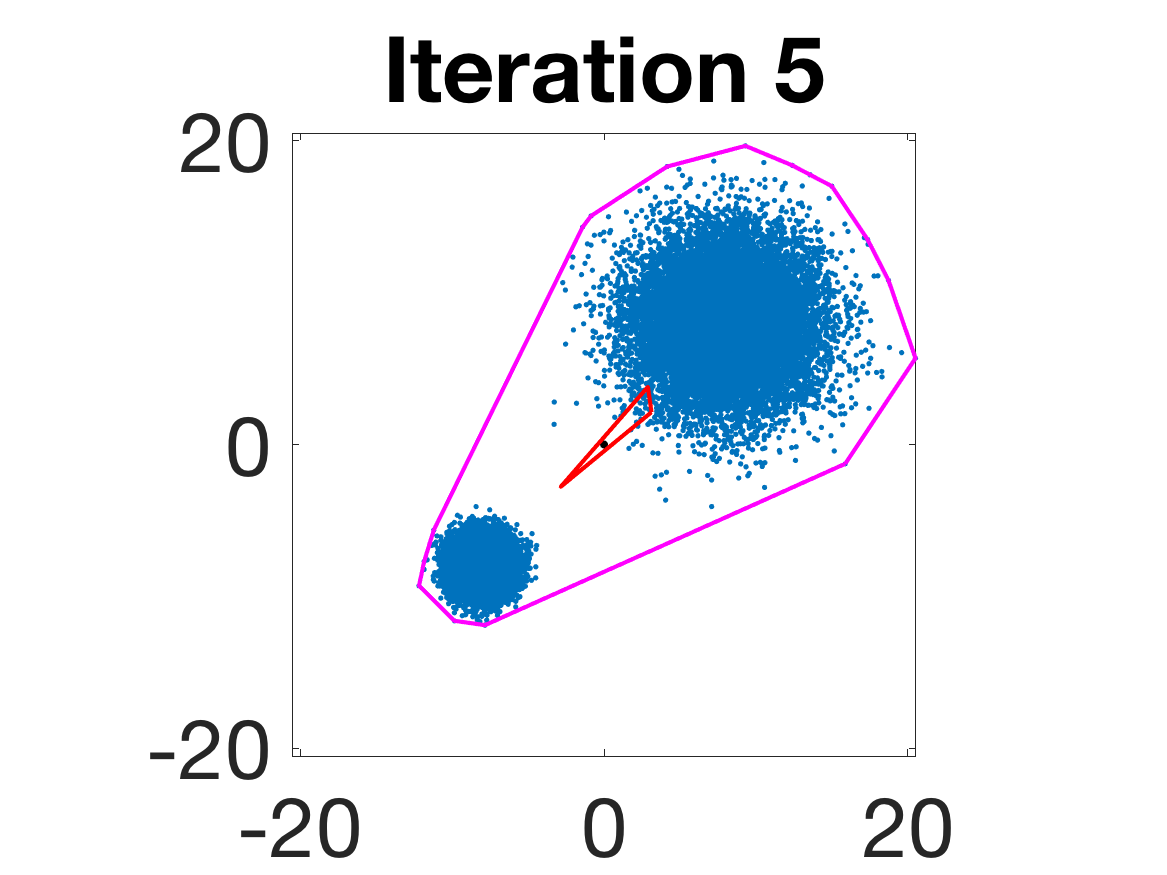}
\includegraphics[width = 0.12\textwidth,clip,trim=2.5cm 0cm 3.5cm 0cm]{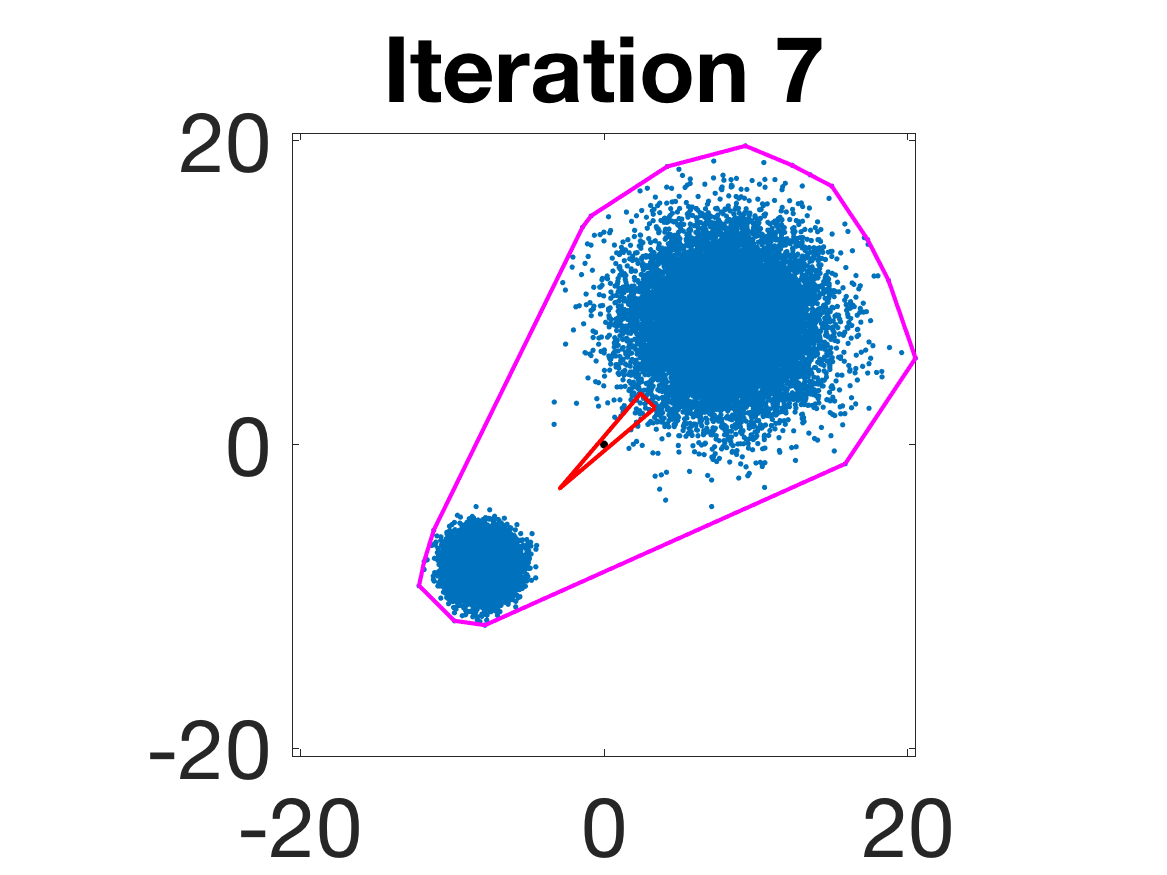}
\includegraphics[width = 0.12\textwidth,clip,trim=2.5cm 0cm 3.5cm 0cm]{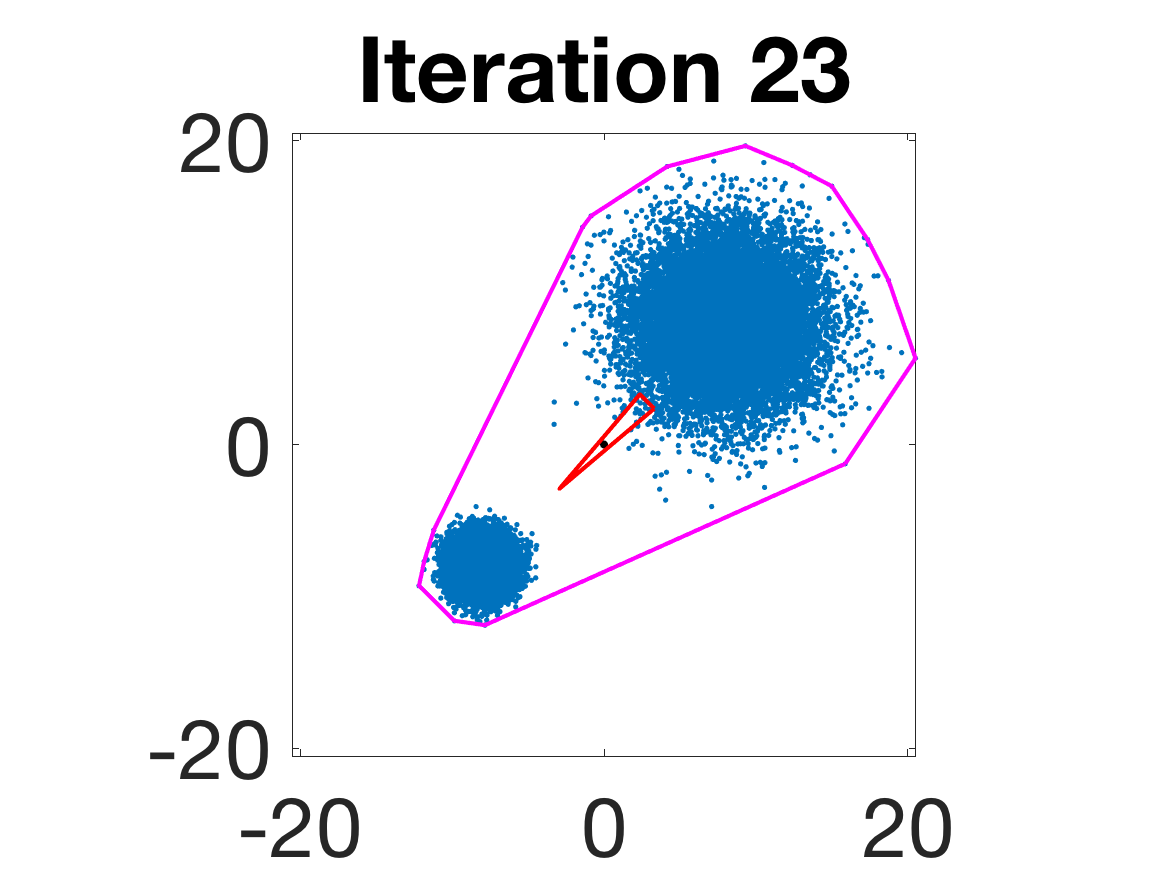}\\
\includegraphics[width = 0.12\textwidth,clip,trim=2.5cm 0cm 3.5cm 0cm]{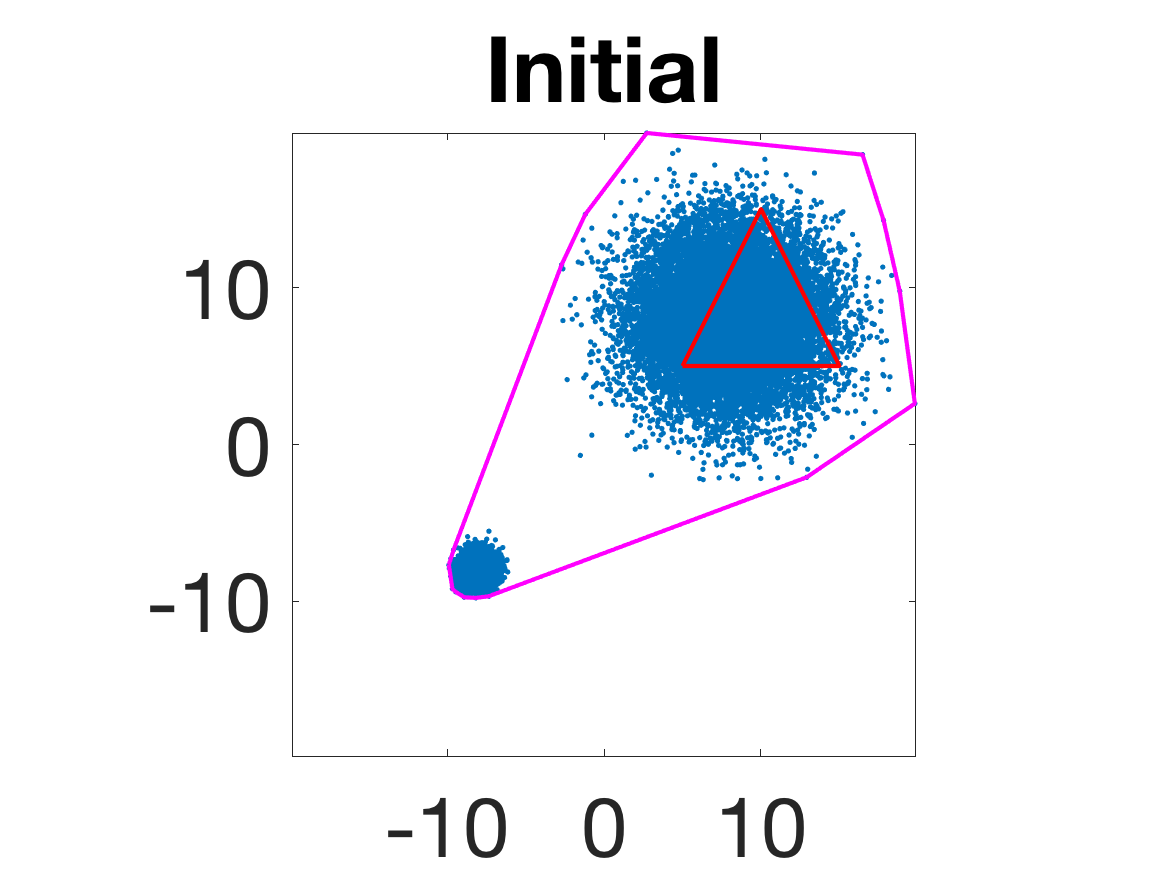}
\includegraphics[width = 0.12\textwidth,clip,trim=2.5cm 0cm 3.5cm 0cm]{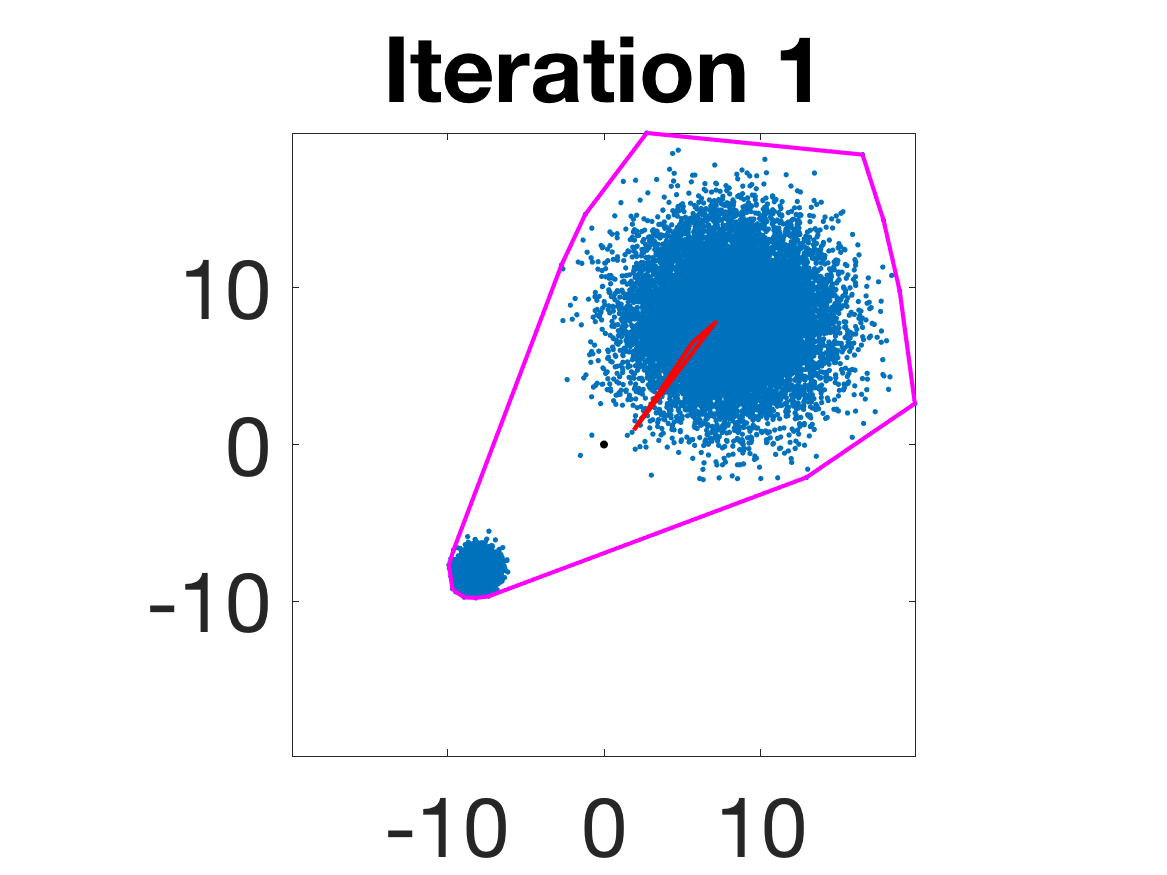}
\includegraphics[width = 0.12\textwidth,clip,trim=2.5cm 0cm 3.5cm 0cm]{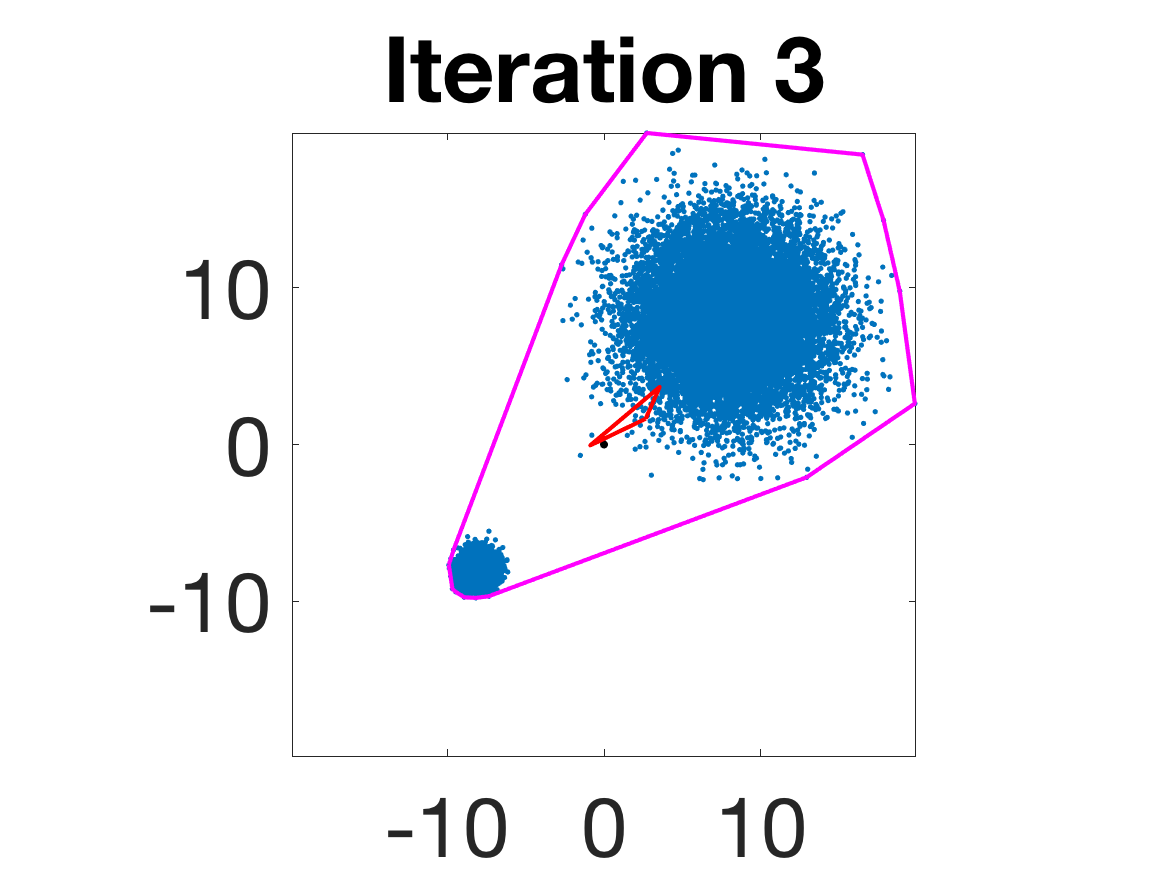}
\includegraphics[width = 0.12\textwidth,clip,trim=2.5cm 0cm 3.5cm 0cm]{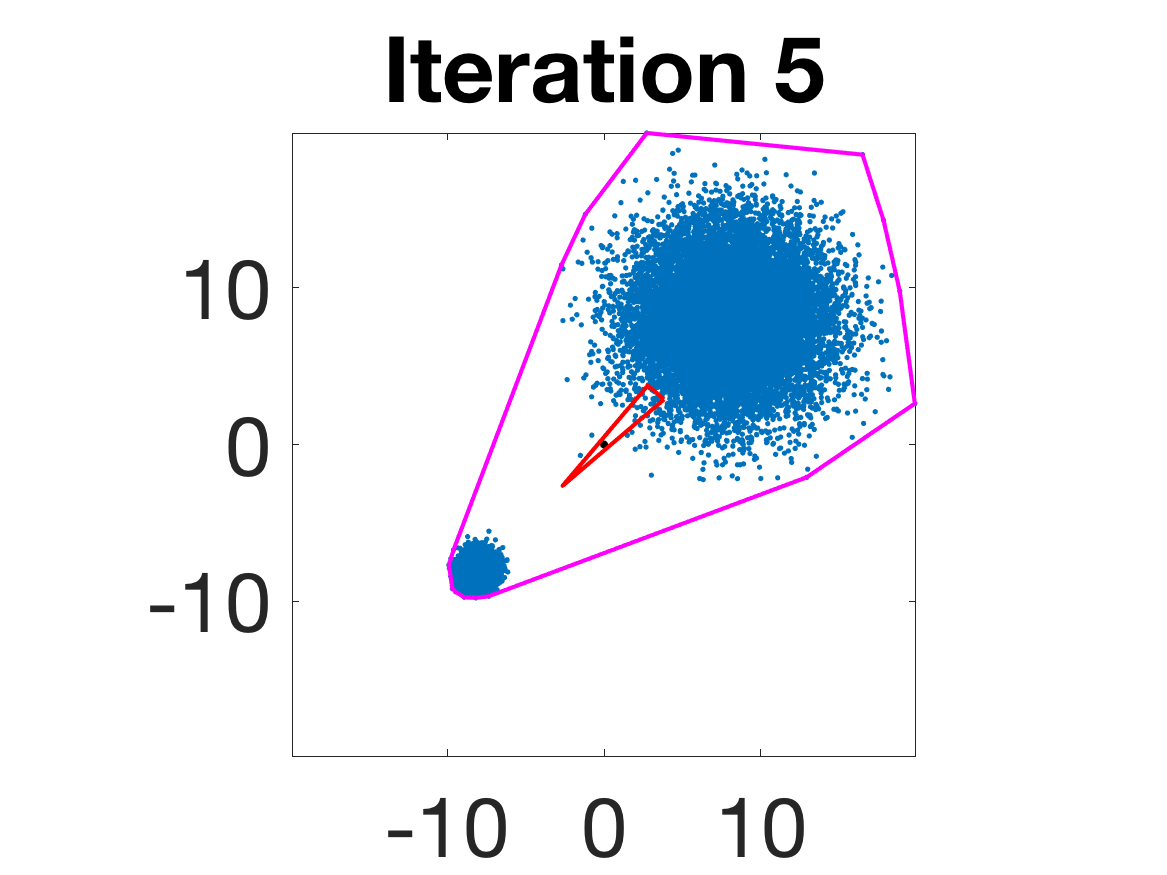}
\includegraphics[width = 0.12\textwidth,clip,trim=2.5cm 0cm 3.5cm 0cm]{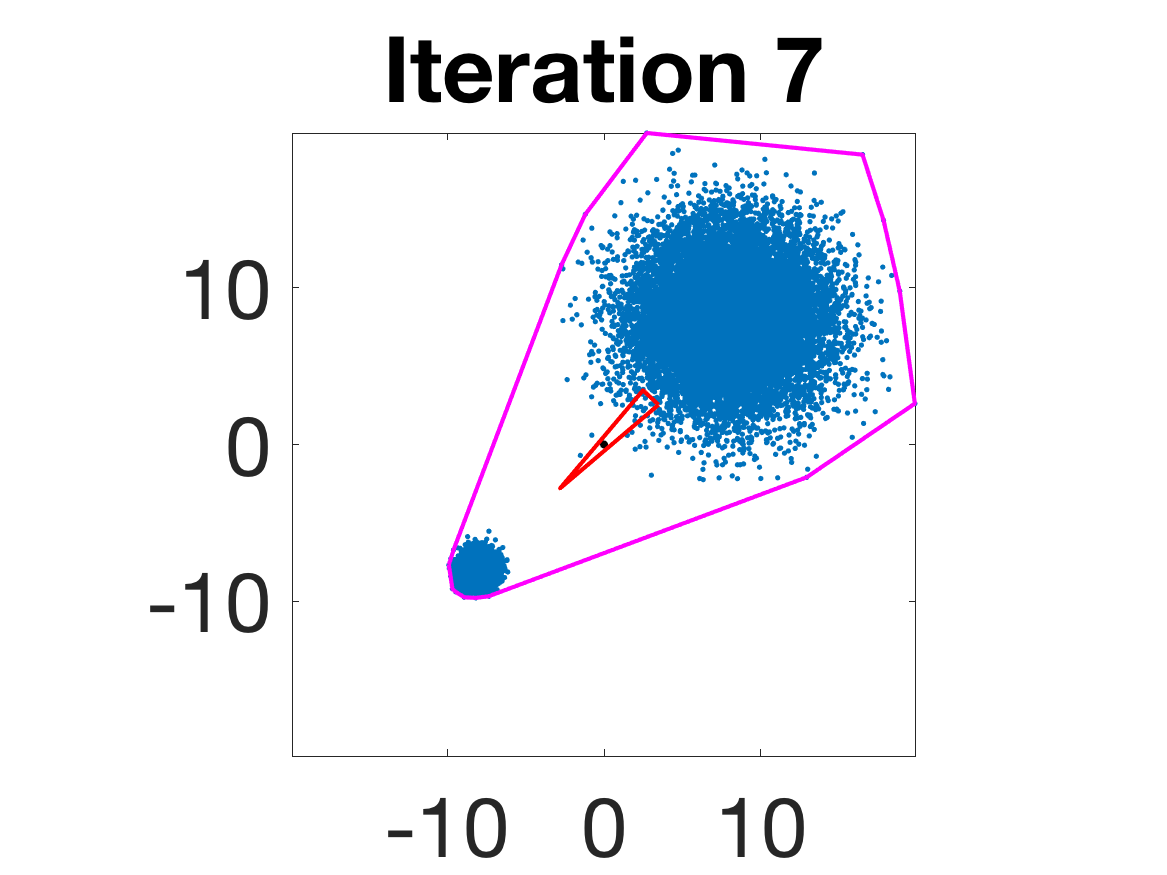}
\includegraphics[width = 0.12\textwidth,clip,trim=2.5cm 0cm 3.5cm 0cm]{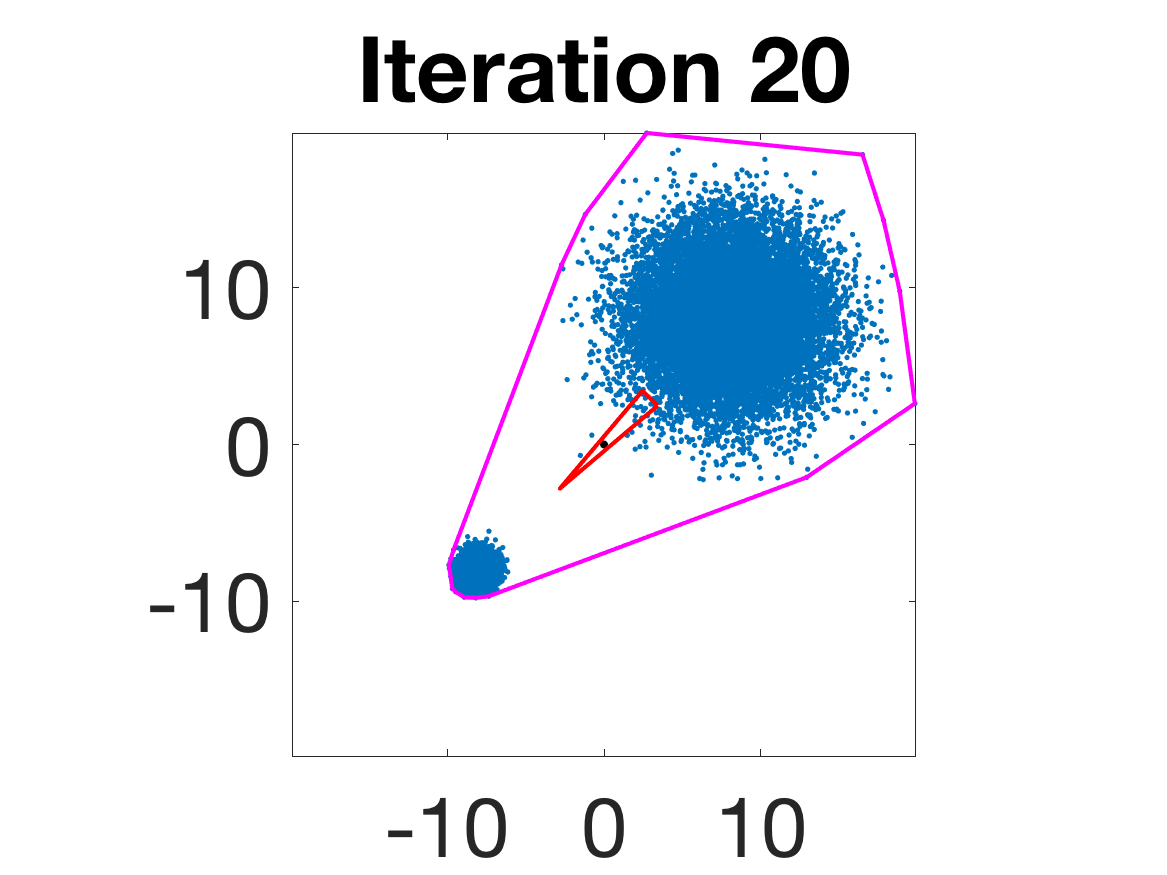}
\caption{Snapshots of the iterations for different initializations and different sampled data from a Gaussian mixture model. The last column is the final iterate. See \cref{sec:unbounded3}.} \label{fig:bimodal4}
\end{figure}

\section{Discussion} \label{s:Disc}
In this paper, we studied the archetypal analysis problem \eqref{e:arch}. 
We proved the consistency result that shows if the data are independently sampled from a probability measure  with bounded support, then the archetype points converge to the solution of the continuum version of the problem,  \eqref{e:arch2g}; see \cref{t:bdd consistency}.
If the points are independently sampled from a distribution with unbounded support, we also prove a consistency result for a modified method  \eqref{e:arch3} that penalizes the dispersion of the archetype points; see \cref{t:ConsVarReg}. 
Our analysis is supported by detailed computational experiments of the archetype points for data sampled from the uniform distribution in a disk, the normal distribution, an annular distribution, and a bimodal distribution.

To handle distributions with unbounded support, we introduced a modified method \eqref{e:arch3} that penalizes the dispersion of the archetype points. There are several other potential modified methods that we could have used. 
(i) One choice would be the boundary of the convex set of the archetype points, $\partial (\co(A))$ rather than $\co(A)$ in the definition of $F_\nu(A)$, but this is a more difficult cost to compute. 
(ii) Interpreting extremal data as outliers, we could have followed the suggestion of \cite{Cutler_1994} and used convex peeling to discard a small fraction of the data. This is equivalent to sampling from a conditional distribution. 
Finally, we could have changed the measure of distance in \eqref{e:arch2g} to, for example, use the $2$-Wasserstein metric $W_2$ and consider the problem of minimizing $W_2(\mu, \xi)$, where $ \xi = \frac{1}{\vol( \co(A) )} \chi_{\co(A)} $.
We intend to pursue this direction in subsequent work.

In this work, we used a projected gradient descent method to approximately solve \eqref{cls}; see \cref{alg1}. 
For a more scalable implementation, \cref{alg1} could be combined with sampling techniques to accelerate the convergence of the iterations. 
A subsampling approach would be to apply \cref{alg1} to a relatively small subsample of the dataset to obtain  an approximate solution and then iteratively use this as an initial guess for increasingly larger (nested) subsamples of the dataset. 
An adaptive sampling technique for distributions with compact support would involve `ignoring' samples of the dataset that lie far from $\partial \left( \co( \supp( \mu)) \right)$ and `focusing' on samples that lie near  $\partial \left( \co( \supp( \mu)) \right)$. 

\appendix

\section{Proof of the claim in \cref{m-variance}}\label{appendix2}
We prove that for every $\mu\in\mathcal V$ (which is defined in \eqref{VV}), $\supp(\mu)$ is affinely independent. 
First note that existence of $\mu$ is guaranteed by Prokhorov's theorem. 
For $a\in \co(D)$, define $\mathscr M_a=\{\nu\in\mathscr M \colon \E_\nu[x]=a\}$, so that
\begin{align*}
\mu\in\arg\max_{\nu\in\mathscr M_{\E_\mu[x]}}\E_\nu\left[ \|x-\E_\mu[x]\|_2^2 \right]=\arg\max_{\nu\in\mathscr M_{\E_\mu[x]}}\E_\nu\left[ \|x\|_2^2 \right].
\end{align*}
Note that $\mathscr M_{\E_\mu[x]}$ is a convex compact set, and $l(\nu):=\E_\nu\left[ \|x\|_2^2 \right] $ is a non-zero bounded linear functional on $\mathscr M_{\E_\mu[x]}$. The maximum of $l$ must be obtained on the extreme set of $\mathscr M_{\E_\mu[x]}$. 
A result in \cite[Theorem 2.1]{Karr_1983} implies that the extreme set of $\mathscr M_{\E_\mu[x]}$ consists of probability measures in $\mathscr M$ whose support is affinely independent, particularly, $\supp(\mu)\leq d+1$. 

\section{Derivation of problem \eqref{prob:updatez1} and \eqref{prob:updatez3}}\label{appendix}
In the follows, we include the derivation of problem \eqref{prob:updatez3}. \eqref{prob:updatez1} can be reduced from \eqref{prob:updatez3} by setting $\alpha = 0$. From \eqref{prob:updatez2}, assume we are updating the $\ell$-th column of $\mathcal Z$ (\ie, $a_\ell$) with other columns being fixed, direct calculation yields
\begin{align*}
& \frac{1}{N} \| X -  \mathcal Z  \mathcal B \|^2_F+\frac{\alpha}{k} \sum_{q=1}^k\left\|a_q - \frac{1}{k}\sum_{s=1}^k a_s \right\|_2^2 \\
=& \frac{1}{N} \sum_{j=1}^d\sum_{i=1}^N \left(x_{ji}^2 -2x_{ji}\sum_{s=1}^k \mathcal Z_{js} \mathcal B_{si} +\left( \sum_{s=1}^k \mathcal Z_{js} \mathcal B_{si} \right)^2\right) +\frac{\alpha}{k} \sum_{q=1}^k\left\|a_q- \frac{1}{k}\sum_{s=1}^k a_s\right\|_2^2 \\
=& \frac{1}{N} \sum_{j=1}^d\sum_{i=1}^N \left(-2x_{ji} \mathcal Z_{j\ell} \mathcal B_{\ell i} +\left(\mathcal Z_{j\ell} \mathcal B_{\ell i} \right)^2 + 2\mathcal Z_{j\ell} \mathcal B_{\ell i} \sum_{s\neq \ell}^k \mathcal Z_{js} \mathcal B_{s i}\right) +  \frac{\alpha(k-1)}{k^2}\|a_\ell\|_2^2 \\
&- 2\frac{\alpha}{k}\left\langle a_\ell, \frac{k-1}{k^2}\sum_{s\neq\ell}^k a_s+\frac{1}{k}\sum_{q \neq \ell}^k\left(a_q - \frac{1}{k}\sum_{s\neq\ell}^k a_s\right) \right\rangle + \mathcal O \\
=& \left(\frac{1}{N} \sum_{i=1}^N \mathcal B_{\ell i}^2+\frac{\alpha(k-1)}{k^2} \right)\|a_\ell\|_2^2 - 2\left\langle a_\ell, \frac{1}{N} \sum_{i=1}^N \left[\mathcal B_{\ell i}\left( x_i - \sum_{s\neq \ell}^k a_s\mathcal B_{s i}  \right)\right]+\frac{\alpha}{k^2} \sum_{s\neq \ell}^k a_s\right\rangle + \mathcal O \\
\end{align*}
where $\mathcal O$ contains all terms independent of $a_\ell$. Completing the squares for $a_\ell$  gives the desired formula.

\section*{Acknowledgments}
We would like to thank the anonymous referees for their very helpful comments and especially for pointing us towards \cite{Brunel_2019}, which we used in the proof of \cref{t:rate}.

\bibliographystyle{siamplain}
\bibliography{refs}

\end{document}